\documentclass[reqno]{amsart}

\usepackage[utf8]{inputenc}
\usepackage[T1]{fontenc}
\usepackage[english]{babel}
\usepackage[a4paper,body={16cm,23cm}]{geometry}
\usepackage{amsmath}
\usepackage[mathscr]{euscript}
\usepackage{amsfonts}
\usepackage{amssymb}
\usepackage{amsthm}
\usepackage{mathtools}
\usepackage{enumerate}
\usepackage{xspace}
\usepackage{cite}
\usepackage{bm}
\usepackage{comment}

\usepackage[colorlinks,linkcolor=blue,citecolor=blue]{hyperref}

\theoremstyle{plain}
\newtheorem{thm}{Theorem}[section]
\newtheorem*{result*}{Main result}
\newtheorem{lemma}[thm]{Lemma}
\newtheorem{proposition}[thm]{Proposition}

\theoremstyle{definition}
\newtheorem{definition}[thm]{Definition}

\theoremstyle{remark}
\newtheorem{rem}[thm]{Remark}

\numberwithin{equation}{section}
\numberwithin{footnote}{section}

\newcommand {\R} {\mathbb R}

\newcommand{\Eqref}[1]{Eq.~\eqref{#1}}

\newcommand{\Sectionref}[1]{Section~\ref{#1}}  
\newcommand{\Defref}[1]{Definition~\ref{#1}}
\newcommand{\Lemref}[1]{Lemma~\ref{#1}}
\newcommand{\Propref}[1]{Proposition~\ref{#1}}
\newcommand{\Theoremref}[1]{Theorem~\ref{#1}}

\newcommand{\pbold}{{\boldsymbol p}}
\newcommand{\lambdabold}{{\boldsymbol\lambda}}

\newcommand{\matchtime}{{t_*}}

\input Tdef.def

\author[F. Beyer]{Florian Beyer}
\address{Dept of Mathematics and Statistics\\
730 Cumberland St\\
University of Otago, Dunedin 9016\\ New Zealand}
\email{fbeyer@maths.otago.ac.nz }

\author[T.A. Oliynyk]{Todd A. Oliynyk}
\address{School of Mathematical Sciences\\
9 Rainforest Walk\\
Monash University, VIC 3800\\ Australia}
\email{todd.oliynyk@monash.edu}

\title{Relativistic 
perfect fluids near Kasner singularities}

\begin{document}

\begin{abstract}
  \noindent 
  We establish the existence of a stable family of solutions to the Euler equations on Kasner backgrounds near the singularity with the full expected asymptotic data degrees of freedom and no symmetry or isotropy restrictions. Existence is achieved through transforming the Euler equations into the form of a symmetric hyperbolic Fuchsian system followed by an application of a new existence theory for the singular initial value problem. Stability is shown to follow from the existence theory for the (regular) global initial value problem for Fuchsian systems that was developed in \cite{beyer2019a}. In fact, for each solution in the family, we prove the existence of an open set of nearby solutions with the same qualitative asymptotics and show that any such perturbed solution agrees again with another solution of the singular initial value problem. All our results hold in the regime where the speed of sound of the fluid is large in comparison to all Kasner exponents. This  is interpreted as the regime of \emph{stable} fluid asymptotics near Kasner big bang singularities. 
\end{abstract}

\maketitle

\section{Introduction}
\label{sec:intro}

\subsubsection*{Perfect fluids on Kasner backgrounds} The main result of this paper concerns the dynamics of solutions of the compressible relativistic Euler equations in the vicinity of the in general highly anisotropic singularity of Kasner spacetimes.
Relativistic compressible perfect fluids with linear equations of state\footnote{In this paper, we choose physical units such that $c=1$ for the speed of light and $G=1/(8\pi)$ for Newton's gravity constant.} 
\begin{equation}
  \label{eq:EulerEOS2}
  P=c_s^2\rho,
\end{equation}
where $P$ is the fluid {pressure}, $\rho$ is the fluid {density} and
$c_s$ the \textit{speed of sound},
are common matter models in cosmology. In the standard model of cosmology \cite{mukhanov2005}, large portions of the history of the universe are taken to be dominated by perfect fluids with $c_s=0$ (\emph{dust}) while the early hot universe, which we are mostly concerned with in this paper here, is dominated by \emph{radiation fluids} given by $c_s^2=1/3$.

One of the main goals in mathematical cosmology is to characterise the big bang asymptotics of general solutions to the coupled Einstein-matter equations, a formidable task. However, the \emph{matter does not matter} hypothesis, part of the famous \emph{BKL conjecture} \cite{barrow1978,belinski1973,belinskii1970,eardley1972,lifshitz1963,wainwright1997},
suggests that the study of the matter equations defined on some relevant class of \emph{fixed} singular cosmological model spacetimes is a useful and significant first step for this long-term  endeavour. This hypothesis states that the big bang asymptotics of the {geometry} for solutions of the Einstein-matter equations is  governed by the  Einstein-vacuum equations, and therefore, the evolution of the geometry effectively decouples from the dynamics of the matter. We take this as a justification to restrict our attention to the perfect fluid equations on \emph{fixed} background spacetimes in this paper as a first step towards the long-term objective of understanding the dynamics of the coupled Einstein-Euler equations near a big-bang singularity.
It is believed that the \emph{matter does not matter} hypothesis is valid only if the matter is not ``too extreme''. In the case of \eqref{eq:EulerEOS2} it is expected that the condition $c_s^2<1$ is sufficient while the dynamics of the ``extreme'' case $c_s^2=1$, the stiff fluid (or scalar field) case, is expected to be significantly different \cite{ames2019,andersson2001,rodnianski2014,rodnianski2018,barrow1978}.  Interestingly, our results do continue to apply to fluids on Kasner-scalar field background spacetimes as we explain below.

The family of \emph{Kasner spacetimes} holds a prominent place in the mathematical cosmology literature as one of the most important singular cosmological models. This is particularly true for the so-called \emph{Kasner-scalar field spacetimes} -- solutions of the Einstein-scalar field equations -- which are expected to be limit points of the coupled Einstein-matter equations and, most importantly, to which the results of this article can be easily applied, see Remark~\ref{rem:KSF}. In the presentation here, we, however, focus on the \emph{Kasner-vacuum spacetimes}, often referred to as simply \emph{Kasner spacetimes}.
A {Kasner-(vacuum) spacetime} is a spatially homogeneous solution $(M,g)$ of Einstein's vacuum equations
for $M=(0,\infty)\times\Sigma$ with $\Sigma=\mathbb T^3$ and
\begin{equation}
\label{Kasner-k-original}
g = -d\tilde t\otimes d\tilde t + \tilde t^{2p_1} d\tilde x\otimes
d\tilde x  + \tilde t^{2p_2} d\tilde y\otimes d\tilde y  + \tilde t^{2p_3} d\tilde z\otimes d \tilde z,
\end{equation}
for \emph{Kasner exponents} $p_1$, $p_2$ and $p_3$ in $\mathbb R$ satisfying the Kasner relations 
\begin{equation}
  \label{eq:Kasnerrel}
  \sum_{i=1}^3 p_i=1,\quad \sum_{i=1}^3 p^2_i=1.
\end{equation}
Except for the \emph{flat Kasner cases} $(p_1,p_2,p_3)=(1,0,0)$, $(p_1,p_2,p_3)=(0,1,0)$ or $(p_1,p_2,p_3)=(0,0,1)$, all Kasner
metrics $g$ have curvature singularities in the limit $\tilde t\searrow 0$. 

In this paper, we consider the compressible relativistic Euler equations defined on Kasner spacetimes in the framework of
\cite{frauendiener2003,walton2005} according to which the perfect fluid 
is represented by a single timelike vector field $V^\alpha$ and the Euler equations take the form
\begin{equation}
\label{eq:AAA1}
{A^\delta}_{\alpha\beta}\nabla_\delta V^\beta=0,\quad
A^\delta_{\alpha\beta}
 =-\frac{3 c_s^2+1}{c_s^2} \frac{V_\alpha V_\beta}{V_\lambda V^\lambda} V^\delta+V^\delta g_{\alpha\beta}
  +2{g^\delta}_{(\beta} V_{\alpha)},
\end{equation}
provided $c_s^2>0$. Despite its relevance for cosmology, the case $c_s^2=0$ is excluded from
our results, note the inequality \eqref{eq:stablefluidrestr} from Theorem \ref{thm:informal} below. It is worth noting that this is not due to our use of the formulation \eqref{eq:AAA1} of the Euler equations, which is undefined at $c_s=0$.

In writing \eqref{eq:AAA1}, we have made use of the Einstein summation convention where indices are lowered and
raised with the given Kasner spacetime metric $g$. The derivative $\nabla$ in \eqref{eq:AAA1} is the covariant derivative associated with $g$.
The key property which characterises this form of the Euler equations is that \eqref{eq:AAA1} is explicitly symmetric hyperbolic, that is, the coefficients $A^\delta_{\alpha\beta}$ in \eqref{eq:AAA1} are  symmetric in $\alpha$ and $\beta$ and $A^0_{\alpha\beta}$ is positive-definite (if the fluid is timelike), which turns out to be crucial for the analysis. According to \cite{frauendiener2003,walton2005},
the 
fluid pressure $P$,
density $\rho$ and normalised fluid $4$-vector field $u^\alpha$ can be calculated from $V^\alpha$ via the expressions\footnote{In consistency with \eqref{Kasner-k-original}, the signature of Lorentzian metrics is taken to be
$(-,+,+,+)$ in this paper. The fluid vector field $V^\alpha$ is therefore timelike if and only if $V_\lambda V^\lambda<0$.}
\begin{equation}
  \label{eq:physicsquantitiesfluid}
P=\rho_0 c_s^2 (-V_\lambda V^\lambda)^{-\frac {c_s^2+1}{2c_s^2}},\quad
\rho=\rho_0 (-V_\lambda V^\lambda)^{-\frac {c_s^2+1}{2c_s^2}},\quad u^\alpha=V^\alpha/\sqrt{-V_\lambda V^\lambda}.
\end{equation}
for an arbitrary constant $\rho_0$ carrying the correct physical units of an energy density.

We are now in the position to state an informal version of the main result of this paper. The precise version, which is a combination of two results, is given in Theorems \ref{thm:fluidresult} and \ref{thm:fluidresult2}, and the proof of these theorems can be found in \Sectionref{sec:applications}.  We emphasise in this article we make no symmetry assumptions for the fluid, nor impose isotropy restrictions on the background spacetime. Furthermore, our results continue to hold for the more general family of \emph{Kasner-scalar field background spacetimes}; see Remark~\ref{rem:KSF} below and \cite{ames2019, rodnianski2014,rodnianski2018}. These are spacetimes of the same form as
\eqref{Kasner-k-original} that solve the coupled Einstein-scalar field equations as opposed to the Einstein-vacuum equations. For these spacetimes, the Kasner relations \eqref{eq:Kasnerrel} take a more general form, see \eqref{eq:Kasnerrelsf}, that involves the scalar field strength parameter $A$, and special cases include the isotropic Friedmann-Robertson-Walker spacetimes for $A=\pm\sqrt{2/3}$ as well as the vacuum spacetimes for $A=0$. It was a major breakthrough of  
\cite{rodnianski2014,rodnianski2018,fournodavlos2020b} to establish that Kasner-scalar field spacetimes (in the sub-critical regime) are stable under generic nonlinear perturbations of the Einstein-scalar field equations. The fact that our results here extend to this family of Kasner-scalar field background spacetimes will therefore be crucial in any extension of our line of research here to include coupling to Einstein gravity. However, for the purposes of this paper, there is no loss of generality, as we
argue in  Remark~\ref{rem:KSF}, to restricting our attention to the special case of vacuum Kasner background spacetimes.

\begin{thm}[Informal version] \label{thm:informal}
  Suppose $p_1, p_2, p_3$ are Kasner exponents satisfying \begin{equation}
  \label{eq:stablefluidrestr}
  1>c_s^2>\max\{p_1,p_2,p_3\}
\end{equation} 
and $W^0_*>0$, $W^1_*$, $W^2_*$ and $W^3_*$ are given functions on $\Tbb^3$ at $\tilde t=0$ that together define the asymptotic data for the Euler equations.
Then near the singularity at $\tilde t=0$ of the Kasner vacuum solution  with exponents $p_1$, $p_2$ and $p_3$ given by \eqref{Kasner-k-original} -- \eqref{eq:Kasnerrel}, there exists a  
solution $V^\alpha$ of the Euler equations \eqref{eq:AAA1}
with equation of state \eqref{eq:EulerEOS2} that can be represented as 
\begin{equation}
  \label{eq:fluidorthonormal}
  V^\alpha=\tilde t^{c_s^2}\left(W^0 e_0^\alpha+\tilde t^{c_s^2-p_1} W^1 e_1^\alpha+\tilde t^{c_s^2-p_2} W^2 e_2^\alpha+\tilde t^{c_s^2-p_3} W^3 e_3^\alpha\right)
\end{equation}
for some uniformly bounded functions $W^0$, $W^1$, $W^2$, and $W^3$ on $(0,T_0]\times \Tbb^3$ that agree in the limit $\tilde t\searrow 0$ with the asymptotic
data $W^0_*>0$, $W^1_*$, $W^2_*$, and $W^3_*$, respectively, and $V^\alpha$ is everywhere timelike including in the limit $\tilde t\searrow 0$.
Here $(e_0^\alpha, e_1^\alpha, e_2^\alpha, e_3^\alpha)$ is the orthonormal frame of the Kasner spacetime given by 
\begin{equation}
  \label{eq:KasnerONF}
  e_0=\partial_{\tilde t},\quad e_1=\tilde t^{-p_1}\partial_{\tilde x},\quad e_2=\tilde t^{-p_2}\partial_{\tilde y},\quad e_3=\tilde t^{-p_3}\partial_{\tilde z}.
\end{equation}
Moreover this solution is  unique within the class of solutions with the asymptotics \eqref{eq:fluidorthonormal}, and it is non-linearly stable against sufficiently small perturbations and the behaviour of the resulting perturbed solution again has the asymptotics \eqref{eq:fluidorthonormal}.
\end{thm}

The main physical restriction in Theorem \ref{thm:informal} is \eqref{eq:stablefluidrestr}, which we refer to as the \emph{stability condition}. This condition determines an additional restriction on the fluid sound speed $c_s$ and the Kasner exponents $p_1$, $p_2$ and $p_3$ beyond the non-negativity of $c_s$ and the Kasner relation \eqref{eq:Kasnerrel}.
According to this condition,  the square of the speed of sound must be smaller than $1$; that is, the fluid is \textit{not} stiff\footnote{The case of (irrotational) stiff fluids is handled in \cite{rodnianski2014,rodnianski2018,fournodavlos2020b}}, and greater than all Kasner exponents. Before we explain the meaning of the lower bound on the sound speed, let us assume for the moment that \eqref{eq:stablefluidrestr} holds. For this case, Theorem \ref{thm:informal} guarantees the existence of perfect fluid solutions described by \eqref{eq:fluidorthonormal} near the big bang singularity at $\tilde t=0$ represented in terms of the orthonormal frame \eqref{eq:KasnerONF} with respect to a (by assumption fixed) background Kasner spacetime. Physically, this orthonormal frame can be interpreted as a \emph{cosmological reference frame}. The functions $W^0$, $W^1$, $W^2$ and $W^3$ in \eqref{eq:fluidorthonormal} distinguish the different fluid solutions and are all convergent at $\tilde t=0$.
The main  fluid dynamics near $\tilde t=0$ is therefore described by the powers of $\tilde t$ in \eqref{eq:fluidorthonormal}, most notably the powers $\tilde t^{c_s^2-p_i}$ in front of the spatial fluid components. The fact that these powers are all positive and therefore bounded at $\tilde t=0$ by virtue of \eqref{eq:stablefluidrestr} is a first indication of the importance of the stability condition. As a consequence of  \eqref{eq:physicsquantitiesfluid} and \eqref{eq:fluidorthonormal},  we note that
\begin{equation}
  \label{eq:pressureblowup}
  P\sim \tilde t^{-(1+c_s^2)}, \quad \rho\sim \tilde
  t^{-(1+c_s^2)},
\end{equation}
and
\begin{equation}
  \label{eq:fluidorthonormalnormal}
  u^\alpha=\frac{e_0^\alpha+Q^1 e_1^\alpha+Q^2 e_2^\alpha+Q^3 e_3^\alpha}{\sqrt{1-(Q^1)^2-(Q^2)^2-(Q^3)^2}}
\end{equation}
where
\begin{equation}
  \label{eq:defQ}
  Q^i:={\tilde t}^{c_s^2-p_i}\,\frac{W^i}{W^0},\quad i=1,2,3.
\end{equation}
It follows from \eqref{eq:fluidorthonormalnormal} that the quantities $Q^i$  can be interpreted as a measure for the \emph{peculiar motions} of the fluids relative to the cosmological reference frame. Provided the stability condition holds, we deduce from Theorem~\ref{thm:informal} that $Q^i\rightarrow 0$ as $\tilde t \searrow 0$, and it follows that all the fluids are \emph{asymptotically co-moving}; i.e., the peculiar motions die out at $\tilde t=0$. This is consistent with one of the key ideas of the \emph{standard model of cosmology}, namely that  peculiar motions of matter should be small.

It is clear that the stability condition imposes a significant restriction on the fluid sound speed. For instance, it follows from \eqref{eq:Kasnerrel} that
\[\max\{p_1,p_2,p_3\}\ge 2/3,\]
and hence the lower bound for $c_s^2$ in \eqref{eq:stablefluidrestr}, notably, rules out  {radiation fluids}  $c_s^2=1/3$ -- one of the most popular matter models for the early universe\footnote{According to Remark~\ref{rem:KSF}, it is interesting to notice that in the special case of the FLRW Kasner-scalar field background, the radiation fluid case  $c_s^2=1/3$ is \emph{exactly} at the border of the stability region}. It is therefore natural to question what happens to the fluids when \eqref{eq:stablefluidrestr} is violated. In general this is an open question. But the following special case provides strong evidence that the fluid behaviour is significantly different when \eqref{eq:stablefluidrestr} is violated (without loss of generality we  assume that $p_3\ge p_2\ge p_1$ in the following). For this special case, we assume that $W^1=W^2=0$ and that $W^0$ and $W^3$ are functions of $\tilde t$ only; in particular, this means that the fluid is spatially homogeneous and only flows in the direction of the largest Kasner exponent. It is then possible to integrate the Euler equations to obtain solutions in the following implicit form\footnote{This implicit solution is derived without imposing the second Kasner constraint in \eqref{eq:Kasnerrel}. In particular, the same result is obtained when the background is a Kasner-scalar field spacetime with \eqref{eq:Kasnerrelsf}.}
\begin{equation}
  \label{eq:explicitEulersol}
  \frac{\bigl|Q^3(\tilde t)\bigr|}{\bigl|1-(Q^3(\tilde t))^2\bigr|^{(1-c_s^2)/2}}=C {\tilde t}^{c_s^2-p_3},\quad \tilde t>0,
\end{equation}
for a constant $C\ge 0$ where $Q^3$ is as defined above by \eqref{eq:defQ}. Observe here how the factor ${\tilde t}^{c_s^2-p_3}$ whose behaviour we control by means of the stability condition appears, and that the fluid is co-moving; i.e., $Q^3\equiv 0$ if and only if $C=0$. Let us suppose that $C\not=0$ so that the fluid has some peculiar motion. If the stability condition holds, i.e., if $p_3<c_s^2<1$, the right hand side converges to zero at $\tilde t=0$ which means by virtue of \eqref{eq:explicitEulersol} that $Q^3\rightarrow 0$ in consistency with our result before. However, if the stability condition is violated and $c_s^2<p_3<1$, then the right-hand side diverges in this limit (since $C\neq0$) from which it follows that $1-(Q^3)^2\rightarrow 0$. The peculiar motions of the fluids are therefore unstable and the fluid becomes ``asymptotically null'' (as suggested by the normalisation factor in \eqref{eq:fluidorthonormalnormal}). This behaviour is not consistent with the standard model of cosmology and the before-mentioned matter does not matter hypothesis. The borderline case $c_s^2=p_3$ is characterised by $Q^3=const$ in this special case. We remark that the critical dependence of the fluid dynamics on the signs of $c_s^2-p_i$ was first observed in \cite{beyer2017,beyer2019a,beyer2020b} while  \cite{hervik2004} provides more general analyses of \emph{tilted fluids} in spatially homogeneous settings.

Even though more work is clearly needed to fully support our claims beyond \Theoremref{thm:informal}, in particular, coupling to Einstein gravity needs to be taken into account, we nevertheless assert that \eqref{eq:stablefluidrestr} is \emph{both sufficient and necessary} to obtain a stable fluid flow at the big bang singularity $\tilde t=0$ with the potential for interesting critical dynamics when $c_s^2=\max\{p_1,p_2,p_3\}$. If this turns out to be correct, then this would have important implications for physics as it might rule out some
of the standard matter models used in cosmology to describe the early universe, and in particular, the radiation fluid $c_s^2=1/3$.

\subsubsection*{Forward and backward Fuchsian methods}

The main results of this paper, in particular, Theorems~\ref{thm:informal}, \ref{thm:fluidresult} and \ref{thm:fluidresult2}, are established through the use of complementary forward and backward Fuchsian techniques. The
\emph{forward Fuchsian method} is a well-established technique that has been used extensively to establish the existence of solutions to singular initial value problems for Fuchsian systems; for example, see \cite{ames2013a,ames2013b,ames2017,andersson2001,beyer2010b,beyer2011,beyer2012a,beyer2017,choquet-bruhat2004,claudel1998,damour2002,fournodavlos2020,heinzle2012,isenberg1999,isenberg2002,kichenassamy1998,klinger2015,petersen2018,rendall2000,rendall2004,rendall2004a,stahl2002}. Its purpose here is to prove the existence of solutions $V^\alpha$ of the form \eqref{eq:fluidorthonormal} with functions $W^0$, $W^1$, $W^2$, $W^3$ that converge at $\tilde t=0$ to specific chosen limits $W^0_*$, $W^1_*$, $W^2_*$, $W^3_*$ -- the asymptotic data. This yields a family of fluid solutions parametrised by four functional degrees of freedom. The precise result is given in Theorem~\ref{thm:fluidresult}. This method is based on \emph{forward} evolutions away from the singular time $\tilde t=0$ launching from $\tilde t=0$.
The \emph{backward Fuchsian method} on the other hand is a relatively new technique, first introduced in \cite{oliynyk2016} and further developed in \cite{beyer2019a,fajman2020a}, for establishing the global existence of solutions to initial value problems for Fuchsian systems of equations; see \Sectionref{sec:FuchsianCauchyProblem} for a brief summary of this method. In this work here, the purpose of the backward method is to construct perturbations of the solutions obtained by the forward method globally backwards in time towards the singular time $\tilde t=0$ and to investigate the limits at $\tilde t=0$. In fact we show that small perturbations of any solution of the forward problem given by asymptotic data $W^0_*$, $W^1_*$, $W^2_*$, $W^3_*$ converges to limits $\tilde W^0_*$, $\Wt^1_*$, $\Wt^2_*$, $\Wt^3_*$ close to the asymptotic data. Our results establish that the forward and backward solutions of the Euler equations are, in a sense, in one-to-one correspondence -- a question pioneered by Ringstr\"om \cite{ringstrom2017,ringstrom2021,ringstrom2021a}. However, the important problem of characterising this correspondence as an invertible continuous map with respect to a given topology remains open. We plan on returning to this question in future work.

In order to illustrate how these techniques are employed in this article, it suffices to consider the following model equation 
\begin{equation}
  \label{eq:6model6}
  \partial_t u(t)=\frac 1t b u(t)+F(t),\quad t>0,
\end{equation}
where here, $b\in \Rbb$ is a constant and $F$ is some sufficiently smooth function on $t>0$. This ordinary differential equation is the prototype of a \emph{Fuchsian equation}; a class of in general nonlinear partial differential equations introduced in \Sectionref{sec:symhypFuchs}. For this simple problem, the solution is clearly
\begin{equation}
  \label{eq:6model7}
  u(t)=t^b\Bigl(t_*^{-b} u_*+\int_{t_*}^t s^{-b} F(s) ds\Bigl),\quad t>0,
\end{equation}
where $t_*>0$ and $u_*\in\R$ are arbitrary.  Since $u(t_*)=u_*$, we can interpret \eqref{eq:6model7} as the solution of the initial value problem for \eqref{eq:6model6} with initial data $u_*$ imposed at the regular time $t=t_*>0$.  We refer to this initial value problem as the \emph{backward problem} with the idea that the initial data is imposed at a regular time $t_*>0$ and the Fuchsian equation is solved \emph{backwards in time} all the way down to the singular time  at $t=0$. By a \emph{backward method}, we then mean a method to construct solutions to the backward problem and analyse their asymptotics as $t\searrow 0$.

In contrast, the \emph{forward problem}, which we also often refer to as the \emph{singular initial value problem}, involves solving the Fuchsian equation forward in time  
starting from prescribed \emph{asymptotic data} at $t=0$. A \emph{forward method} then refers to a method to construct solutions of the forward problem.
In the case of \eqref{eq:6model6}, assuming that $F(t)=\Ord(t^{b-1+\epsilon})$ for some $\epsilon>0$, this corresponds to expressing the solution given by \eqref{eq:6model7} as
\begin{equation}
  \label{eq:6model8} 
  u(t)=t^b\Bigl(\ut_*+\int_{0}^t s^{-b} F(s) ds\Bigl), \quad t>0,
\end{equation}
where here, the constant $\ut_*$ defines the \textit{asymptotic data}.
The \emph{leading order term}
$u_0(t)=t^{b}\ut_*$ of the solution \eqref{eq:6model8} is therefore parametrised by the asymptotic data while the \emph{remainder} $w(t)= t^b\int_{0}^t s^{-b} F(s) ds$ is of higher order at $t=0$, i.e., $w(t)=O(t^{b+\epsilon})$.

In the case of the model equation \eqref{eq:6model6}, it is clear that both the backward and the forward problems are equivalent provided that $F(t)=\Ord(t^{b-1+\epsilon})$, and consequently, they contain the same information about the class of all solutions. While it has been established that this equivalence between the forward and backward problems also holds for  certain Fuchsian systems of partial differential equations \cite{isenberg1990,ringstrom2017, ames2019}, it
is definitely not the case that these two complementary problems always contain the same information \cite{ringstrom2008, ringstrom2017,fajman2020a}. In the case of perfect fluids on Kasner spacetimes, however, they do as we prove in this article, and this leads to the 
detailed description of the singular dynamics given in Theorems \ref{thm:informal}, \ref{thm:fluidresult} and \ref{thm:fluidresult2}.

In \Sectionref{sec:SingSymHypSyst}, we develop a new forward method to solve the singular initial value problem for symmetric hyperbolic systems of Fuchsian equations, which yields a new existence result given in  \Theoremref{thm:mainresult} that has two significant advantages over current results. The first advantage is that  \Theoremref{thm:mainresult} allows the coefficients of the spatial derivatives in the Fuchsian equations to contain $1/t$ singular terms. Spatial derivatives are therefore not necessarily required to be negligible at $t=0$. While it is expected that this will open up a significant range of new applications, we do not exploit this particular property here.  The second advantage of our new existence result, which \emph{is} exploited here, is that it is formulated using the same framework that is employed in \cite{beyer2019a} to solve backward problem for Fuchsian equations. The importance of this is that it makes it possible, for the first time, to efficiently study the forward and the backward problems together using the same methodology for large classes of nonlinear partial differential equations. The application to the Euler equations in this paper demonstrates the enormous potential that the forward and backward Fuchsian methods have when used in combination to characterise the asymptotics of solutions to nonlinear partial differential equations.

\subsubsection*{Outline of the paper}
The technical arguments of this paper consist  of two distinct parts that are contained in \Sectionref{sec:SingSymHypSyst} and \Sectionref{sec:applications}. In \Sectionref{sec:SingSymHypSyst}, we consider the forward problem for general \emph{symmetric hyperbolic Fuchsian systems}, and we state and prove a new existence and uniqueness result for the forward problem. The precise statement of the existence and uniqueness result is given in Theorem \ref{thm:mainresult}. This forward method together with the backward method from \cite[Theorem~3.8]{beyer2019a}, see Theorem \ref{thm:Fuchsian-IVP} from the appendix for details, are then applied to the Euler equations on Kasner backgrounds in \Sectionref{sec:applications}. In particular, we use Theorem \ref{thm:mainresult} to prove an existence and uniqueness result in \Sectionref{sec:SIVPEuler} for the singular initial value problem for the Euler equations. This yields solutions with the full expected asymptotic data degrees of freedom and leads to a detailed description of the asymptotics of a large class of fluid solutions. In \Sectionref{sec:EulerPert}, we complement this by establishing the nonlinear stability of these solutions via an application of Theorem \ref{thm:Fuchsian-IVP}. Specifically,  we show that for each of the solutions of the forward problem there is an open set of nearby solutions of the backward problem with the same qualitative asymptotics. Combining this with the uniqueness result for the forward method, we finally establish that each \emph{perturbation of a solution of the forward problem of the Euler equations is again a solution of the forward problem}.

\section{Symmetric hyperbolic Fuchsian systems and the singular
  initial value problem}
\label{sec:SingSymHypSyst}

\subsection{Preliminaries\label{prelim}}

The definitions and notation set out here agree, for the most part, with those of \cite{beyer2019a}.

\subsubsection{Spatial manifolds, coordinates, indexing and partial derivatives}
Throughout this article, unless stated otherwise, $\Sigma$ will denote a closed $n$-dimensional manifold, lower case Latin indices, e.g. $i,j,k$, will run from $1$ to $n$ and will index coordinate indices associated to a local coordinate system $x=(x^i)$ on
$\Sigma$, and $t$ will denote a time coordinate on intervals of the form $[T_0,T_1)$ or $[T_0,T_1]$ for $T_1\le 0$. Note that this convention agrees with that of \cite{beyer2019a} where the time $t$ is taken to be  negative. It worthwhile noting that this convention is opposite to that employed in the introduction where time is assumed to be positive. However, this cause no real difficulties since it is trivial to change between the two conventions using the transformation $t\mapsto -t$.

Partial derivatives with respect to the coordinates will be denoted by
\begin{equation*}
\del{t} = \frac{\partial \;}{\partial t} \AND \del{i} = \frac{\partial \;}{\partial x^i}. 
\end{equation*}

\subsubsection{Vector Bundles\label{Vbundle}}
In the following, we will let $\pi \: :\: V \longrightarrow \Sigma$
denote a rank $N$ vector bundle with fibres $V_x = \pi^{-1}(\{x\})$, $x\in \Sigma$, and use $\Gamma(V)$ to denote the set of all smooth sections of $V$. 
We will assume that $V$ is equipped
with a time-independent connection\footnote{$[\del{t},\nabla]=0$.} $\nabla$, and a time-independent, compatible\footnote{$\del{t}h=0$ and $\nabla_X (h(u,v)) = h(\nabla_X u,v)
+h(u,\nabla_X v)$ for all  $X\in \mathfrak{X}(\Sigma)$ and $u,v\in \Gamma(V)$.}, positive definite metric $h \in 
\Gamma(T^0_2(V))$. 
We will also denote the vector bundle of linear operators that act on the fibres of $V$ by $L(V)=\cup_{x\in \Sigma} L(V_x) \cong V\otimes V^*$. The \textit{transpose} of $A_x \in L(V_x)$, denoted by $A_x^{\tr}$, is then defined as the unique element of $L(V_x)$ satisfying
\begin{equation*}
h(x)(A_x^{\tr}u_x,v_x)= h(x)(u_x,A_x v_x), \quad \forall \: u_x,v_x\in V_x. 
\end{equation*}
Furthermore, given two vector bundles $V$ and $W$ over $\Sigma$, we will
use $L(V,W) = \cup_{x\in \Sigma} L(V_x,W_x) \cong W \otimes V^*$
to denote the vector bundle of linear maps from the fibres
of $V$ to the fibres of $W$.

Additionally, we employ the standard notation $V\oplus W$ for the direct sum, which has fibres $(V\oplus W)_x = V_x\times W_x$  and where addition and scalar multiplication are defined via $\lambda (v_1,w_2)+(v_2,w_2)=(\lambda v_1+w_2,\lambda v_1+w_2)$. Furthermore, if $\Vc\subset V$ and $\Wc\subset W$ are subsets satisfying $\pi(\Vc)=\pi(\Wc)=\Sigma$, then we define
\begin{equation*}
    \Vc\oplus \Wc = \bigcup_{x\in\Sigma}\Vc_x\times \Wc_x \subset V\oplus W
\end{equation*}
where $\Vc_x=V_x\cap\Vc$ and $\Wc_x=W_x\cap \Wc$.
We will often abuse notation and write elements of  $\Vc\oplus \Wc$ as $(v,w)$ where $v\in \Vc$, $w\in\Wc$ and $\pi(w)=\pi(v)$.

For any given vector bundle over $\Sigma$, e.g. $V$, $L(V)$, $V\otimes V$, etc., we will generally use $\pi$ to denote the canonical projection onto $\Sigma$.

Here and below, unless stated otherwise, we will use upper case Latin indices, i.e., $I,J,K$, that run from $1$ to $N$ to index elements of $V$ with respect to a local basis $\{e_I\}$ of $V$. By introducing a local basis $\{e_I\}$, we can represent $u\in \Gamma(V)$ and the inner-product
$h$ locally as
\begin{equation*}
u = u^I e_I \AND
h = h_{IJ}\theta^I\otimes\theta^J,
\end{equation*}
respectively,
where $\{\theta^I\}$ is local basis of $V^*$ determined from the basis $\{e_I\}$ by duality.
Moreover, assuming that the local coordinates $(x^i)$ and the local basis $\{e_I\}$ are defined on the same open region of $\Sigma$,
we can represent the covariant derivative  $\nabla u \in \Gamma(V\otimes T^*M)$ locally by
\begin{equation*}
\nabla u = \nabla_i u^I e_I \otimes dx^i,
\end{equation*}
where 
\begin{equation*}
\nabla_i u^I = \del{i} u^I + \omega_{iJ}^I u^J,
\end{equation*}
and the $\omega_{iJ}^I$ are the connection coefficients determined, as usual, by
\begin{equation*} 
\nabla_{\del{i}} e_J = \omega_{iJ}^I e_I.
\end{equation*} 

We further assume that $\Sigma$ is equipped with a time-independent\footnote{$\del{t}g=0$.} Riemannian metric\footnote{This reference metric $\mathtt g$ is in general unrelated to the spacetime metric $g$.} $\mathtt g\in \Gamma(T^0_2(\Sigma))$ that is given locally in coordinates $(x^i)$ by
\begin{equation*}
\mathtt g = \mathtt g_{ij} dx^i\otimes dx^j.
\end{equation*}
Since the metric determines the Levi-Civita
connection  on the tensor bundle $T^r_s(M)$ uniquely, we can, without confusion, use $\nabla$ to also denote the Levi-Civita connection. The connection on $V$ and
the Levi-Civita connection on $T^r_s(M)$ then determine a connection on the
tensor product $V\otimes T^r_s(\Sigma)$ in a unique fashion, which we will again denote by 
$\nabla$. This connection is compatible with the positive definite inner-product induced on $V\otimes T^r_s(\Sigma)$ by the inner-product $h$ on $V$
and the Riemannian metric $g$ on $\Sigma$. 
With this setup, the \textit{covariant derivative of order $s$} of a section $u \in \Gamma(V)$, denoted $\nabla^s u$, defines an
element of $\Gamma(V\otimes T^0_s(\Sigma))$ that is given locally by
\begin{equation*}
\nabla^s u = \nabla_{i_s} \cdots \nabla_{i_2} \nabla_{i_1} u^I e_I \otimes dx^{i_1}\otimes dx^{i_2}\otimes \cdots \otimes dx^{i_s}.
\end{equation*}
When $s=2$,  the components of $\nabla^2 u$ can be computed using the formula
\begin{equation*}
\nabla_j\nabla_i u^I = \del{j}\nabla_i u^I -\Gamma_{ji}^k\nabla_k u^I + \omega_{j J}^I \nabla_i u^J,
\end{equation*}
where $\nabla_i u^I$ is as defined above, and $\Gamma_{ij}^k$ are the Christoffel symbols of $\mathtt g$. 
Similar formulas exist for the higher covariant derivatives.

\subsubsection{Inner-products and operator inequalities}
We define the \textit{norm} of $v\in V_x$, $x\in \Sigma$, by
\begin{equation*}
|v|^2 = h(x)(v,v).
\end{equation*}
Using this norm, we then define, for $R>0$, the \textit{bundle of open balls of radius $R$ in $V$} by 
\begin{equation*}
B_R(V) = \{\, v\in V \, | \, |v|<R\,\}.
\end{equation*}
Given elements $v, w \in V_x\otimes T^0_s(M_x)$, we can expand them in a local basis as
\begin{equation*}
v= v^I_{i_1 i_2 \cdots i_s} e_I\otimes dx^{i_1}\otimes dx^{i_2}\otimes \cdots \otimes dx^{i_s} \AND w= w^I_{i_1 i_2 \cdots i_s} e_I\otimes dx^{i_1}\otimes dx^{i_2}\otimes \cdots \otimes dx^{i_s},
\end{equation*}
respectively.
Using these expansions, we define the inner-product of $v$ and $w$ by
\begin{equation*}
\ipe{v}{w} = g^{i_1 j_1} g^{i_2 j_2} \cdots g^{i_s j_s} h_{IJ}v^I_{i_1 i_2 \cdots i_s}w^J_{j_1 j_2 \cdots j_s},
\end{equation*}
while the norm of $v$ is defined via
\begin{equation*}
|v|^2 = \ipe{v}{v}.
\end{equation*}

For $A\in L(V_x)$, the \textit{operator norm} of $A$, denoted $|A|_{\op}$, is defined by
\begin{equation*}
|A|_{\op} =\sup\bigl\{\, |\ipe{w}{A v}| \, \bigl| \, w,v \in B_1(V_{x}) \,\bigr\}. 
\end{equation*}
We also define  a related norm for elements $A\in L(V_x) \otimes T_x^*M$  by
\begin{equation*}
|A|_{\op} = \sup\bigl\{\, |\ipe{v}{Aw}| \, \bigl| \, (v,w) \in B_1(V_x \otimes T_x^*M)\times B_1(V_{x}) \,\bigr\}. 
\end{equation*}
From this definition, it not difficult to verify that
\begin{equation*}
|A|_{\op} = \sup \bigl\{\, |\ipe{v}{Aw}| \, \bigl|\, (v,w)\in B_1(V_x\otimes T^0_{s+1}(T_x M))\times B_1(V_x\otimes T^0_{s}(T_x M))\,\bigr\}.
\end{equation*}
This definition can also be extended to elements of $A\in L(V_x) \otimes T_x^*M\otimes T_x M$ in a similar fashion; that is,
\begin{equation*}
|A|_{\op} = \sup\bigl\{\, |\ipe{v}{Aw}| \, \bigl| \, (v,w) \in B_1(V_x \otimes T_x^*M\otimes T_x M)\times B_1(V_{x}) \,\bigr\},
\end{equation*}
where again we have that
\begin{equation*}
|A|_{\op} = \sup \bigl\{\, |\ipe{v}{Aw}| \, \bigl|\, (v,w)\in B_1(V_x\otimes T^1_{s+1}(T_x M))\times B_1(V_x\otimes T^0_{s}(T_x M))\,\bigr\}.
\end{equation*}
Finally, for $A,B \in L(V_x)$, we define
\begin{equation*}
A\leq B
\end{equation*}
if and only if 
\begin{equation*}
\ipe{v}{A v} \leq \ipe{v}{B v}, \quad \forall\: v \in  V_x.
\end{equation*}

\subsubsection{Constants, inequalities and order notation}
We will use the letter $C$ to denote generic constants whose exact dependence on other quantities is not specified and whose value may change from line to line. For such constants,
we will often employ 
the standard notation
\begin{equation*}
a \lesssim b
\end{equation*}
for inequalities of the form
\begin{equation*}
a \leq Cb.
\end{equation*}
On the other hand, when the dependence of the constant on other inequalities needs to be specified, for
example if the constant depends on the norm $\norm{u}_{L^\infty}$, we use the notation
\begin{equation*}
C=C(\norm{u}_{L^\infty}).
\end{equation*}
Constants of this type will always be \textit{non-negative, non-decreasing, continuous functions of their arguments}.

We now turn to defining a notion of the order notation that is applicable for maps from one vector bundle to another. This notation will be used frequently in the proof of various nonlinear estimates. The definition   
begins with four vector bundles $V$, $W$, $Y$ and $Z$ that sit over $\Sigma$, and
maps 
\begin{equation*}
f\in C^0\bigl([T_0,0),C^\infty(\Wc\oplus B_R(V),Z)\bigr)
\AND
g\in C^0\bigl([T_0,0),C^\infty(B_{R}(V),Y)\bigr),
\end{equation*}
where $\Wc\subset W$ is open and $\pi(\Wc)=\Sigma$.
We then say that
\begin{equation*}
f(t,w,v) = \Ordc(g(t,v))  
\end{equation*}
if there exists a $\Rt \in (0,R)$ and a map
\begin{equation*}
\ft \in C^0\bigl([T_0,0),C^\infty(\Wc\oplus B_{\Rt}(V),L(Y,Z))\bigr)
\end{equation*}
such that\footnote{Here, we are using $\nabla_{w,v}$ to denote a covariant derivative operator on the product manifold $W\times V$. Since $\Sigma$ is compact, we know
that such a covariant derivative always exists and it does not matter for our purposes which one is employed.}
\begin{gather*}
f(t,w,v) = \tilde{f}(t,w,v)g(t,v),\\
|\ft(t,w,v)| \leq 1 \AND
|\nabla^s_{w,v}\ft(t,w,v)| \lesssim 1
\end{gather*}
for all $(t,w,v) \in  [T_0,0) \times \Wc \oplus B_{\Rt}(V)$ and $s\geq 1$. For situations where we want to bound $f(t,w,v)$ by $g(t,v)$ up to an undetermined constant of proportionality,
we define
\begin{equation*}
f(t,w,v) = \Ord(g(t,v))  
\end{equation*}
if there exists a $\Rt \in (0,R)$ and a map
\begin{equation*}
\ft \in C^0\bigl([T_0,0),C^\infty(\Wc\oplus B_R(V),L(Y,Z))\bigr)
\end{equation*}
such that 
\begin{gather*}
f(t,w,v) = \tilde{f}(t,w,v)g(t,v)
\intertext{and}
|\nabla^s_{w,v}\ft(t,w,v)| \lesssim 1
\end{gather*}
for all $(t,w,v) \in  [T_0,0) \times \Wc \oplus B_{\Rt}(V)$ and $s\geq 0$.

Finally, we use the notation $C_b^0([0,T),C^\infty(\Wc,Z))$ to denote the subspace of $C^0([0,T),C^\infty(\Wc,Z))$ consisting of all the maps $f\in C^0([0,T),C^\infty(\Wc,Z))$ satisfying $f(t,w)=\Ord(1)$.

\subsubsection{Sobolev spaces}
The $W^{k,p}$, $k\in \Zbb_{\geq 0}$, norm of a section $u\in \Gamma(V)$ is defined by
\begin{equation*}
\norm{u}_{W^{k,p}(\Sigma)} = \begin{cases} \begin{displaystyle}\biggl( \sum_{\ell=0}^k \int_{\Sigma} |\nabla^{\ell} u|^p \nu_g\biggl)^{\frac{1}{p}}  \end{displaystyle} & \text{if $1\leq p < \infty $} \\
 \begin{displaystyle} \max_{0\leq \ell \leq k}\sup_{x\in \Sigma}|\nabla^{\ell} u(x)|  \end{displaystyle} & \text{if $p=\infty$}
\end{cases},
\end{equation*}
where $\nu_g \in \Omega^n(\Sigma)$ denotes the volume form of $g$. The Sobolev space $W^{k,p}(\Sigma,V)$ can then be defined as the completion of the space of smooth sections $\Gamma(V)$ in the norm
$\norm{\cdot}_{W^{k,p}(\Sigma)}$. When $V=\Sigma \times \Rbb$ or the vector bundle is clear from context, we will often write $W^{k,p}(\Sigma)$ instead. 
Furthermore, when $p=2$, we will employ the standard notation $H^k(\Sigma,V)=W^{k,2}(\Sigma,V)$, and we recall that $H^k(\Sigma,V)$ is a Hilbert space with the inner-product given by
\begin{equation*}
\ip{u}{v}_{H^k(\Sigma)} = \sum_{\ell=0}^k \ip{\nabla^{\ell} u}{\nabla^{\ell} v},
\end{equation*}  
where the $L^2$ inner-product $\ip{\cdot}{\cdot}$ is defined by
\begin{equation*}
 \ip{w}{z} = \int_\Sigma \ipe{w}{z} \, \nu_g.
\end{equation*}
We also employ the notation
\begin{equation*}
C_b^0([T_0,0),H^k(V))=C^0([T_0,0),H^k(\Sigma,V))\cap L^\infty([T_0,0),H^k(\Sigma,V))
\end{equation*}
for the set of continuous and bounded maps from
$[T_0,0)$ to $H^k(\Sigma,V)$. 

\subsection{The class of symmetric hyperbolic Fuchsian systems}
\label{sec:symhypFuchs}
In this section, we consider partial differential equations of the form
\begin{equation}
B^0(t,w_1,u)\del{t}u + B^i(t,w_1,u)\nabla_{i} u  = \frac{1}{t}\Bc(t,w_1,u)u + F(t,w_2,u) \quad \text{in $[T_0,0)\times \Sigma$,} \label{symivp.1} 
\end{equation}
where $T_0<0$, the coefficients satisfy the conditions below, and $w_1$ and $w_2$ should be interpreted as prescribed time-dependent sections of vector bundles $Z_1$ and $Z_2$, respectively, over $\Sigma$ while the solution $u$ takes values in a
vector bundle $V$ over $\Sigma$. Since the assumptions below imply, in particular, 
that \eqref{symivp.1} is symmetric hyperbolic, we 
know from standard local-in-time existence and uniqueness theory for symmetric hyperbolic equations that
given initial data at time $t_*\in [T_0,0)$, there exists a unique solution of \eqref{symivp.1} that is defined for a small interval of time about $t_*$ and agrees with the initial data at $t_*$.
What is not at all clear from the local-in-time existence theory is if there are solutions that exist for all times in the interval $[T_0,0)$.

\begin{definition}[Symmetric hyperbolic Fuchsian systems]
  \label{def:symmhypFuchssystems}
  A partial differential equation of the form \eqref{symivp.1} is called a \textit{symmetric hyperbolic Fuchsian system} for the constants $T_0<0$, $\Rsc>0$, $R>0$, $\gamma_1>0$, $\gamma_2>0$, $q\geq 0$, $0<p\leq 1$, $\mu\in\R$, $\lambda>-1$, $\theta\geq 0$ and $\beta \geq 0$, open and bounded subsets $\Zc_1\subset Z_1$ and $\Zc_2\subset Z_2$ with $\pi(Z_1)=\pi(Z_2)=\Sigma$, and maps $B^0$, $\Bt^0$, $\Bc$, 
  $\Ft$, $F_0$, $B_0$, $B_1$ and $\Bt_1$
 provided the following conditions are satisfied:
  \begin{enumerate}
\item The maps $(t,z_1,v)\mapsto B^0(t,z_1,|t|^\mu v)$ and $(t,z_1,v)\mapsto \Bc(t,z_1,|t|^\mu v)$ are in
\begin{equation*}
   C^0_b\bigl([T_0,0), C^\infty\bigl(\Zc_1\oplus B_R(V),L(V)\bigr)\bigr)\cap C^1\bigl([T_0,0), C^\infty\bigl(\Zc_1\oplus B_R(V),L(V)\bigr)\bigr)
\end{equation*} and
$C^0_b\bigl([T_0,0), C^\infty\bigl(\Zc_1\oplus B_R(V),L(V)\bigr)\bigr)$, respectively, and
\begin{gather} 
\pi(B^0(t,z_1,|t|^\mu v))=\pi(\Bc(t,z_1,|t|^\mu v))= \pi((z_1,v)),\notag\\
(B^0(t,z_1,|t|^\mu v))^{\tr}= B^0(t,z_1,|t|^\mu v) \label{B0sym}
\intertext{and}
\frac{1}{\gamma_1} \text{id}_{V_{\pi(v)}} \leq  B^0(t,z_1,|t|^\mu v) \le\gamma_2 \text{id}_{V_{\pi(v)}}\label{B0pos}
\end{gather}
for all $(t,z_1,v)\in  [T_0,0)\times\Zc_1\oplus B_R(V)$.
Moreover, there exists a map
$\Bt^0
\in C^0_b\bigl([T_0,0), C^\infty(\Zc_1,L(V))\bigr)$ satisfying
\begin{equation*}
    \pi(\Bt^0(t,z_1))=\pi(z_1)
\end{equation*}
for all $(t,z_1)\in  [T_0,0)\times \Zc_1$ and
\begin{align*}
B^0(t,z_1,|t|^\mu v)-\Bt^0(t,z_1) &= \Ord(v). 
\end{align*}
\item The map $F$ can be expanded as
\begin{equation} \label{fexp}
F(t,z_2,\cdot) = |t|^{\lambda}\Ft(t,z_2) + |t|^{\lambda} F_0(t,z_2,\cdot)
\end{equation}
where $\Ft \in C^0\bigl([T_0,0),C^\infty(\Zc_2,V)\bigr)$,
the map $(t,z_2,v)\mapsto F_0(t,z_2,|t|^\mu v)$ is
in $C^0_b\bigl([T_0,0), C^\infty(\Zc_2\oplus B_R(V),V)\bigr)$, and
\begin{equation*}
    \pi(\Ft(t,z_2))=\pi(z_2) \AND
\pi (F_0 (t,z_2,|t|^\mu v))=\pi((z_2,v))
\end{equation*}
for all $(t,z_2,v)\in  [T_0,0)\times \Zc_2\oplus B_R(V)$ and
\begin{equation*}
F_0(t,z_2,|t|^\mu v) =\Ord(v).
\end{equation*}
\item The map\footnote{Here, we are using the notation $\sigma(A)$ to denote the natural action of a differential 1-form $\sigma\in \mathfrak{X}^*(\Sigma)$ on an element of
$A\in \Gamma(L(V)\otimes T\Sigma)$, which if we express $A$ and $\sigma$ locally as
\begin{equation*}
A=A^{iI}_J \theta^J\otimes e_I \otimes \del{i} \AND \sigma = \sigma_i dx^i,
\end{equation*} 
respectively, is given by
\begin{equation*}
\sigma(A) = \sigma_i A^{iI}_J  \theta^J\otimes e_I.
\end{equation*}
} $B$ locally defined as $B^{iI}_J \theta^J\otimes e_I \otimes \del{i}$
can be expanded as 
\begin{equation} \label{Bexp}
B(t,z_1, \cdot) = |t|^{-(1-p)}B_0(t,z_1, \cdot) + |t|^{-1}B_1(t,z_1, \cdot)
\end{equation}
where the maps 
$(t,z_1,v)\mapsto B_0(t,z_1, |t|^\mu v)$ and $(t,z_1,v)\mapsto B_1(t,z_1, |t|^\mu v)$ are
in $C^0_b\bigl([T_0,0), C^\infty\bigl(\Zc_1\oplus B_R(V),L(V)\otimes T\Sigma\bigr)\bigr)$ and satisfy
\begin{align*}
\bigl[\sigma(\pi(v))(B(t,z_1,|t|^\mu v))\bigr]^{\tr}=\sigma(\pi(v))(B(t,z_1, |t|^\mu v))
\intertext{and}
\pi(B_0(t,z_1,v))=\pi(B_1(t,z_1,v))=\pi((z_1,v))
\end{align*}
for all $(t,z_1,v)\in [T_0,0)\times \Zc_1\oplus B_R(V)$ and for all $\sigma \in \mathfrak{X}^*(\Sigma)$. Moreover, there
exists a map $\Bt_1\in  C^0_b\bigl([T_0,0), C^\infty(\Zc_1,L(V)\otimes T\Sigma)\bigr)$ such that
\begin{equation*}
  \pi(\Bt_1(t,z_1))=\pi(z_1)
\end{equation*}
for all $(t,z_1)\in [T_0,0)\times \Zc_1$
and
\begin{equation*}
B_1(t,z_1, |t|^\mu v)-\Bt_1(t,z_1) = \Ord(v).
\end{equation*}
\item  The map $\Div\!  B(t,z_1,\dot{z}_1,z'_1,v,v')$
defined locally by\footnote{The notation $D_{\#}$ denotes the linear operator associated to differentiating with respect to the variable slot occupied by $\#$; e.g. $D_v f(v,w)v' = \frac{d\;}{dt}\bigl|_{t=0}f(v+t v',w)$.}
\begin{align} 
\Div\!  B(t,x,z_1,\dot{z}_1,z_1',z_2,v,v') =& \del{t}B^0(t,x,z_1,v) +D_{z_1} B^0(t,x,z_1,v)\cdot p|t|^{q-1}\dot{z}_1
 \notag \\
 &+ D_v B^0(t,x,z_1,v)\cdot (B^0(t,x,z_1,v))^{-1}\Bigl[-B^i(t,x,z_1,v)\cdot v'_i \notag \\
&+
 \frac{1}{t}\Bc(t,x,z_1,v) v + F(t,x,z_2,v)
\Bigr]+\del{i}B^i(t,x,z_1,v) \notag \\
& + D_{z_1} B^i(t,x,z_1,v)\cdot (z'_{1i}-\omega_i(x) z_1)+ D_v B^i(t,x,z_1,v)\cdot (v'_i-\omega_i(x) v) \notag \\
&+\Gamma^i_{ij}(x)B^j(t,x,z_1,v)+\omega_i(x)B^i(t,x,z_1,v)-B^i(t,x,z_1,v)\omega_i(x), \label{divBdef}
\end{align}
where $v=(v^J)$, $v'=(v'{}^J_i)$, 
$z_1=(z^J_1)$, $\dot{z}_1=(\dot{z}{}^J_1)$,
$z'_1=(z'{}^J_{1i})$, $z_2=(z^J_2)$, $\omega_i=(\omega_{iI}^J)$,
$B^i=(B^{iJ}_I)$, $x=\pi(v)=\pi(z_1) =\pi(z_2)$, and
\begin{equation*}
(t,z_1,\dot{z}_1,z'_1,z_2,v,v')\in [T_0,0) \times \Zc_1\oplus B_\Rsc(Z_1)\oplus 
B_\Rsc(Z_1\otimes T^*\Sigma)\oplus \Zc_2
\oplus B_R(V)\oplus B_R(V\otimes T^*\Sigma),
\end{equation*}
satisfies
\begin{align}
 \Div \! B(t,z_1,\dot{z}_1,z'_1,z_2,|t|^\mu v, |t|^\mu v'_1) &= \Ordc\bigl( |t|^{-(1-p)}\theta+ |t|^{-1}\beta\bigl)\label{divBbnd.1}
\end{align}
for some constants $q,\theta,\beta\ge 0$.
We note, by definition, that $\Div\! B$ satisfies 
\begin{equation}\label{divBonshell}
  \Div\!  B(t,w_1,p^{-1}|t|^{1-q}\del{t}w_1,\nabla w_1,w_2,u,\nabla u)
  = \del{t}\bigl(B^0(t,w_1,u)\bigr)+\nabla_i(B^i(t,w_1,u)) 
\end{equation}
for any solution $u$ of \eqref{symivp.1}.
\end{enumerate}
\end{definition}

As we shall see in our application to the Euler equations in Section \ref{sec:applications},  the similarity between the
structural assumptions given above by \Defref{def:symmhypFuchssystems} for the forward problem, i.e., going away from the singularity at $t=0$, and those given in 
Appendix \ref{coeffassumps}
for the backward problem, i.e., going towards the singularity at $t=0$, is a decisive advantage as it allows us to analyse both the forward and backward evolution problems under a similar set of conditions. We expect that this feature will be important in future application.


The main genuine difference between the assumptions from \Defref{def:symmhypFuchssystems} and those from Appendix \ref{coeffassumps} is the \emph{upper} bound in \eqref{eq:evassumptionforwardN} versus the \emph{lower} bound in \eqref{B0BCbndtrafoBW} for the linear operator $\Bc$, which is one of the key structural differences between forward and backward problems as has been noted previously in \cite{ringstrom2017,ames2019}. We further note that the parameter $\lambda$ from \Defref{def:symmhypFuchssystems}, which is not present in Appendix \ref{coeffassumps}, introduces control of the decay of the source term function in a manner that is crucial for the forward problem as obvious from \eqref{eq:minimaldecayassumption} below.
Making appropriate choices for the parameter $\mu$ in \Defref{def:symmhypFuchssystems}, which also appears in \eqref{eq:minimaldecayassumption},  is as critical.
It is worth noting that some of the conditions in \Defref{def:symmhypFuchssystems} become less restrictive  the larger we choose $\mu$ due to the factor $|t|^\mu$ that appears in the arguments of some of the maps there.
However, $\mu$ must be no larger than $\lambda+1$ as a particular consequence of \eqref{eq:minimaldecayassumption} given that $\alpha$ and $p$ are positive.

We close this subsection with the remark that it might be of general interest to allow $\lambda$, $p$ and $\mu$ to be sufficiently smooth \emph{functions} on $\Sigma$. Depending on the application, this might yield a more localised and finer control. In this paper, however, we fully restrict to the constant case for simplicity.

\subsection{The singular initial value problem}
\label{sec:SingSymHypSyst2}
We are now ready to state our new existence and uniqueness theorem for the forward problem, that is, the singular initial value problem for Fuchsian systems. The constant $C_{\text{Sob}}$ that appears in the following theorem is the constant from Sobolev's inequality \cite[Ch.~13, Prop.~2.4.]{taylor2011} for $k\in\Zbb_{> n/2+1}$, that is,
\begin{equation} \label{Sobolev}
\max\bigl\{\norm{f}_{L^\infty(\Sigma)},\norm{\nabla f}_{L^\infty(\Sigma)}\bigr\}\leq C_{\text{Sob}} \norm{f}_{H^{k}(\Sigma)}.
\end{equation}

\begin{thm}[The singular initial value problem for symmetric hyperbolic Fuchsian systems]
\label{thm:mainresult}
Suppose $k \in \Zbb_{>n/2+3}$, 
\Eqref{symivp.1} is  a symmetric hyperbolic Fuchsian system 
for constants $T_0<0$, $\Rsc>0$, $R>0$, $\gamma_1>0$, $\gamma_2>0$, $q\geq 0$, $0<p\leq 1$, $\mu\in \Rbb$, $\lambda>-1$, $\theta\geq 0$ and $\beta\geq 0$,
open bounded sets $\Zc_1\subset Z_1$ and  $\Zc_2\subset Z_2$ with $\pi(\Zc_1)=\pi(
\Zc_2)=\Sigma$, and maps $B^0$, $\Bt^0$, $\Bc$, 
$\Ft$, $F_0$, $B_0$, $B_1$, and $\Bt_1$ as in \Defref{def:symmhypFuchssystems} with 
\begin{equation}
  \label{eq:minimaldecayassumption}
  \lambda+1-(1+\alpha)p = \mu
\end{equation}
for some $\alpha> 0$, $w_1\in C^0_b\bigl([T_0,0),H^k(\Sigma,Z_1)\bigr)\cap C^1\bigl([T_0,0),H^{k-1}(\Sigma,Z_1)\bigr)$ and $w_2\in C^0_b\bigl([T_0,0),H^k(\Sigma,Z_2)\bigr)$ satisfy
\begin{equation}
   \sup_{T_0<t<0} \max\Bigl\{ \norm{\nabla w_1(t)}_{H^{k-1}(\Sigma)},
   \norm{|t|^{1-q}\del{t}w_1(t)}_{H^{k-1}(\Sigma)}\Bigr\} < \frac{\Rsc}{C_{\text{Sob}}},
\end{equation}
and
\begin{equation*}
w_a([T_0,0)\times\Sigma)\subset \Zc_a, \quad a=1,2.
\end{equation*}
Let
\begin{equation}
\label{eq:defb}
\mathtt{b} =  \sup_{T_0 \leq t < 0}\Bigl( \big\|\big|\nabla(\tilde{B}^0(t))\tilde{B}^0(t)^{-1} \tilde{B}_1(t)\bigl|_{\op}\bigr\|_{L^\infty}
+ \big\|\big|\nabla \tilde{B}_1(t)\bigl|_{\op}\bigr\|_{L^\infty} \Bigr)
\end{equation}
where $\tilde{B}^0(t)=\tilde{B}^0(t,w_1(t))$ and
$\tilde{B}_1(t)=\tilde{B}_1(t,w_1(t))$, and
in addition, suppose there exists a $\eta>0$ satisfying
\begin{equation}
  \label{eq:forwardposN}
  2\eta -\gamma_1\beta-k(k+1) \mathtt{b}\gamma_1>0
\end{equation}
for which the inequality
\begin{equation} 
  \label{eq:evassumptionforwardN}  
\Bc(t,z_1,|t|^\mu v)\le \bigl(\mu-\eta \bigr) B^0(t,z_1, |t|^\mu v)
\end{equation}
holds for  all $(t,z_1,v)\in [T_0,0)\times \Zc_1 \oplus  B_{R}(V)$. 
Then there exist
constants $C,\delta> 0$ such that if
\begin{equation}
  \label{eq:forwardsmallnessTh}
\int_{T_0}^0 |s|^{p-1}\norm{\tilde F(s,w_2(s))}_{H^k(\Sigma)}\, ds < \delta,
\end{equation}
then there exists a solution
\begin{equation}
\label{eq:symhypSIVPregNNNNNNNN}
u \in C^0_b\bigl([T_0,0),H^{k}(\Sigma,V)\bigr)\cap C^1\bigl([T_0,0),H^{k-1}(\Sigma,V)\bigr)\subset C^1([T_0,0)\times\Sigma,V)
\end{equation}
of \eqref{symivp.1} that is bounded by
\begin{equation}
\label{eq:solution1estimateboundednessNNNN}
\norm{|t|^{-(\lambda+1-p)}u(t)}_{H^{k}(\Sigma)}\leq C \int_t^0 |s|^{p-1}\norm{\Ft(s,w_2(s))}_{H^k(\Sigma)}\, ds, \quad T_0\leq t<0,
\end{equation}
and
\begin{equation} \label{L^infty-property}
\sup_{t\in [T_0,0)}\max\Bigl\{\bigl\| |t|^{-\mu}u(t)\bigr\|_{L^\infty(\Sigma)},\bigl\| |t|^{-\mu}\nabla u(t)\bigr\|_{L^\infty(\Sigma)}\Bigr\} < R.
\end{equation}
Moreover, the solution $u$ is the unique solution of \eqref{symivp.1} 
within the class $C^1([T_0,0)\times \Sigma,V)$ satisfying
\eqref{L^infty-property}.
\end{thm}

\begin{rem} \label{rem:Fuchsian-SIVP}
If the map $B_1$, see \eqref{Bexp}, vanishes,
then the regularity requirement in the statement of Theorem~\ref{thm:mainresult} can be lowered to $k \in \Zbb_{>n/2+1}$. This is also true for the statements of Propositions~\ref{prop:forward} and \ref{prop:canonicalSIVP} below. The origin of this improvement is that Lemma~3.5 from \cite{beyer2019a} is no longer needed
to derive the estimate \eqref{Gell-estimate} in the proof of Proposition~\ref{prop:forward}. 
\end{rem}

The proof of Theorem \ref{thm:mainresult} is given in \Sectionref{sec:proofs} below, but before we consider the proof, we first make a few observations. First, a direct application of \Theoremref{thm:mainresult} to the model problem \eqref{eq:6model6} would only address the case with \emph{zero asymptotic data} $\ut_*=0$, cf.~\eqref{eq:6model8}. However, this is no loss of generality since \Theoremref{thm:mainresult} can instead be applied to the equation for the \emph{remainder} $w=u-u_0$ defined from the original unknown $u$ and the leading-order term  $u_0$  parametrised by the asymptotic data, e.g.\ $u_0(t)=t^b\ut_*$ in \eqref{eq:6model8}. In this way, we can obtain the solution \eqref{eq:6model8} to the model Fuchsian equation \eqref{eq:6model6} via an application of
\Theoremref{thm:mainresult}.

It is an important property of Fuchsian systems, as we demonstrate in the proof of  \Theoremref{thm:mainresult} below, that a solution to the singular initial value problem \emph{can be approximated on $[T_0,0)$ with arbitrary accuracy} by a solution of the (regular) initial value problem on intervals $[T_0,t_*]$ with suitable (regular) initial data imposed at sufficiently small (regular) times $t_*<0$.
The central idea of the proof of Theorem~\ref{thm:mainresult} is to show that the sequence of solutions $u_{\matchtime}(t)$ generated by solving the {regular initial value problem} of \eqref{symivp.1} with zero Cauchy data $u_{\matchtime}(\matchtime)=0$ imposed at a monotonic sequence of positive times $\matchtime$ with $\matchtime\nearrow 0$ converges to a solution of the singular initial value problem. This idea was put forward in \cite{beyer2010b} and then used in \cite{ames2013a,ames2013b} to establish the first  existence proof of solutions of the singular initial value problem for quasilinear symmetric hyperbolic Fuchsian equations.  Besides being a useful analytic technique, this approximation technique can be employed to construct numerical solutions of the singular initial value problem
\cite{beyer2010b,beyer2011, beyer2020b}.

\subsection{Proof of \Theoremref{thm:mainresult}}
\label{sec:proofs}
\subsubsection{Transformation to a canonical Fuchsian system}
\label{sec:usefultrafos}

The first step in the proof of \Theoremref{thm:mainresult}
is to show that the Fuchsian system  \eqref{symivp.1} 
can be transformed into \textit{canonical form}, which
we now define. The point of this transformation is 
that it allows us to deduce existence for the general class of Fuchsian systems \eqref{symivp.1} from an existence result for the simpler class of
canonical Fuchsian systems.

\begin{definition}[Canonical symmetric hyperbolic Fuchsian systems]
  \label{def:CansymmhypFuchssystems}
  A system of partial differential equations of the form \eqref{symivp.1} is called  a \emph{canonical symmetric hyperbolic Fuchsian system} 
for the constants $T_0<0$, $\Rsc>0$, $R>0$,  $\gamma_1>0$, $\gamma_2>0$, $q\geq 0$, $\lambda>0$, $\theta\geq 0$ and $\beta\geq 0$, open sets $\Zc_1\subset Z_1$ and $\Zc_2\subset Z_2$ with $\pi(\Zc_1)=\pi(\Zc_2)=\Sigma$, and maps $B^0$, $\Bt^0$, $\Bc$, 
$\Ft$, $F_0$, $F_1$, $B_0$, $B_1$ and $\Bt_1$ if these choices of constants, sets and maps along with $p=1$ and
$\mu=0$ make the system symmetric hyperbolic Fuchsian
according to \Defref{def:symmhypFuchssystems}.
\end{definition}

The following lemma guarantees that the Fuchsian system  \eqref{symivp.1} can
always be brought into canonical form provided the condition \eqref{eq:transformcond} below holds. Since the proof follows from a straightforward calculation, the details will be left to the reader.

\begin{lemma}
  \label{lem:trafoproof}
  Suppose $\pbold\in\R\backslash\{0\}$, $\lambdabold\in\R$ and $T_0<T_1\le 0$. Then $u\in C^1([T_0,T_1)\times \Sigma,V)$ is a solution of \eqref{symivp.1} if and only if
  \begin{equation}
    \label{eq:trafodef}
    \bar u(\tau,x)=(-\tau)^{-\lambdabold/\pbold} u(-(-\tau)^{1/\pbold},x),\quad \tau=-(-t)^\pbold,
  \end{equation}
  is a solution of 
\begin{equation}
  \label{eq:trafoFirst}
\Bb^0(\tau,\wb_1,\ub)\del{\tau} \ub + \Bb^i(\tau,\wb_1,\ub)\nabla_i \ub = \frac{1}{\tau} \bar{\Bc}(\tau,\wb_1,\ub) \ub + \Fb(\tau,\wb_2,\ub),
\end{equation}
where
\begin{align}
\wb_1(\tau,x) &= w_1(-(-\tau)^{1/\pbold},x), \\
\wb_2(\tau,x) &= w_2(-(-\tau)^{1/\pbold},x), \\
\Bb^0(\tau,z_1,v) &= B^0\bigl(-(-\tau)^{1/\pbold},z_1, (-\tau)^{\lambdabold/\pbold}v\bigr),\label{Bb0def}\\
\Bb^i(\tau,z_1,v) &= \frac{(-\tau)^{\frac{1-\pbold}{\pbold}}}{\pbold}B^i\bigl(-(-\tau)^{1/\pbold},z_1, (-\tau)^{\lambdabold/\pbold}v\bigr), \label{Bbidef}\\ 
\bar{\Bc}(\tau,z_1,v) &= \frac{\Bc\bigl(-(-\tau)^{1/\pbold},z_1, (-\tau)^{\lambdabold/\pbold}v\bigr)
-\lambdabold B^0\bigl(-(-\tau)^{1/\pbold},z_1, (-\tau)^{\lambdabold/\pbold}v\bigr)}{\pbold},
\label{Bcbdef}\\
\Fb(\tau,z_2,v) &= \frac{(-\tau)^{\frac{1-\pbold-\lambdabold}{\pbold}}}{\pbold} F\bigl(-(-\tau)^{1/\pbold},z_2, (-\tau)^{\lambdabold/\pbold}v\bigr),
\label{Fbdef}
\end{align}
with $\ub\in C^1([-(-T_0)^\pbold,-(-T_1)^\pbold)\times \Sigma,V)$. Furthermore, the map $\Div\! B$ defined by \eqref{divBdef} transforms as
\begin{equation}
  \label{eq:divtrafo}
 \Div\! \Bb\bigl(\tau,z_1,\dot{z}_1,z'_1,z_2,v,v'\bigr)
 = \frac{(-\tau)^{\frac{1-\pbold}{\pbold}}}{\pbold}\Div\! B\bigl(-(-\tau)^{1/\pbold},z_1,\pbold\dot{z}_1\pb/p,z'_1,z_2,(-\tau)^{\lambdabold/\pbold} v, (-\tau)^{\lambdabold/\pbold} v'\bigr),
\end{equation}
provided $\qb$ and $\pb$ are the constants used to define the divergence map on the left side of \eqref{eq:divtrafo} and $q$ and $p$ are the ones used to define the divergence map on the right side with $\qb=q/\pbold$.

\medskip

\noindent In particular, if
 \eqref{symivp.1} is a symmetric hyperbolic Fuchsian system 
for constants $T_0<0$, $\Rsc>0$, $R>0$,  $\gamma_1>0$, $\gamma_2>0$, $q\geq 0$, $0<p\leq 1$, $\mu\in \Rbb$, $\lambda>-1$, $\theta\geq 0$ and $\beta\geq 0$, open bounded sets $\Zc_1\subset Z_1$ and $\Zc_2\subset Z_2$ with $\pi(\Zc_1)=\pi(\Zc_2)=\Sigma$, and maps $B^0$, $\Bt^0$, $\Bc$, 
$\Ft$, $F_0$, $B_0$, $B_1$, and $\Bt_1$ as in \Defref{def:symmhypFuchssystems} such that
\begin{equation}
  \label{eq:transformcond}
    \mu =\lambda+1-(1+\alpha)p 
\end{equation}
for some $\alpha > 0$,
then \eqref{eq:trafoFirst} -- \eqref{Fbdef} with
\begin{equation*}
  \pbold=p,\quad 
  \lambdabold=\mu
\end{equation*}
is a canonical symmetric hyperbolic Fuchsian system
for the choice of constants
$\bar{T}_0=-(-T_0)^p$, $\qb=q/p$, $\bar{\Rsc}=\Rsc$, $\Rb=R$, 
$\gammab_1=\gamma_1$, $\gammab_2=\gamma_2$, $\bar{\lambda}=\alpha$, $\thetab=\theta/p$ and $\bar\beta=\beta/p$, the sets
$\bar{\Zc}_1=\Zc_1$, $\bar{\Zc}_2=\Zc_2$, and maps 
\begin{align*}
    \tilde\Bb^0(\tau,z_1)&=\Bt^0(-|\tau|^{1/p},z_1)  
\intertext{and}
\tilde\Fb(\tau,z_2)&=\Ft\bigl(-|\tau|^{1/p},z_2\bigr)/p,\quad
\Fb_0(\tau,z_2,v)=F_0\bigl(-|\tau|^{1/p},z_2, (-\tau)^{\mu/p} v\bigr)/p.
\end{align*}
Furthermore, $\wb$ and $w$ are related via 
\begin{equation*}
    |\tau|^{1-\qb}\del{\tau}\wb(\tau,x) = \frac{1}{p}(|t|^{1-q}\del{t}w)(-(-\tau)^{1/p},x).
\end{equation*}
\end{lemma}

\subsubsection{Existence and uniform estimates on time intervals 
$[T_0,\matchtime]$ with $\matchtime<0$}
The next step in the proof of Theorem \ref{thm:mainresult} is to establish the existence of solutions to canonical Fuchsian systems on time intervals of the form $[T_0,\matchtime]$
where $|\matchtime|$ can be chosen as small as we like. 
This existence result is established in Proposition
\ref{prop:forward} below. Crucially, the initial data smallness parameter $\delta$ in Proposition \ref{prop:forward} 
is independent of $\matchtime$. This property will
allow us to send $\matchtime$ to zero  in Proposition \ref{prop:canonicalSIVP} below to obtain, in the limit, solutions of the singular initial value problem for canonical Fuchsian systems. We note that  for this construction it essential that the
other constants, e.g. $\Delta,C$, appearing in \Propref{prop:forward} are also independent of $t_*$. 

\begin{proposition} \label{prop:forward}
Suppose $k \in \Zbb_{>n/2+3}$, \Eqref{symivp.1} is a
canonical symmetric hyperbolic Fuchsian system 
for the constants $T_0<0$, $\Rsc$, $R>0$,  $\gamma_1>0$, $\gamma_2>0$, $q\geq 0$, $\lambda>0$, $\theta\geq 0$ and $\beta\geq 0$, open bounded sets $\Zc_1\subset Z_1$ and $\Zc_2\subset Z_2$ with $\pi(\Zc_1)=\pi(\Zc_2)=\Sigma$, and maps $B^0$, $\Bt^0$, $\Bc$, 
$\Ft$, $F_0$, $F_1$, $B_0$, $B_1$ and $\Bt_1$ as in \Defref{def:CansymmhypFuchssystems}, and
$w_1\in C^0_b\bigl([T_0,0),H^k(\Sigma,Z_1)\bigr)\cap C^1\bigl([T_0,0),H^{k-1}(\Sigma,Z_1)\bigr)$ and $w_2\in C^0_b\bigl([T_0,0),H^k(\Sigma,Z_2)\bigr)$ satisfy
\begin{equation*}
   \sup_{T_0<t<0} \max\Bigl\{ \norm{\nabla w_1(t)}_{H^{k-1}(\Sigma)},
    \norm{|t|^{1-q}\del{t}w_1(t)}_{H^{k-1}(\Sigma)}\Bigr\} < \frac{\Rsc}{C_{\text{Sob}}}
\end{equation*}
and
\begin{equation*}
w_a([T_0,0)\times\Sigma)\subset \Zc_a, \quad a=1,2.
\end{equation*}
In addition, suppose
there exists a constant $\eta>0$ such that  
\begin{equation} 
  \label{eq:evassumptionforward}  
\Bc(t,z_1,v)\le -\eta B^0(t,z_1,v)
\end{equation}
for  all $(t,z_1,v)\in [T_0,0)\times \Zc_1\oplus B_{R}(V)$ and that
\begin{equation}
  \label{eq:forwardposA}
  2\eta-\gamma_1\beta-k(k+1) \mathtt{b}\gamma_1>0,
\end{equation}
where $\mathtt{b}$ is defined in \eqref{eq:defb}.
Then there exist constants $0<\Delta< R/C_{\text{Sob}}$ and  $C>0$, independent of $T_1 \in [T_0,0)$
and $t_*\in (T_1,0)$,
such that every solution
\begin{equation*}
u \in C^0\bigl((T_1,\matchtime],H^k(\Sigma,V)\bigr)\cap C^1\bigl((T_1,\matchtime],H^{k-1}(\Sigma,V)\bigr)\subset C^1\bigl((T_1,\matchtime]\times \Sigma,V\bigr) 
\end{equation*}
of \eqref{symivp.1} that is bounded by
\begin{equation}
  \label{eq:defDeltaT1}
  \sup_{t\in (T_1,\matchtime]}\norm{u(t)}_{H^k(\Sigma)}\le\Delta
\end{equation}
also satisfies
\begin{equation}
  \label{eq:forwardestimate}
\norm{|t|^{-\nu}u(t)}_{H^k(\Sigma)}
\leq C\Bigl( \norm{|\matchtime|^{-\nu}u(\matchtime)}_{H^k(\Sigma)} + \int_t^\matchtime \norm{|s|^{\lambda-\nu}\Ft(s,w_2(s))}_{H^k(\Sigma)} ds\Bigr)
\end{equation}
for all $t\in (T_1,\matchtime]$ and $\nu \in \Rbb$.
Moreover, there exists
a $\delta > 0$  such that, for any $\matchtime \in (T_0,0)$ and $u_0\in H^k(\Sigma)$ with 
\begin{equation}
  \label{eq:forwardsmallness}
 \max\left\{\norm{u_0}_{H^k(\Sigma)}, \int_{T_0}^0 \norm{\Ft(s,w_2(s))}_{H^k(\Sigma)}ds\right\}
 \leq \delta,
\end{equation}
there exists a unique solution 
\begin{equation*}
u \in C^0\bigl([T_0,\matchtime],H^k(\Sigma)\bigr)\cap  C^1\bigl([T_0,\matchtime],H^{k-1}(\Sigma)\bigr)
\end{equation*}
of \eqref{symivp.1} that satisfies the initial condition
\begin{equation}
u(t_*) =u_0 \label{symivpforward.2}
\end{equation}
and the estimates \eqref{eq:defDeltaT1} and \eqref{eq:forwardestimate}
for $T_1=T_0$.
\end{proposition}
\begin{proof}
The first step of the proof is to derive the  estimate \eqref{eq:forwardestimate}. The derivation of this estimate is modelled on the derivation of related estimates from the existence part of the proof of Theorem~3.8 in \cite{beyer2019a}. To complete the proof, we then use the estimate \eqref{eq:forwardestimate} in
conjunction with the continuation principle for symmetric hyperbolic equations to establish the existence of solutions to \eqref{symivp.1} on intervals of the form $[T_0,t_*]$ where the initial data is specified at a fixed time $t_*$ in the interval $(T_0,0)$. Crucially, we show for sufficiently small initial data that the time $t_*$ can be chosen as close to zero as we like.

To begin the first step, we suppose $\matchtime\in (T_1,0)$
and $u\in C^0((T_1,\matchtime],H^k(\Sigma,V))\cap C^1((T_1,\matchtime],H^{k-1}(\Sigma,V))$ is a solution of \eqref{symivp.1} that satisfies the initial condition \eqref{symivpforward.2} and the bound \eqref{eq:defDeltaT1} for some  $\Delta< R/C_{\text{Sob}}$,
which we note by the Sobolev inequality \eqref{Sobolev}
implies that
\begin{equation*}
\sup_{t\in (T_1,\matchtime]}\max\Bigl\{\norm{\nabla u(t)}_{L^\infty},
\norm{u(t)}_{L^\infty}\Bigr\}<R.
\end{equation*}
Then by a slight variation of the arguments in \cite{beyer2019a}, we see by taking $\ell$ spatial derivatives of \eqref{symivp.1} for  $0\leq \ell \leq k$ that
$\nabla^\ell u$ satisfies
\begin{align*}
B^0\del{t} \nabla^\ell u + B^i\nabla_i \nabla^\ell  u = & \frac{1}{t}\Bigl[\Bc \nabla^\ell  u  - [\nabla^\ell,B^0](B^0)^{-1}\Bc  u +[\nabla^\ell,\Bc]u\Bigr] 
+ [\nabla^\ell ,B^0](B^0)^{-1}B^i\nabla_i u  \notag \\
&-[\nabla^\ell ,B^i]\nabla_i u 
-B^i[\nabla^\ell,\nabla_i]u 
 - [\nabla^\ell ,B^0](B^0)^{-1}F 
 +  \nabla^\ell F.
\end{align*}
For this proof, we consider $w_1$ and $w_2$ as fixed and for brevity, we omit these from the argument list for all maps. As a consequence of the Moser and Sobolev inequalities, all constants in the following depend on $w_1$ and $w_2$ only via their uniform $H^k(\Sigma)$ bounds over $[T_0,0)$.

From the standard $L^2$-energy identity, we have
\begin{equation} \label{energy-identity}
\frac{1}{2} \del{t} \ip{\nabla^\ell u}{B^0 \nabla^\ell u} 
= -\frac{1}{|t|}\ip{\nabla^\ell u}{\Bc  \nabla^\ell u}+\frac{1}{2} \ip{\nabla^\ell u}{\Div \! B \nabla^\ell u} + \ip{\nabla^\ell u}{G_\ell}, \qquad
0\leq \ell \leq k,
\end{equation}
where  $G_0=F$ and
\begin{align*}
G_\ell = & |t|^{-1}\left([\nabla^\ell,B^0](B^0)^{-1}\Bc  u)-[\nabla^l,\Bc]u\right)
+ [\nabla^\ell,B^0](B^0)^{-1}B^i\nabla_i u \notag \\
&-[\nabla^\ell,B^i]\nabla_i u 
-B^i[\nabla^\ell,\nabla_i]u
 - [\nabla^\ell ,B^0](B^0)^{-1}F 
 + \nabla^\ell F, \qquad  1\leq \ell\leq  k.
\end{align*}

By a slight variation of the estimates from Propositions 3.4.,~3.6.~and 3.7.~from \cite{beyer2019a} and the expansion \eqref{fexp} with $p=1$ and $\mu=0$, we then have
\begin{equation} \label{G0-estimate}
\ip{u}{G_0}\ge -C \bigl(\norm{u}^2_{L^2(\Sigma)} + \norm{|t|^\lambda \Ft}_{L^2(\Sigma)}\norm{u}_{L^2(\Sigma)}\bigr)
\end{equation}
and
\begin{align}
& \ip{\nabla^\ell u}{G_\ell}\geq  -\frac{1}{|t|}\Bigl[\ell\mathtt{b} \norm{ u}^2_{H^k(\Sigma)}+ C(\norm{u}_{H^k(\Sigma)})\bigl(\norm{ u}_{H^k(\Sigma)}\norm{ u}_{H^{k-1}(\Sigma)}+\norm{u}_{H^k(\Sigma)}\norm{ u}^2_{H^k(\Sigma)}\bigr)\Bigr]
\notag \\
&\qquad 
- C(\norm{u}_{H^k(\Sigma)})\Bigl[\norm{u}^2_{H^k(\Sigma)} + \norm{u}_{H^k(\Sigma)}\norm{|t|^{\lambda}\Ft}_{H^k(\Sigma)}\Bigr] \label{Gell-estimate}
\end{align}
for $1\leq \ell \leq k$ where 
$\mathtt{b}$ is defined in \eqref{eq:defb}. 
We see also from \eqref{divBbnd.1} with $\mu=0$ and $p=1$ that
\begin{equation}\label{DivB-estimate}
|\ip{\nabla^\ell u}{\Div\! B \nabla^\ell u}| \leq \theta \norm{\nabla^\ell u}^2_{L^2}
+ |t|^{-1}{\beta} \norm{ \nabla^\ell u}_{L^2}^2.
\end{equation}

Before proceeding, we define the energy norm
\begin{equation*}
\nnorm{u}^2_s =\sum_{\ell=0}^s \ip{\nabla^\ell u}{B^0 \nabla^\ell u},
\end{equation*}
which we note is equivalent to the standard Sobolev norm $\norm{u}_{H^s(\Sigma)}$ since
\begin{equation*}
\frac{1}{\sqrt{\gamma_1}} \norm{\cdot}_{H^s} \leq \nnorm{\cdot}_s \leq \sqrt{\gamma_{2}} \norm{\cdot}_{H^s}
\end{equation*}
holds by \eqref{B0pos} with $\mu=0$.  We will employ this equivalence below without comment.

Now, by \eqref{eq:evassumptionforward} and \eqref{energy-identity}--
\eqref{DivB-estimate}, we find that
\begin{align*} 
\frac{1}{2}& \del{t} \nnorm{\nabla^\ell u}^2_0 
  \ge\frac 1{|t|}\frac{2{\eta}-\gamma_1{\beta}_1}{2} \nnorm{\nabla^\ell u}^2_0-\frac{\gamma_1{\theta}}{2} \nnorm{\nabla^\ell u}^2_{0}
\\
 &-\frac{1}{|t|}\Bigl[\ell {\mathtt{b}} \gamma_1\nnorm{ u}^2_{k}+ C(\nnorm{u}_{k})\bigl(\nnorm{ u}_{k}\nnorm{ u}_{{k-1}}+\nnorm{u}_{k}^3\bigr)\Bigr]
- C(\nnorm{u}_{k})\Bigl[\nnorm{u}^2_{k} + \norm{|t|^\lambda\Ft}_{H^k(\Sigma)}\nnorm{u}_{k}\Bigr]
\end{align*}
for  $1\leq |\ell| \leq k$,
and
\begin{align} 
\del{t} \nnorm{u}^2_0 
  \ge&\frac 1{|t|}(2\eta-\gamma_1{\beta}_1) \nnorm{ u}^2_0 
- C(\nnorm{u}_{k})\Bigl[\nnorm{u}^2_{k} + \norm{|t|^\lambda\Ft}_{H^k(\Sigma)}\nnorm{u}_{k}\Bigr]. \label{k=0-estimate}
\end{align}
Summing these estimates over $\ell$ from $0$ to $k$ yields
\begin{align} 
\del{t} \nnorm{u}^2_k 
  \ge&\frac 1{|t|}\bigl[2\eta-\gamma_1{\beta}_1-k(k+1) \mathtt{b}\gamma_1-C(\nnorm{u}_{k}) \nnorm{u}_{k}\bigr] \nnorm{u}^2_k 
\notag\\
&
 -\frac{1}{|t|}C(\nnorm{u}_{k})\nnorm{ u}_{k}\nnorm{ u}_{{k-1}}
 - C(\nnorm{u}_{k})\Bigl[\nnorm{u}^2_{k} + \norm{|t|^\lambda\Ft}_{H^k(\Sigma)}\nnorm{u}_{k}\Bigr].  \label{k-estimate}
\end{align}
From this inequality and the bound \eqref{eq:defDeltaT1}, we see that
\begin{align} 
\del{t} \nnorm{u}^2_0 
  \ge&\frac 1{|t|}{\rho_0} \nnorm{ u}^2_0
- C(\Delta)\nnorm{u}^2_{k} - C(\Delta)\norm{|t|^\lambda\Ft}_{H^k(\Sigma)}\nnorm{u}_{k} \label{k=0-estimateA}
\end{align}
where $\rho_0$ is defined by setting $k=0$ in
\begin{equation*}
    \rho_k=2{\eta}-\gamma_1{\beta}_1-k(k+1) {\mathtt{b}}\gamma_1.    
\end{equation*}
Due to assumption \eqref{eq:forwardposA}, we
note that $\rho_0>0$ and $\rho_k>0$.

Applying the Sobolev interpolation estimates\footnote{Or alternatively Ehrling's lemma.} \cite[Thm.~5.2]{AdamsFournier:2003} 
to \eqref{k-estimate}, we find, with the help of the bound \eqref{eq:defDeltaT1}, that
\begin{align*} 
\del{t} \nnorm{u}^2_k 
  \ge&\frac{1}{|t|}\bigl[2\eta-\gamma_1{\beta}_1-k(k+1) \mathtt{b}\gamma_1-C(\Delta) (\epsilon+\nnorm{u}_{k})-\epsilon\gamma_1{\beta}_0\bigr] \nnorm{u}^2_k \notag \\
 &-\frac{1}{|t|}c(\Delta,\epsilon^{-1})\nnorm{ u}_{0}^2
 - C(\Delta)\nnorm{u}^2_{k} -C(\Delta)\norm{|t|^\lambda\Ft}_{H^k(\Sigma)}\nnorm{u}_{k}
\end{align*}
for any $\epsilon>0$.
Adding ${1}/\rho_0 c(\Delta,\epsilon^{-1})$ times  \eqref{k=0-estimate} to the above inequality yields
\begin{align} 
\del{t} \left(\nnorm{u}^2_k +\frac{1}{\rho_0} c(\Delta,\epsilon^{-1}) \nnorm{u}^2_0\right)
  \ge&\frac{1}{|t|}\Bigl[2\eta-\gamma_1{\beta}_1-k(k+1) {\mathtt{b}}\gamma_1-C(\Delta) (\epsilon+\nnorm{u}_{k})-\epsilon\gamma_1{\beta}_0\Bigr] \nnorm{u}^2_k \notag \\
 &
 - C(\Delta,\epsilon^{-1})\nnorm{u}^2_{k} -C(\Delta,\epsilon^{-1})\norm{|t|^\lambda\Ft}_{H^k(\Sigma)}\nnorm{u}_{k}, \label{propEnergyA}
\end{align}
which holds for all $t\in (T_1,\matchtime]$ and $\epsilon>0$ small enough. 
Since $\rho_k>0$ by assumption, it follows, by 
choosing $\epsilon>0$ sufficiently small and by shrinking  
$\Delta \in (0,R/C_{\text{Sob}})$ if necessary, that
\[\tilde\rho_k:=\rho_k-C(\Delta) (\epsilon+\Delta) >0.\]
Assuming $\Delta$ and $\epsilon$ are chosen so that this holds, we shall no longer write the monotonic dependencies of the constants on our particular choices of $\Delta$ and $\epsilon$ as we now consider these choices as fixed.

Now, we observe that \eqref{propEnergyA} can be written as
\begin{align*} 
\del{t} \Bigl(\underbrace{\nnorm{u}^2_k +\frac{c}{\rho_0}  \nnorm{u}^2_0
-\int_t^\matchtime\frac{\rhot_k}{s} \nnorm{u(s)}^2_kds}_{=:E(t)}\Bigr)
\ge
 - C\left(\nnorm{u}^2_{k} +\norm{|t|^\lambda\Ft}_{H^k(\Sigma)}\nnorm{u}_{k}\right),
\end{align*}
which implies the  differential energy inequality
\begin{equation*}
\del{t} E(t)\ge -C E(t)-C \norm{|t|^\lambda\Ft}_{H^k(\Sigma)}\sqrt{E(t)}.
\end{equation*}
But this implies
\begin{equation}
\label{diff-energy-inequality}
\del{t} \sqrt{E(t)}\ge -C \sqrt{E(t)}-C \norm{|t|^\lambda\Ft}_{H^k(\Sigma)},
\end{equation}
and so we conclude, by Gr\"onwall's inequality, that
\[\sqrt{E(t)}\le e^{C|T_1|}\left(\sqrt{E(\matchtime)}+C\int_t^\matchtime \norm{s^\lambda\Ft(s)}_{H^k(\Sigma)}\,ds\right),\]
where here and below
$\Ft(s)$ will be used to denote $\Ft(s,w_2(s))$.
Since this inequality holds for all  $t\in (T_1,\matchtime]$, we have established the estimate \eqref{eq:forwardestimate} for $\nu=0$. 
We further note that the differential energy inequality \eqref{diff-energy-inequality} implies
\begin{equation*}
\del{t} \Bigl(|t|^{-\nu}\sqrt{E(t)}\Bigr)\ge \frac{\nu}{|t|}|t|^{-\nu}\sqrt{E(t)}-C |t|^{-\nu} E(t)-C \norm{|t|^{\lambda-\nu}\Ft}_{H^k(\Sigma)} \ge -C |t|^{-\nu} \sqrt{E(t)}-C \norm{|t|^{\lambda-\nu}\Ft}_{H^k(\Sigma)},
\end{equation*}
which shows, via another application of  Gr\"onwall's inequality, that
\eqref{eq:forwardestimate} holds for \emph{arbitrary} $\nu \in \Rbb$.

To complete the proof, we now turn to establishing the  small data existence result for the Cauchy problem of \eqref{symivp.1} and \eqref{symivpforward.2} on $[T_0,\matchtime]$ for arbitrary $\matchtime\in (T_0,0)$. Let us start by imposing the small data condition
\begin{equation*}
    \|u_0\|_{H^k(\Sigma)}\le\delta < \frac{\Rc}{5}
\end{equation*}
where
\begin{equation}
  \label{eq:RcPrelimchoice}
  \Rc=\frac{R}{C_{\text{Sob}}},
\end{equation}
which we note by Sobolev's inequality \eqref{Sobolev} implies
that $\max\bigl\{\|u(\matchtime)\|_{L^\infty(\Sigma)},\|\nabla u(\matchtime)\|_{L^\infty(\Sigma)} \bigr\}\le R/5$.
As a consequence of the standard local-in-time existence
and uniqueness result for the Cauchy problem of quasilinear symmetric hyperbolic systems, see  e.g. \cite[Ch.~16, Prop.~1.4.]{taylor2011},
there is therefore a $T^*\in [T_0,t_*]$ such that this Cauchy problem has a unique solution
\begin{equation}
  \label{eq:bootstrap0}
  u\in C^0((T^*,\matchtime],H^k(\Sigma,V))\cap C^1((T^*,\matchtime],H^{k-1}(\Sigma,V))
\end{equation}
satisfying
\begin{equation}
  \label{eq:bootstrap2}
  \|u(t)\|_{H^k(\Sigma)} < \frac{\Rc}{2}
\end{equation}
for all $t\in(T^*,\matchtime]$.
In fact, we can assume that $|T^*|$ is the \emph{maximal time} for which \eqref{eq:bootstrap0} and \eqref{eq:bootstrap2} hold for all $t\in (T^*,\matchtime]$. 
Suppose now that $T^*$ was strictly larger than $T_0$, i.e.\ the solution $u$ can \emph{not} be extended to the \emph{whole} interval $(T_0,\matchtime]$ as a solution with the properties \eqref{eq:bootstrap0} and \eqref{eq:bootstrap2}.
As a consequence of continuity of $u$ on $(T^*,\matchtime]$ and the fact that $\|u(\matchtime)\|_{H^k(\Sigma)}\le\Rc/5$, it follows that either $\|u(t)\|_{H^k(\Sigma)}< \Rc/3$
for all $t\in (T^*,\matchtime]$, or, there exists a first $T_*\in (T^*,\matchtime)$ such that
$\|u(T_*)\|_{H^k(\Sigma)}=\Rc/3$. 
In the first case, we set $T_*=T^*$. In both cases, we have
\begin{equation} \label{eq:bootstrap1}
\|u(t)\|_{H^k(\Sigma)}\leq \Rc/3
\end{equation}
for all $t\in (T_*,\matchtime]$. 

Now, shrinking $\Rc$ from its preliminary choice \eqref{eq:RcPrelimchoice}  so that $\Rc\le\Delta$ holds now, we can therefore guarantee by \eqref{eq:bootstrap1} that \eqref{eq:defDeltaT1} holds for $T_1=T_*$, and
by \eqref{eq:forwardestimate} with $\nu=\lambda$, that
\begin{equation*}
\norm{u(t)}_{H^k(\Sigma)}\leq C\Bigl( \delta + |T_0|^\lambda \int_t^\matchtime \norm{\Ft(s))}_{H^k(\Sigma)}\, ds\Bigr)
\end{equation*}
for all $t\in (T_*,\matchtime]$. Assuming now that
\begin{equation*}
  \int_{T_0}^0 \norm{\Ft(s)}_{H^k(\Sigma)}\,ds\le \delta
\end{equation*}
and tightening the condition for $\delta$ so that
$\delta\le \frac {\Rc}{4C(1+|T_0|^\lambda)}$,
it follows that $\norm{u(t)}_{H^k(\Sigma)}\le \Rc/4$ 
for all $t\in (T_*,\matchtime]$. By the definition of $T_*$ and continuity, we conclude that $T_*=T^*$ and  that therefore $\norm{u(t)}_{H^k(\Sigma)}\le \Rc/4$ for all $t\in (T^*,\matchtime]$. But this
implies that
\begin{equation*}
\max \Bigl\{\norm{u(t)}_{L^\infty(\Sigma)},
\norm{u(t)}_{L^\infty(\Sigma)}\Bigr\}\le \frac{R}{4},
\quad T^*<t<\matchtime,
\end{equation*}
and so we conclude from the maximality of $T^*$ and the continuation principle for symmetric hyperbolic equations \cite[Ch.~16, Prop.~1.5.]{taylor2011} that we must in fact have that $T^*=T_0$ and the solution extends continuously to the closed time interval $[T_0,\matchtime]$,
which completes the proof. 
\end{proof}

\subsubsection{Singular initial value problem existence and uniqueness for canonical Fuchsian systems}
We now turn to establishing the existence and uniqueness of
solutions to the singular initial value problem for canonical Fuchsian systems. Existence for the singular initial value problem is obtained from the solutions to the Cauchy problem on the intervals $[T_0,t_*]$ from Proposition \ref{prop:forward} by letting $\matchtime\nearrow 0$. 

\begin{proposition}\label{prop:canonicalSIVP}
Suppose $k \in \Zbb_{>n/2+3}$,  \Eqref{symivp.1} is a canonical symmetric hyperbolic Fuchsian system for the constants $T_0<0$, $R>0$, $\gamma_1>0$, $\gamma_2>0$, $q\geq 0$, $\lambda>0$, $\theta\geq 0$ and $\beta\geq 0$, open bounded sets $\Zc_1\subset Z_1$ and $\Zc_2\subset Z_2$ with $\pi(\Zc_1)=\pi(\Zc_2)=\Sigma$, and maps $B^0$, $\Bt^0$, $\Bc$, $\tilde{\Bc}$, $\Ft$, $F_0$, $F_1$, $B_0$, $B_1$ and $\Bt_1$ as in \Defref{def:CansymmhypFuchssystems}, and
$w_1\in C^0_b\bigl([T_0,0),H^k(\Sigma,Z_1)\bigr)\cap C^1\bigl([T_0,0),H^{k-1}(\Sigma,Z_1)\bigr)$ and $w_2\in C^0_b\bigl([T_0,0],H^k(\Sigma,Z_2)\bigr)$ satisfy
\begin{equation*}
   \sup_{T_0<t<0} \max\Bigl\{ \norm{\nabla w_1(t)}_{H^{k-1}(\Sigma)},
    \norm{|t|^{1-q}\del{t}w_1(t)}_{H^{k-1}\infty(\Sigma)}\Bigr\} < \frac{\Rsc}{C_{\text{Sob}}}
\end{equation*}
and
\begin{equation*}
w_a([T_0,0)\times\Sigma)\subset \Zc_a, \quad a=1,2.
\end{equation*}
In addition, suppose
there exists a constant $\eta>0$ such that  
\begin{equation} 
  \label{eq:evassumptionforwardA}  
\Bc(t,z_1,v)\le -\eta B^0(t,z_1,v)
\end{equation}
for  all $(t,z_1,v)\in (T_0,0)\times \Zc_1 \times B_{R}(V)$ and that \eqref{eq:forwardposA} holds,
where $\mathtt{b}$ is defined in \eqref{eq:defb}.
Then there are
constants  $C, \delta > 0$ such that if
\begin{equation} \label{Fbnd}  
\int_{T_0}^0 \norm{\Ft(s,w_2(s))}_{H^k(\Sigma)}\, ds < \delta,
\end{equation}
there exists a classical solution 
\begin{equation}
u \in C^0_b\bigl([T_0,0),H^{k}(\Sigma,V)\bigr)\cap C^1\bigl([T_0,0),H^{k-1}(\Sigma,V)\bigr) \subset C^1\bigl([T_0,0)\times\Sigma,V\bigr)
\end{equation}
of \eqref{symivp.1}
that is bounded by
\begin{equation} \label{supbound}
\max\Bigl\{\norm{u}_{L^\infty([T_0,0)\times \Sigma)},
    \norm{\nabla u}_{L^\infty([T_0,0)\times \Sigma)}\Bigr\} < R
\end{equation}
and
\begin{equation}
\label{eq:solution1estimateboundedness}
\norm{|t|^{-\lambda}{u}(t)}_{H^{k}(\Sigma)}\leq C\int_t^0 \norm{\Ft(s,w_2(s))}_{H^k(\Sigma)}\, ds
\end{equation}
for all $t\in[T_0,0)$. Moreover, the solution $u$ is the unique
solution within the class $C^1([T_0,0)\times \Sigma,V)$ satisfying the bound \eqref{supbound}.
\end{proposition}
\begin{proof}$\;$

\noindent \underline{Existence:}
Let $\{t_n\}_{n=1}^\infty \subset [T_0,0)$ be a monotonically increasing sequence satisfying $\lim_{n\rightarrow \infty} t_n =0$. 
Then by Proposition \ref{prop:forward},
we know, for $\delta>0$ small enough, that there exist solutions
\begin{equation*}
u_n \in C^0\bigl([T_0,t_n],H^k(\Sigma,V)\bigr)\cap C^1\bigl([T_0,t_n],H^{k-1}(\Sigma,V)\bigr) 
\end{equation*}
of \eqref{symivp.1} satisfying the initial condition
\begin{equation} \label{u_nID}
    u_n(t_n)=0.
\end{equation}
Moreover, there exist constants $\Delta,C$ that satisfy
$\Delta < R/C_{\text{Sob}}$ and
are both independent of $n$ 
such that the solutions $u_n$ are bounded by
\begin{equation} \label{u_nDeltabnd}
  \sup_{t\in[T_0,t_n]}
  \norm{u_n(t)}_{H^k(\Sigma)}\le\Delta
\end{equation}
and
\begin{equation}\label{u_nCbnd}
\norm{|t|^{-\lambda}u_n(t)}_{H^k(\Sigma)}
\leq C\int_t^{t_n} \norm{\Ft(s)}_{H^k(\Sigma)}\, ds
\end{equation}
for all $t\in [T_0,t_n]$, where
here, and below $\Ft(s)$ is used to denote $\Ft(s,w_2(s))$.
Throughout this proof, $w_1$ and $w_2$ are taken to be fixed and for brevity, we omit these from the argument list for all maps. As a consequence of the Moser and Sobolev inequalities, all constants in the following depend on $w_1$ and $w_2$ only via their uniform $H^k(\Sigma)$ bounds over $[T_0,0)$.

Now, using the evolution equation \eqref{symivp.1}, which $u_n$ satisfies, we can express $\del{t}u_n$ in terms of $u_n$ and its spatial derivative
as follows
\begin{equation} \label{dtu_n}
    t\del{t}u_n= B^0(t,u_n)^{-1}\bigl(-t B^i(t,u_n)\nabla_i u_n
    + \Bc(t,u_n)u_n + t F(t,u_n) \bigr).
\end{equation}
With the help of the calculus inequalities, in particular the Sobolev, Product and Moser calculus inequalities, see  \cite[Ch.~13, \S 2 \&  3]{taylor2011}, it is then not difficult to
verify from the formula \eqref{dtu_n} and the coefficient assumptions, see Definition \eqref{def:CansymmhypFuchssystems}, that
\eqref{Fbnd} and \eqref{u_nDeltabnd} imply the uniform
bound
\begin{equation} \label{dtu_nbnd}
  \int_{t_n}^0 
  \norm{t\del{t}u_n(s)}_{H^{k-1}(\Sigma)}\, ds\lesssim 1, \quad
  n\geq 1.
\end{equation}

To proceed, we define, for $s_1, s_2\in \Zbb_{\geq 0}$, $\rho\in \Rbb$
and $1\leq p \leq \infty$, the spaces $X^{s_1,s_2,p}_{\rho}$ as the
closure of $C^\infty([T_0,0]\times \Sigma,V)$ in the norm
\begin{equation*}
\norm{v}_{X^{s_1,s_2,p}_{\rho}}= \sum_{0\leq i\leq s_1}
\Bigr\|\big\||t|^{-\rho+i}\del{t}^iv(t)\bigr\|_{H^{s_2-i}(\Sigma)}\Bigr\|_{L^p([T_0,0))}.
\end{equation*}
The spaces $X^{s_1,s_2,p}_{\rho}$ are reflexive for $1<p<\infty$,
and we know from the Rellich-Kondrachov theorem \cite[Thm.~6.3]{AdamsFournier:2003} that
the inclusion\footnote{The only thing that is not standard about this statement is due to the weighting of the time derivatives by powers of $t$. However, by introducing a new time coordinate via $\tau=-\ln(-t)$, these weighted spaces can be turned into exponential weighted Sobolev space on $(e^{-|T_0|},\infty)$. Compactness for the inclusions of the  exponential weighted spaces then follows easily from the Rellich-Kondrachov theorem on unweighted Sobolev space and a covering argument; see \cite[\S 3.2]{Oliynyk:AHP_2006} for details.}
\begin{gather} \label{compact}
  X^{1,k,1}_{\rhot} \subset X^{0,k-1,q}_{\rho}  
\end{gather}
is compact for $1\leq q < \infty$ and $\rhot > \rho$.

Next, we extend the solutions $u_n$ to the whole time interval $[T_0,0)$ by defining 
\begin{equation} \label{ut_ndef}
    \ut_n(t,x) = H_{n}(t)u_n(t,x)
\end{equation}
where 
\begin{equation*}
    H_n(t) = \begin{cases} 1 & \text{if $t\leq t_n$}\\
    0 & \text{if $t>t_n$}\end{cases}
\end{equation*}
is the unit step function with jump at $t_n$. Then, differentiating $\ut_n$, we see, in the sense of distributions, that 
\begin{equation*}
    \del{t} \ut_n(t,x) = u_n(t,x)\del{t}H_{n}(t)
    + H_{n}(t)\del{t}u_n(t,x) =  u_n(t,x)\delta_{t_n}(t) + H_{n}(t)\del{t}u_n(t,x),
\end{equation*}
which by \eqref{u_nID} reduces to
\begin{equation} \label{dtut_n}
    \del{t} \ut_n(t,x) = H_{n}(t)\del{t}u_n(t,x).
\end{equation}
By \eqref{u_nDeltabnd}, \eqref{dtu_nbnd}, \eqref{ut_ndef}, and \eqref{dtut_n}, we see that the sequence $\ut_n$ is bounded by
\begin{equation}
  \label{eq:sequencebounded}
\norm{\ut_n}_{X^{0,k,p}_{0}} \leq \norm{1}_{L^p([T_0,0))}\Delta 
\AND
    \norm{\ut_n}_{X^{1,k,1}_{0}} \leq C
  \end{equation}
for $1<p<\infty$ and $n\geq 1$ where the constant $C$ is independent of $n$.
Since $X^{0,k,p}_{0}$ is reflexive for $1<p<\infty$, we conclude from
the first inequality in \eqref{eq:sequencebounded} and the sequential Banach–Alaoglu theorem\footnote{See Theorem~3.17 in \cite{rudin1990}.}
the existence of a subsequence of $\ut_n$, which we also denote by $\ut_n$, that converges weakly to an element $u\in X^{0,k,p}_{0}$
satisfying the estimate
\begin{gather*}
\norm{u}_{X^{0,k,p}_{0}} \leq \norm{1}_{L^p([T_0,0))}\Delta 
\end{gather*}
for $1<p<\infty$.
Using the fact that $\sup_{T_0 \leq t\leq 0}|f(t)|= \lim_{p\rightarrow \infty}\bigl( \frac{1}{|T_0|}\int_{T_0}^0 |f(\tau)|^p \, d\tau\bigr)^{\frac{1}{p}}$
holds for measurable functions $f(t)$ on $[T_0,0)$ satisfying $|f|^{p_*}\in L^1([T_0,0))$ for some $p_*\in \Rbb$, we deduce 
from the above inequality that 
\begin{gather*}
\norm{u}_{X^{0,k,\infty}_{0}} \leq \Delta.
\end{gather*}
Moreover, by the compactness of the inclusion \eqref{compact},
we can, by selecting a further subsequence of $\ut_n$ if necessary,
assume that the sequence $\ut_n$ converges strongly in $X^{0,k-1,q}_{\epsilon}$ to $u$ for any $\epsilon>0$ and $1\leq q <\infty$ due to the second inequality in \eqref{eq:sequencebounded}. Since $\Delta < R/C_{\text{Sob}}$,
we note from the above bound and Sobolev's inequality \eqref{Sobolev} that
\begin{equation} \label{supboundA}
\max\Bigl\{\norm{u}_{L^\infty([T_0,0)\times \Sigma)},
    \norm{\nabla u}_{L^\infty([T_0,0)\times \Sigma)}\Bigr\} < R
    \AND \sup_{t\in [T_0,0)}\norm{u(t)}_{H^k(\Sigma)} < \frac{R}{C_{\text{Sob}}}.
\end{equation}

Since the sequence $t_n$ is monotonically increasing and converges to $0$ as $n\rightarrow \infty$, it is clear from the definition \eqref{ut_ndef} that for any test function $\psi\in C^\infty_0([T_0,0)\times\Sigma,V)$ there exists an $N=N(\psi)\in \Zbb_{>0}$ such that
\begin{equation} \label{psiut_n}
    \del{t}\psi\ut_n = \del{t}\psi u_n \AND  \psi\ut_n = \psi u_n,\quad n\geq N,
\end{equation} 
and so we conclude from
the weak convergence $\ut_n \rightharpoonup u\in X^{0,k,p}_0$ that
\begin{equation*}
    \ip{\del{t}\psi}{u}= \lim_{n\rightarrow\infty}\ip{\del{t}\psi}{u_n}
    = -\lim_{n\rightarrow\infty}\ip{\psi}{\del{t}u_n},
\end{equation*}
where $\ip{\cdot}{\cdot}$ is the $L^2$ inner-product on $[T_0,0)\times\Sigma$. Using \eqref{dtu_n} and \eqref{psiut_n}, we
can write this as
\begin{equation*}
    \ip{\del{t}\psi}{u}
    = -\lim_{n\rightarrow\infty} \bigl\langle\psi \Bigl|
    B^0(t,\ut_n)^{-1}\bigl(-B^i(t,\ut_n)\nabla_i \ut_n
    + t^{-1}\Bc(t,\ut_n)\ut_n + F(t,\ut_n) \bigr)\bigr \rangle.
\end{equation*}
It is not difficult to verify from 
the strong convergence $\ut \rightarrow u \in X^{0,k-1,q}_{\epsilon}$ that, in the limit $n\rightarrow \infty$, we
can replace $\ut_n$ with $u$ in the above expression to
get 
\begin{equation*}
  \ip{\del{t}\psi}{u}
    = -\bigl\langle\psi \Bigl|
    B^0(t,u)^{-1}\bigl(-B^i(t,u)\nabla_i u
    + t^{-1}\Bc(t,u)u + F(t,u) \bigr)\bigr \rangle.
\end{equation*}
Since the test function $\psi$ was chosen arbitrarily, we
conclude that $u$ defines a weak solution
of the Fuchsian system \eqref{symivp.1} 
where
\begin{equation*}
    \del{t}u= B^0(t,u)^{-1}\bigl(-B^i(t,u)\nabla_i u
    + t^{-1}\Bc(t,u)u + F(t,u)\bigr).
\end{equation*}
By another application of the calculus inequalities (i.e. Sobolev, Product and Moser calculus inequalities, see  \cite[Ch.~13, \S 2 \& 3]{taylor2011}), we get, for each $T_1\in (T_0,0)$, that $\sup_{t\in[T_0,T_1)}\norm{t\del{t}u(t)}_{H^{k-1}(\Sigma)}\lesssim 1$, and hence, that
\begin{equation} \label{uregANN}
    u \in L^\infty\bigl([T_0,0),H^k(\Sigma)\bigr)\cap W^{1,\infty}_{\text{loc}}\bigl([T_0,0),H^{k-1}(\Sigma)\bigr).
\end{equation}
By standard local-in-time existence and uniqueness theory for symmetric 
hyperbolic equations, e.g. \cite[Ch.~16, Prop.~1.4.]{taylor2011}, it follows that the solution $u$ enjoys the improved regularity
\begin{equation*} 
    u \in C^0_b([T_0,0),H^k(\Sigma,V)\bigr)\cap C^{1}\bigl([T_0,0),H^{k-1}(\Sigma,V)\bigr) 
    \subset C^1\bigl([T_0,0)\times \Sigma,V\bigr),
\end{equation*}
and consequently, $u$ defines a classical solution
of \eqref{symivp.1} on $[T_0,0)\times\Sigma$. 

To finish the existence part of the proof, let $\eta \in C^\infty_0([T_0,0))$ be any non-negative test function. Then we observe from
\eqref{u_nCbnd} and \eqref{psiut_n} that there
exists an $N=N(\eta)\in \Zbb_{>0}$ such that
\begin{equation*}
    \norm{\eta(t)|t|^{-\lambda}\ut_n(t)}_{H^k(\Sigma)}\lesssim \eta(t)\int^{0}_{t} \norm{\Ft(s)}_{H^k(\Sigma)}\, ds \leq \delta \eta(t)
\end{equation*}
for all $t\in [T_0,0)$, $n\geq N$ and $1<p<\infty$.
But this implies
\begin{equation*}
\norm{\eta\ut_n}_{X^{0,k,p}_{0}}^p 
\lesssim \int_{T_0}^0 \eta(t)^p \left(\int_t^{0} \norm{\Ft(s)}_{H^k(\Sigma)}\,ds\right)^p \, dt,
\quad n\geq N,
\end{equation*}
and so, we conclude via the weak convergence 
$\eta\ut_n \rightharpoonup \eta u\in X^{0,k,p}_{0}$
that
\begin{equation*}
\norm{\eta u}_{X^{0,k,p}_{0}}^p=\int_{T_0}^0 \eta(t)^p \norm{|t|^{-\lambda}u(t)}_{H^k(\Sigma)}^p\, dt
\lesssim \int_{T_0}^0 \eta(t)^p \left(\int_t^{0} \norm{\Ft(s)}_{H^k(\Sigma)}\,ds\right)^p \, dt, \quad 1<p<\infty.
\end{equation*}
Letting $p\searrow 1$ in the above expression then gives 
\begin{equation*}
\int_{T_0}^0 \eta(t) \norm{|t|^{-\lambda}u(t)}_{H^k(\Sigma)}\, dt
\lesssim \int_{T_0}^0 \eta(t) \int_t^{0} \norm{\Ft(s)}_{H^k(\Sigma)}\,ds \, dt
\end{equation*}
by the Dominated Convergence Theorem. 
However, since $\eta$ was an arbitrarily chosen non-negative test function, we must have
\begin{equation}\label{ut2bnd}
\norm{|t|^{-\lambda}u(t)}_{H^k(\Sigma)}\, dt
\leq C \int_t^{0} \norm{\Ft(s)}_{H^k(\Sigma)}\,ds,
\quad T_0<t<0,
\end{equation}
which completes the
existence part of the proof.

\bigskip

\noindent \underline{Uniqueness:} To establish uniqueness, suppose
that $\uh \in C^1([T_0,0)\times \Sigma,V)$ is another classical solution of \eqref{symivp.1} that satisfies
\begin{gather}
    \max\Bigl\{\norm{\uh}_{L^\infty([T_0,0)\times \Sigma)},
    \norm{\nabla \uh}_{L^\infty([T_0,0)\times \Sigma)}\Bigr\} < R.
    \label{uhbndA}
\end{gather}
Letting  
\begin{equation*}
    w= \uh-u
\end{equation*}
denote the difference between the two solutions, a short calculation shows that $w$ satisfies
\begin{equation} \label{weqn}
B^0(t,\uh)\del{t}w+B^i(t,\uh)\nabla_i w = \frac{1}{t}\Bc(t,\uh)w
+ G,
\end{equation}
where
\begin{align*}
    G&=\bigl(B^0(t,u)-B^0(t,\uh)\bigr)\del{t}u+\bigl(B^i(t,u)-B^i(t,\uh)\bigr)\nabla_iu + \frac{1}{t}\bigl(\Bc(t,\uh)-\Bc(t,u)\bigr)u + F(t,\uh)-F(t,u)\notag \\
    &=\bigl(B^0(t,u)-B^0(t,\uh)\bigr)B^0(t,u)^{-1}\biggl(- B^i(t,u)\nabla_i u
    +\frac{1}{t}\Bc(t,u)u + F(t,u)\biggr)\notag \\
    &\qquad +\bigl(B^i(t,u)-B^i(t,\uh)\bigr)\nabla_iu + \frac{1}{t}\bigl(\Bc(t,\uh)-\Bc(t,u)\bigr)u + F(t,\uh)-F(t,u).
\end{align*}
Observe here that $G$ vanishes if $w=0$, i.e., $\hat u=u$.

Next, we derive from the evolution equation \eqref{weqn} the $L^2$-energy identity
\begin{equation} \label{wenergyA}
\frac{1}{2}\del{t}\nnorm{w}^2 = -\frac{1}{|t|}\ip{w}{\Bc(t,\uh)w}+
\frac{1}{2}\ip{w}{\Div\! B(t,\uh,\nabla\uh) w} + \ip{w}{G}
\end{equation}
where 
\begin{equation*}
    \nnorm{w}^2 = \ip{w}{B^0(t,\uh)w}
\end{equation*}
is the energy norm.
But we have, from the bounds \eqref{uhbndA}, and the fact that coefficients $B^0$ and $B^i$ satisfy the assumptions of Definition \ref{def:CansymmhypFuchssystems}, see in particular \eqref{B0sym} and \eqref{divBbnd.1} with $\mu=0$ and $p=1$, that
\begin{equation*}
|\ip{w}{\Div\! B(t,\uh,\nabla\uh) w}| \leq \theta \norm{w}^2_{L^2} 
+ |t|^{-1}{\beta} \norm{w}_{L^2}^2,
\end{equation*}
\begin{equation} \label{wnormequiv}
    \frac{1}{\gamma_1}\norm{w}^2_{L^2(\Sigma)} \leq 
    \nnorm{w}^2 \leq \gamma_2\norm{w}^2_{L^2(\Sigma)}
\end{equation}
and
\begin{equation*}
    \ip{w}{\Bc(t,\uh)w}\leq -\eta \nnorm{w}^2.
\end{equation*}
With the help of these inequalities and the energy identity
\eqref{wenergyA}, we obtain the
energy estimate 
\begin{equation} \label{wenergyB}
\del{t}\nnorm{w}^2 \geq \frac{2\eta-\gamma_1\beta}{|t|} \nnorm{w}^2
 - \theta\nnorm{w}^2+ \ip{w}{G}.
\end{equation}

Next, from the bounds \eqref{supboundA}, \eqref{ut2bnd} and \eqref{uhbndA}, 
it is not difficult to verify from the assumptions from Definition \ref{def:CansymmhypFuchssystems} on the coefficients $B^0$, $B^i$, $\Bc$ and $F$, and H\"{o}lder's and Sobolev's inequalities that we can estimate $G$ by
\begin{equation*}
    \norm{G}_{L^2} \lesssim \frac{1}{|t|^{1-\lambda}}\norm{w}_{L^2}
\end{equation*}
uniformly for all $t\in [T_0,0)$. Observe that it is crucial for this argument here to exploit the positive decay for $u$ implied by \eqref{ut2bnd} in the limit $t\nearrow 0$ and the fact that $\lambda$ is assumed to be positive.
Using this, we get from  \eqref{wenergyB} and the Cauchy-Schwartz inequality that
\begin{equation*}
\del{t}\nnorm{w}^2 \geq \frac{2\eta-\gamma_1\beta}{|t|} \nnorm{w}^2
-\frac{C}{|t|^{1-\lambda}}\nnorm{w}^2
\end{equation*}
for some constant $C>0$. From this differential inequality, we then
conclude, via  an application of Gr\"{o}nwall's inequality, that
\begin{equation*}
\norm{w(t)}_{L^2(\Sigma)} \lesssim \biggl(\frac{|t_*|}{|t|}
\biggr)^{2\eta-\gamma_1 \beta}\norm{w(t_*)}_{L^2(\Sigma)} , \quad T_0 \leq t < t_*.
\end{equation*}
But by H\"{o}lder's
inequality we have that $\norm{w(t_*)}_{L^2(\Sigma)}\lesssim \norm{w(t_*)}_{L^\infty(\Sigma)}$, and so we see from
\eqref{dtu_nbnd} and \eqref{wnormequiv} that
\begin{equation*}
\lim_{t_*\nearrow 0} \biggl(\frac{|t_*|}{|t|}
\biggr)^{2\eta-\gamma_1 \beta}\norm{w(t_*)}_{L^2(\Sigma)} = 0
 \end{equation*}
since $2\eta-\gamma_1 \beta>0$ by assumption.
 Thus we conclude that $w=0$
 in $[T_0,0)\times \Sigma$, which establishes uniqueness.
\end{proof}

\subsubsection{Completion of the proof of  Theorem \ref{thm:mainresult} }

We are now ready to complete the proof of Theorem \ref{thm:mainresult}.
By Lemma \ref{lem:trafoproof}, we know that the symmetric hyperbolic Fuchsian system \eqref{symivp.1}, which satisfies the hypotheses of Theorem \ref{thm:mainresult}, can be transformed into a canonical symmetric hyperbolic Fuchsian system with $\lambda=\alpha>0$. The existence and uniqueness of solutions to the singular initial value problem for this canonical Fuchsian system is guaranteed by Proposition \ref{prop:canonicalSIVP}. Transforming back from the canonical Fuchsian system to the starting one, see Lemma  \ref{lem:trafoproof}, it is then not difficult to verify from the properties of the canonical solution, as determined by Proposition \ref{prop:canonicalSIVP},  that the corresponding solution to the singular initial value problem for the original Fuchsian system satisfies all the properties as stated in Theorem \ref{thm:mainresult}.

\section{Euler equations on Kasner backgrounds}
\label{sec:applications}

In this section, we state and prove a precise version of Theorem \ref{thm:informal}. This precise version is separated into two parts.
The first part concerns the existence and uniqueness of solution to the singular initial value problem for the Euler equations (i.e. the forward problem) and statement of the result is given below in \Theoremref{thm:fluidresult}.
The second part, which is presented below in \Theoremref{thm:fluidresult2}, expresses the nonlinear stability of solutions from \Theoremref{thm:fluidresult} under perturbations of the initial data at $t=T_0$ sufficiently close to $0$ (i.e. the backward problem).

In the following, we find it advantageous to change from the coordinates $(\tilde t,\tilde x,\tilde y,\tilde z)$  in \eqref{Kasner-k-original} to a new set of coordinates $(t,x,y,z)$ defined by
\begin{equation}
\label{eq:KasnerCoordTrafo}
\tilde t=\frac{4}{{K}^2+3} (-t)^{\frac{{K}^2+3}4},\quad 
\tilde x=\left(\frac{{K}^2+3}{4}\right)^{\frac{{K}^2-1}{{K}^2+3}} x,\quad
\tilde y=\left(\frac{{K}^2+3}{4}\right)^{\frac{2(1-{K})}{{K}^2+3}} y,\quad
\tilde z=\left(\frac{{K}^2+3}{4}\right)^{\frac{2(1+{K})}{{K}^2+3}} z.
\end{equation}
A short calculation then shows that, in terms of these new coordinates, the Kasner metric \eqref{Kasner-k-original} takes the form 
\begin{equation}
\label{Kasner-k}
g = (-t)^{\frac{{K}^2-1}{2}} \big( - d t\otimes dt + dx\otimes dx \big) + (-t)^{1-{K}} dy\otimes dy +(-t)^{1+{K}} dz\otimes dz.
\end{equation}
As a consequence of \eqref{eq:Kasnerrel}, the Kasner exponents $p_1$, $p_2$ and $p_3$ can be expressed in terms of the single parameter $K\in\R$ by
\begin{equation}
\label{eq:Kasnerexpo}
 p_1 =({K}^2-1)/({K}^2+3), \quad
 p_2 =2(1-{K})/({K}^2+3),\quad
 p_3 =2(1+{K})/({K}^2+3),
\end{equation}
which is often referred to as the \emph{asymptotic velocity}.
Notice that, in agreement with the conventions in \Sectionref{sec:SingSymHypSyst}, the time $t$ is negative now. We will always assume that $x$, $y$ and $z$ are periodic over the domain $[0,2\pi)$, 
and rather than using $c_s^2$ to parameterise the square of the sound speed, we will instead use the constant
\begin{equation}
  \label{eq:gammadef}
  \gamma=c_s^2+1
\end{equation}
as is common in the mathematical cosmology literature. With this definition, the equation of state \eqref{eq:EulerEOS2} then reads
\begin{equation}
  \label{eq:EulerEOS}
  P=(\gamma-1)\rho.
\end{equation}
We  will also interpret the fluid vector field $V=V^\alpha\del{\alpha}$ as a
 time-dependent section of the trivial vector bundle
 \[\Vbb=\Tbb^3\times \Rbb^4\]
 that sits over the spatial manifold $\Sigma=\Tbb^3$ by
 using the using the representation $V=V^\alpha\del{\alpha}$ in terms
of the coordinate frame $\{\partial_t, \partial_x, \partial_y,  \partial_z\}$ to identify $V=V^\alpha\del{\alpha}$ with its components $V^0$, $V^1$, $V^2$, $V^3$.
Furthermore, on the spatial manifold $\Sigma=\Tbb^3$, we employ the flat metric $dx\otimes dx+dy\otimes dy+dz\otimes dz$, which allows us to identify the associated covariant derivative, see \Sectionref{sec:SingSymHypSyst}, with the spatial partial derivatives
$\del{x},\del{y},\del{z}$, while on $\Vbb$, we use the metric $h$ defined by
\begin{equation}
\label{eq:EulerBundleMetric}
h(V_1,V_2) = \delta_{\alpha\beta}V_1^\alpha V_2^\beta.
\end{equation}

\subsection{Singular initial value problem} 
\label{sec:SIVPEuler}

\subsubsection{Statement of the theorem}

We are now in the position to formulate our first main result concerning relativistic fluids on Kasner spacetimes. This theorem establishes the existence of solutions to the relativistic Euler equations near the Kasner singularities with the asymptotics stated in Theorem \ref{thm:informal}.

\begin{thm}[Singular initial value problem for fluids on Kasner spacetimes]
  \label{thm:fluidresult}
Suppose
\begin{equation}
    \label{eq:gammarestr}
    {K}\in [0,1),\quad 2>\gamma>\frac{{K}^2+2 {K}+5}{{K}^2+3},
  \end{equation}
and set
 \begin{equation}
    \label{eq:GammaDef}
    \Gamma_1=\frac {3\gamma-2-{K}^2(2-\gamma)}4,\quad
    \Gamma_2=\frac {3\gamma-5+2{K}+{K}^2(\gamma-1)}4,\quad
    \Gamma_3=\frac {3\gamma-5-2{K}+{K}^2(\gamma-1)}4.
  \end{equation}
Furthermore, suppose $k\in \Zbb_{\geq 3}$,
$\ell \in \Zbb_{\geq 1}$ satisfies
$\ell>\Gamma_1/q$ for
$q=\min\{1-\Gamma_1,2\Gamma_3\}$, and set
$\epsilon=\min\{\Gamma_1+\Gamma_3,1,\ell (1-\Gamma_1), 2\ell \Gamma_3\}-\Gamma_1$.

\medskip

\noindent Then for each choice of 
 {asymptotic data} $v_*=(v_*^0,\ldots,v_*^3)^{\tr}\in H^{k+\ell}(\Sigma)$ where $v_*^0>0$, there exists, for $T_0<0$ close enough to zero, a solution 
\begin{equation}
\label{eq:EulerregNN}
V \in C^0_b\bigl([T_0,0),H^{k}(\Sigma)\bigr)\cap C^1\bigl([T_0,0),H^{k-1}(\Sigma)\bigr)
\subset C^1\bigl([T_0,0)\times \Sigma,\Vbb\bigr)
\end{equation}
of the Euler equations \eqref{eq:AAA1} with equation of state \eqref{eq:EulerEOS} satisfying
the decay estimate
\begin{align}
  &\norm{(-t)^{-\Gamma_1}V^0(t)-v_*^0}_{H^k(\Sigma)}
  +\norm{(-t)^{-2\Gamma_1}V^1(t)-v_*^1}_{H^k(\Sigma)}
  \notag \\
  &\hspace{2.0cm}+\norm{(-t)^{-2\Gamma_2}V^2(t)-v_*^2}_{H^k(\Sigma)} +\norm{(-t)^{-2\Gamma_3}V^3(t)-v_*^3}_{H^k(\Sigma)}\lesssim |t|^{\epsilon}
  \label{eq:Eulerdecay}
\end{align}
for all $t\in [T_0,0)$.

\medskip

\noindent Furthermore, if $\tilde{V} \in C^1\bigl([T_0,0)\times\Sigma, \Vbb\bigr)$ is {any} classical solution of the Euler equations \eqref{eq:AAA1} with equation of state \eqref{eq:EulerEOS} for which
\begin{equation} \label{L^infty-propertyEulerAppl}
\sup_{t\in [T_0,0)}\max\Bigl\{\bigl\| |t|^{-\mu}T^{-1}(V(t)-\Vt(t))\bigr\|_{L^\infty(\Sigma)},\bigl\| |t|^{-\mu}T^{-1}(\nabla V(t)-\nabla\Vt(t))\bigr\|_{L^\infty(\Sigma)}\Bigr\} \lesssim 1
\end{equation}
for some $\mu>\Gamma_1$ where
\begin{equation}
 \label{eq:EulerResc0}
  T:=\mathrm{diag}\bigl(|t|^{\Gamma_1},|t|^{\Gamma_1}, |t|^{\Gamma_2}, |t|^{\Gamma_3}\bigr),
\end{equation}
then $V=\tilde V$ in $[T_0,0)\times \Sigma$.
\end{thm}

The proof of this theorem is presented in \Sectionref{sec:proofEulerSIVP}, but before we proceed with the proof, we first make a number of remarks regarding this theorem. We begin with noting that the existence result contained in Theorem \ref{thm:fluidresult} does not rely on any symmetry or isotropy assumptions.
This should be contrasted with previous results from
the literature on relativistic fluids near Kasner singularities. For example, the existence results for the Einstein-Euler equations from
\cite{beyer2017} assumes $\Tbb^2$-symmetry while those from \cite{anguige1999,tod1999,tod1999a} rely on an isotropy assumption. 

We further
observe that \Theoremref{thm:fluidresult}, by itself,  does not characterise the asymptotics of \emph{generic} solutions. After all, this theorem is concerned with a family of solutions with the particular leading-order behaviour given by \eqref{eq:Eulerdecay}, albeit with the full expected number of free data functions.
With that said, a significant result of this paper is that we are able to show in \Theoremref{thm:fluidresult2}  that the solutions from  \Theoremref{thm:fluidresult} are stable under nonlinear perturbations.

Next, we claim that the condition ${K}\in [0,1)$ of \Theoremref{thm:fluidresult} is  not much of a restriction at all. 
To see why, given an {arbitrary}
$K\in\R\setminus \{\pm 1\}$, we can employ a  diffeomorphism to map the Kasner metric \eqref{Kasner-k} to another Kasner metric \eqref{Kasner-k} with $K=\Kt$ where $\Kt \in [0,1)$ as follows. First, if ${K}<0$, we can use the diffeomorphism $(t,x,y,z)\mapsto (t,x,z,y)$ to map the Kasner metric \eqref{Kasner-k} to another Kasner metric of the same form with $\tilde{K}=-{K}>0$. Therefore, without loss of generality, we can restrict our attention to the parameter values $K\ge 0$. We further note that $K=1$ is excluded by the second inequality in \eqref{eq:gammarestr}, and therefore, we can assume that $K\neq 1$. Now, if $1<K \leq 3$, we can apply the transformation \eqref{eq:KasnerCoordTrafo} to bring the Kasner metric \eqref{Kasner-k} into the form \eqref{Kasner-k-original} with exponents given by \eqref{eq:Kasnerexpo}. Then by applying the transformation $(\tilde t,\tilde x,\tilde y,\tilde z)\mapsto (\tilde t,\tilde y,\tilde x,\tilde z)$ and following it with the inverse transformation \eqref{eq:KasnerCoordTrafo}, we obtain a Kasner metric of the form \eqref{Kasner-k} with the new value $\tilde{K}=(3-{K})/(1+{K})\in [0,1)$ up to some irrelevant additional constant factors for the spatial components. 
On the other hand if ${K}>3$, we can first apply to the Kasner metric  \eqref{Kasner-k}  the transformation \eqref{eq:KasnerCoordTrafo}. We then follow it with the transformation  $(\tilde t,\tilde x,\tilde y,\tilde z)\mapsto (\tilde t,\tilde y,\tilde z,\tilde x)$ and finish with the inverse transformation of \eqref{eq:KasnerCoordTrafo} to yield a new Kasner metric \eqref{Kasner-k}  with parameter value $\tilde{K}=({K}-3)/(1+{K})\in (0,1)$ for the Kasner metric in \eqref{Kasner-k}.

For the subsequent analysis, it is important to note
that the assumption \eqref{eq:Kasnerexpo} from the statement of Theorem \ref{thm:fluidresult} implies that the Kasner exponents satisfy $p_3\ge p_2\ge p_1$
while it is clear from \eqref{eq:Kasnerexpo} and the assumption $K\in [0,1)$ that $p_1<0$.
Due to \eqref{eq:Kasnerexpo}, \eqref{eq:gammadef} and  \eqref{eq:GammaDef}, we can
express the constants $\Gamma_1$, $\Gamma_2$ and $\Gamma_3$ as
\begin{equation}
\label{eq:GammaaltN}
\Gamma_1=\frac{c_s^2-p_1}{1-p_1},\quad \Gamma_2=\frac{c_s^2-p_2}{1-p_1} \AND \Gamma_3=\frac{c_s^2-p_3}{1-p_1},
\end{equation}
respectively. From these expressions and the inequalities $p_3\geq p_2\geq p_1$ and $p_1<0$, it is clear that the constants $\Gamma_1$, $\Gamma_2$ and $\Gamma_3$ are ordered by $\Gamma_1\ge \Gamma_2\ge \Gamma_3$. Now, while the restriction $\gamma<2$ in \eqref{eq:gammarestr} implies that $\Gamma_1<1$, we note that $\Gamma_3>0$ is a consequence of the lower bound in \eqref{eq:gammarestr}. Combining these observations, we deduce that the constants $\Gamma_1$, $\Gamma_2$ and $\Gamma_3$ satisfy 
\begin{equation}
  \label{eq:Gammarestr}
  1>\Gamma_1\ge \Gamma_2\ge \Gamma_3>0
\end{equation}
as a consequence of \eqref{eq:gammarestr}.

The asymptotic data $v_*=(v_*^0,\ldots,v_*^3)^{\tr}$ in Theorem~\ref{thm:fluidresult} determines via the decay
estimate \eqref{eq:Eulerdecay} the leading order behavior of the solution given by $V_0=\bigl(v_*^0 (-t)^{\Gamma_1}, v_*^1 (-t)^{2\Gamma_1}, v_*^2 (-t)^{2\Gamma_2}, v_*^3 (-t)^{2\Gamma_3}\bigr)^{\tr}$. As we shall see below, the proof of Theorem~\ref{thm:fluidresult} is more complicated than simply applying \Theoremref{thm:mainresult} directly to the equation satisfied by the remainder $V-V_0$. The main reason for this is that the leading order term is, in general, not accurate enough and it must be improved in an iterative process until it is accurate enough to apply \Theoremref{thm:mainresult} to the remainder. It is interesting that the more the exponents $\Gamma_1$, $\Gamma_2$ and $\Gamma_3$ in \eqref{eq:GammaDef} differ, the more steps are needed in the iterative process to improve the leading order term. As a consequence, the analysis becomes simpler as the the background spacetime becomes more isotropic. This observation is particularly relevant in the context of Remark~\ref{rem:KSF} below.
Moreover, since we lose an order of differentiability in each step of the iterative process, the differentiability of the solutions is tied to the isotropy of the background Kasner spacetime with decreasing isotropy leading to an increase in the differentiability requirements as measured by the integer $\ell$ in \Theoremref{thm:fluidresult}.

Additionally, we note that in the coordinates of the metric \eqref{Kasner-k}, the restriction that $v^0_*$ is positive means that the fluid is time-oriented \emph{towards} the big bang time $t=0$. In cosmology, where the time orientation is usually chosen such that the big bang singularity represents the past, one mostly cares about fluids with the opposite time orientation. This is easy reconciled by noting that \eqref{eq:AAA1} is invariant under the transformation $V^\alpha\mapsto -V^\alpha$. In particular, every past directed solution asserted by \Theoremref{thm:fluidresult} can be transformed into a future directed one and vice versa, and so the choice of orientation is immaterial.

As another remark, we emphasize that $T_0$ is in general expected to
depend on the choice of asymptotic data given that some fluid solutions are expected to break down earlier than others in virtue of shock formation etc.

We close the discussion of Theorem \ref{thm:fluidresult} with a remark about the quantities $\epsilon$ and $\ell$ that appear in its statement. First, we note that $\epsilon$ is positive as a consequence of the assumption $\ell>\Gamma_1/q$ from Theorem \ref{thm:fluidresult} and the inequality \eqref{eq:Gammarestr} satisfied by the constants $\Gamma_1$, $\Gamma_2$ and $\Gamma_3$.
We further observe that the value of $\epsilon$ in Theorem~\ref{thm:fluidresult} depends on the choice of $\ell$, which determines the the order of differentiability of the solutions. As the following calculations show, the largest possible value  for $\epsilon$ is $\epsilon_0=\min\{\Gamma_3,1-\Gamma_1\}$ and it can be achieved choosing $\ell\geq \Gamma_1/q+1$:
  \begin{description}
  \item [Case $\Gamma_1+2\Gamma_3\le 1$] In this case, we have $q=2\Gamma_3$. Since $\Gamma_1+\Gamma_3<\Gamma_1+2\Gamma_3\le 1$, we get $\epsilon=\min\{\Gamma_1+\Gamma_3, 2\ell \Gamma_3\}-\Gamma_1$. The largest value $\epsilon_0=\Gamma_3$ is therefore obtained for $\ell$ satisfying
\[\ell\ge \frac{\Gamma_1+\Gamma_3}{2\Gamma_3}=\frac{\Gamma_1}q+\frac 12.\]
\item [Case $\Gamma_1+\Gamma_3\le 1\le \Gamma_1+2\Gamma_3$] 
In this case, we have $q=1-\Gamma_1$ and $\epsilon=\min\{\Gamma_1+\Gamma_3, \ell (1-\Gamma_1)\}-\Gamma_1$. 
The largest value $\epsilon_0=\Gamma_3$ is therefore obtained for $\ell$ satisfying
\[\ell\ge \frac{\Gamma_1+\Gamma_3}{1-\Gamma_1}=\frac{\Gamma_1}q+\frac{\Gamma_3}{1-\Gamma_1}.\]
Hence, the condition $\ell\ge \frac{\Gamma_1}q+1$ is sufficient.
\item [Case $1\le\Gamma_1+\Gamma_3$] In this case, we have $q=1-\Gamma_1$ and $\epsilon=\min\{1, \ell (1-\Gamma_1)\}-\Gamma_1$. 
The largest value $\epsilon_0=1-\Gamma_1$ is therefore obtained for $\ell$ satisfying
\[\ell\ge \frac{1}{1-\Gamma_1}=\frac{\Gamma_1}q+1.\]
\end{description}

\begin{rem}
\label{rem:KSF}
For future applications, we make here the important observation that \Theoremref{thm:fluidresult} and \Theoremref{thm:fluidresult2} continue to apply to the larger class of \emph{Kasner-scalar field spacetimes} \cite{ames2019, rodnianski2014,rodnianski2018} within which \eqref{Kasner-k-original} -- \eqref{eq:Kasnerrel} is a special case.
As we have noted already in the introduction, the Kasner-scalar field spacetimes provide important singularity models for solutions to the Einstein equations coupled to matter fields some of which are known to have stable big-bang singularities \cite{rodnianski2014,rodnianski2018}.

By definition, the Kasner-scalar field spacetimes are spatially homogeneous solutions of the Einstein-scalar field equations (minimally coupled, zero potential) where the scalar field is $\phi=A\log t+B$ for constants $A\in [-\sqrt{2/3},\sqrt{2/3}]$ and $B\in\R$, and the metric is of the same form as \eqref{Kasner-k-original}, but with exponents 
  \begin{equation}
  \label{eq:Kasnerrelsf}
  \sum_{i=1}^3 p_i=1,\quad \sum_{i=1}^3 p^2_i=1-A^2,
\end{equation}
instead of \eqref{eq:Kasnerrel}.
The parameter $B$ is non-dynamical and can be assumed to be zero without loss of generality. In the case $A=0$, we obtain the \emph{vacuum} Kasner solutions with the exponents satisfying \eqref{eq:Kasnerrel}. The solution given by the \emph{maximal scalar field strength} $A=\pm \sqrt{2/3}$ is isotropic, i.e., $p_1=p_2=p_3=1/3$, and therefore agrees with a spatially flat Friedmann-Robertson-Walker model.

We claim that all of the results of this paper can be easily generalised to hold for the whole family of Kasner-scalar field spacetimes. The main reason for the validity of this generalisation is that the Kasner-scalar field spacetimes become less anisotropic the larger the parameter $A$, and consequently, the vacuum case $A=0$ is the most technically challenging case. Because of this property, we assert that our results, which are written for $A=0$, imply that analogous results continue to hold for $A\neq 0$. 
To see why this is, consider an arbitrary Kasner-scalar field spacetime with $A\in [-\sqrt{2/3},\sqrt{2/3}]$. Then it is straightforward to find a coordinate transformation that brings the metric \eqref{Kasner-k-original} given by Kasner exponents satisfying \eqref{eq:Kasnerrelsf}  to the form
\begin{equation}
\label{Kasner-ksf}
g = (-t)^{\frac{{K}^2-1}{2}} \big( - d t\otimes dt + dx\otimes dx \big) + (-t)^{1-{L}} dy\otimes dy +(-t)^{1+{L}} dz\otimes dz
\end{equation}
where
\begin{equation}
\label{eq:KasnerexpoKL}
 p_1 =({K}^2-1)/({K}^2+3), \quad
 p_2 =2(1-{L})/({K}^2+3),\quad
 p_3 =2(1+{L})/({K}^2+3),
\end{equation}
which we observe is  very similar to \eqref{Kasner-k} and \eqref{eq:Kasnerexpo}. We also notice that \eqref{eq:Kasnerrelsf} is equivalent to
\begin{equation}
  \label{eq:LSFL}
  K^2=L^2+\frac{A^2}{8}(3+K^2)^2.
\end{equation}
It follows in the vacuum case $A=0$ that $L=\pm K$ and \eqref{Kasner-ksf} reduces to \eqref{Kasner-k} as expected. Given any $A$ with $0<A^2\le 2/3$, the restriction $L^2\ge 0$ implies the following restriction for $K^2$ in \eqref{eq:LSFL}:
\begin{equation}
  \label{eq:LSFKrest}
  K^2\in \Bigl[\frac{4-3A^2-2\sqrt2\sqrt{2-3A^2}}{A^2}, \frac{4-3A^2+2\sqrt2\sqrt{2-3A^2}}{A^2}\Bigr],
\end{equation}
which degenerates to no restriction in the case $A=0$.
Given all this, we now define
\begin{equation}
    \label{eq:GammaDefsf}
    \Gamma_1=\frac {3\gamma-2-{K}^2(2-\gamma)}4,\quad
    \Gamma_2=\frac {3\gamma-5+2{L}+{K}^2(\gamma-1)}4,\quad
    \Gamma_3=\frac {3\gamma-5-2{L}+{K}^2(\gamma-1)}4,
  \end{equation}
  which reduces to \eqref{eq:GammaDef} for $A=0$. Observe that \eqref{eq:GammaaltN} holds for these quantities $\Gamma_1$, $\Gamma_2$, $\Gamma_3$ in \eqref{eq:GammaDefsf} for \emph{all} $A$ with $0\le A^2\le 2/3$ as a consequence of \eqref{eq:KasnerexpoKL}.

  Anticipating parts of the proofs of \Theoremref{thm:fluidresult} and \Theoremref{thm:fluidresult2}, it is a remarkable fact that the Euler equations take the same form \eqref{eq:EulerResc} -- \eqref{eq:GEulerPol}, whether $A=0$ or not, as long as the fluid exponents $\Gamma_1$, $\Gamma_2$ and $\Gamma_3$ are defined by \eqref{eq:GammaDefsf} as opposed to by \eqref{eq:GammaDef}. It is a further remarkable fact of the proofs that condition \eqref{eq:gammarestr} never needs to be invoked directly; all arguments in the proofs rely on condition \eqref{eq:Gammarestr} instead (which is equivalent to \eqref{eq:stablefluidrestr} as a consequence of \eqref{eq:GammaaltN} irrespective of the value of $A$). We conclude from this that all the results about the dynamics of fluids given in \Theoremref{thm:fluidresult} and \Theoremref{thm:fluidresult2} are valid  for fluids on Kasner-scalar field spacetimes for \emph{any} $A\in [-\sqrt{2/3},\sqrt{2/3}]$ provided the results are expressed in terms of the fluid exponents in \eqref{eq:GammaDefsf} and provided \eqref{eq:Gammarestr} holds (while \eqref{eq:gammarestr} can be ignored). We remark that also the expression \eqref{eq:fluidorthonormal} then turns out to hold when the fluid vector $V^\alpha$ is expressed in terms of the orthonormal frame \eqref{eq:KasnerONF} irrespective of the value of $A$. 

  In order to illustrate  this further, let us consider the isotropic case $A=\pm \sqrt{2/3}$ as an example. In this case, \eqref{eq:LSFKrest} implies that $K^2$ has to take the value $3$ and \eqref{eq:LSFL} implies that $L=0$.
\eqref{eq:GammaDefsf} then yields that $\Gamma_1=\Gamma_2=\Gamma_3=(3\gamma-4)/2$. Our discussion therefore implies that all our results about fluids in \Theoremref{thm:fluidresult} (and similarly in \Theoremref{thm:fluidresult2}) hold for arbitrary equation of state parameters $\gamma$ with $2>\gamma>4/3$. This is an interesting outcome since it shows that  radiation fluids given by $\gamma=4/3$ correspond to the \emph{borderline case of stability} for $A=\pm\sqrt{2/3}$ (while they are in the \emph{unstable regime} for $A=0$).
\end{rem}

\subsubsection{Proof of Theorem~\ref{thm:fluidresult}}
\label{sec:proofEulerSIVP}

The proof of Theorem \eqref{thm:fluidresult} involves four steps. The first step is, given a leading order term $U_*$, which can be thought of as an approximate solution, to transform the Euler equations \eqref{eq:AAA1}
into a symmetric hyperbolic Fuchsian system in accord with Definition \ref{def:symmhypFuchssystems}. The second step is to apply Theorem \ref{thm:mainresult} to obtain the existence and uniqueness of a solution to the Fuchsian system, which yields a corresponding solution to the Euler equations. This is a conditional existence result because it relies on the leading order term $U_*$ being sufficiently accurate. This leads to the third step where sufficiently accurate leading order terms are constructed. With a sufficiently accurate leading order term in hand, we are then able, in the final step, to obtain solutions to the Euler equations from the conditional existence result.

\bigskip
\noindent \underline{Step 1: Transformation to Fuchsian form}
\bigskip

\noindent We multiply \eqref{eq:AAA1}  through with $(-t)^{-({K}^2-1)/2-5\Gamma_1}  V_\lambda V^\lambda$
and express that system as a partial differential equation for the unknown
\begin{equation}
  \label{eq:EulerResc}
  U=T^{-1} V,
\end{equation}
where $T$ is given by \eqref{eq:EulerResc0}.
In this way, the Euler equations take the form
\begin{equation}
    \label{eq:Eulereqssymmhyp2N}
    B^0(U) \partial_t U+ B^1(U) \partial_x U+ B^2(t,U) \partial_y U+ B^3(t,U) \partial_z U=\frac 1t \Bc(U) U+G(t,U),
\end{equation}
where
\begin{align}
\label{eq:B0EulerPol}
B^0(v)&=
  \left(
\begin{array}{cccc}
 {P_0} v^0 & -{Q_0} v^1 & -{Q_0} v^2 & -{Q_0} v^3 \\
 -{Q_0} v^1 & {Q_1} v^0 & v^0 v^1 v^2 (3 \gamma -2) & v^0 v^1 v^3 (3 \gamma -2) \\
 -{Q_0} v^2 & v^0 v^1 v^2 (3 \gamma -2) & {Q_2} v^0 & v^0 v^2 v^3 (3 \gamma -2) \\
 -{Q_0} v^3 & v^0 v^1 v^3 (3 \gamma -2) & v^0 v^2 v^3 (3 \gamma -2) & {Q_3} v^0 \\
\end{array}
\right),\\
\label{eq:B1EulerPol}
B^1(v)&=\left(
\begin{array}{cccc}
 {Q_0} v^1 & -{Q_1} v^0 & -v^0 v^1 v^2 (3 \gamma -2) & -v^0 v^1 v^3 (3 \gamma -2) \\
 -{Q_1} v^0 & {P_1} v^1 & {Q_1} v^2 & {Q_1} v^3 \\
 -v^0 v^1 v^2 (3 \gamma -2) & {Q_1} v^2 & {Q_2} v^1 & v^1 v^2 v^3 (3 \gamma -2) \\
 -v^0 v^1 v^3 (3 \gamma -2) & {Q_1} v^3 & v^1 v^2 v^3 (3 \gamma -2) & {Q_3} v^1 \\
\end{array}
\right),\\
B^2(t,v)&=(-t)^{\Gamma_2-\Gamma_1}\left(
\begin{array}{cccc}
 {Q_0} v^2 & -v^0 v^1 v^2 (3 \gamma -2) & -{Q_2} v^0 & -v^0 v^2 v^3 (3 \gamma -2) \\
 -v^0 v^1 v^2 (3 \gamma -2) & {Q_1} v^2 & {Q_2} v^1 & v^1 v^2 v^3 (3 \gamma -2) \\
 -{Q_2} v^0 & {Q_2} v^1 & {P_2} v^2 & {Q_2} v^3 \\
 -v^0 v^2 v^3 (3 \gamma -2) & v^1 v^2 v^3 (3 \gamma -2) & {Q_2} v^3 & {Q_3} v^2 \\
\end{array}
\right),
\end{align}
\begin{align}
\label{eq:B3EulerPol}
B^3(t,v)&=(-t)^{\Gamma_3-\Gamma_1}\left(
\begin{array}{cccc}
 {Q_0} v^3 & -v^0 v^1 v^3 (3 \gamma -2) & -v^0 v^2 v^3 (3 \gamma -2) & -{Q_3} v^0 \\
 -v^0 v^1 v^3 (3 \gamma -2) & {Q_1} v^3 & v^1 v^2 v^3 (3 \gamma -2) & {Q_3} v^1 \\
 -v^0 v^2 v^3 (3 \gamma -2) & v^1 v^2 v^3 (3 \gamma -2) & {Q_2} v^3 & {Q_3} v^2 \\
 -{Q_3} v^0 & {Q_3} v^1 & {Q_3} v^2 & {P_3} v^3 \\
\end{array}
\right),
\end{align}
\begin{align}
\label{eq:BcEulerPol}
\Bc(v)&=r v^0\mathrm{diag}\,\left(0,\Gamma_1,\Gamma_2,\Gamma_3\right)
\intertext{and}
\label{eq:GEulerPol}
G(t,v)&=-r \frac {{\Gamma_1} (v^1)^2+{\Gamma_2} (v^2)^2+{\Gamma_3} (v^3)^2}t\left(1,0,0,0\right)^{\tr}
\end{align}
with $t<0$ and $v\in\Vbb$\footnote{Recall that $\Vbb$ is the trivial bundle $\Tbb^3\times \Rbb^4$.}.
Here, $r$, $P_0$, \ldots, $P_3$ and $Q_0$, \ldots, $Q_3$ are all quadratic polynomials  in the components\footnote{Observe that strictly speaking $v$ is an element of the trivial bundle $\Vbb$. Thinking of $v$ as a vector in $\Rbb^4$ is justified by our conventions above.} $v\in \Rbb^4$ with $(t,x,y,z)$-independent coefficients, which only depend on the parameter $\gamma$, and not, for example, on $K$. These polynomials are all chosen to be strictly positive whenever they are evaluated at $v=(v^0,0,0,0)^{\tr}$ with $v^0>0$, and in particular,
\begin{align}
  \label{eq:Eulerpol1}
  r(v^0,0,0,0)&=Q_1(v^0,0,0,0)=Q_2(v^0,0,0,0)=Q_3(v^0,0,0,0)=(\gamma-1)(v^0)^2,\\
  \label{eq:Eulerpol2}
  P_0(v^0,v^1,v^2,v^3)&=(v^0)^2+3(\gamma-1)((v^1)^2+(v^2)^2+(v^3)^2)\quad
  \Longrightarrow\quad P_0(v^0,0,0,0)=(v^0)^2.
\end{align}

To proceed with the transformation to Fuchsian form, we introduce a leading order term $U_*$, to be specified\footnote{A sufficient list of precise assumptions for $U_*$ and the related quantity $\Uh_*$ defined in \eqref{eq:EulerLOTRem} is given in Proposition~\ref{prop:EulerProp2}.}, and we formulate the Euler equations in terms of the remainder 
\begin{equation}\label{remainder}
u=U-U_*.
\end{equation}
We will also find it convenient at times to work with
a rescaled version of $U_*$, denoted by $\Uh_*$, that is defined via
\begin{equation}
  \label{eq:EulerLOTRem}
  U_*= \Th^{-1} \Uh_*
\end{equation}
where
\begin{equation}\label{Thdef}
\Th=\Th(t):=\mathrm{diag}\,\left(1,(-t)^{-\Gamma_1}, (-t)^{-\Gamma_2}, (-t)^{-\Gamma_3}\right)
\end{equation}
and
\begin{equation*}
    \Uh_*=(\Uh_*^0,\Uh_*^1,\Uh_*^2,\Uh_*^3)^{\tr}.
\end{equation*} 
Now, a straightforward calculation shows that the Euler equations \eqref{eq:Eulereqssymmhyp2N}
can be expressed in terms of the remainder \eqref{remainder}
as
\begin{equation}
    \label{eq:Eulereqssymmhyp3N}
    B^0(U_*,u) \partial_t u+ B^1(U_*,u) \partial_x u+ B^2(t, U_*,u) \partial_y u
+ B^3(t, U_*,u) \partial_z u=\frac 1t \Bc(U_*,u) u+F(t,W_*,u)
\end{equation}
where 
\begin{equation}
 W_*= \bigl(\Uh_*,(-t)^{1-q}\del{t}U_*,\del{x}\Uh_*,\del{y}\Uh_*,\del{z}\Uh_*\bigr), \quad q\geq 0,
 \label{W*def}
 \end{equation}
and  by a slight abuse of notation, we have set
\begin{align}
B^0(\Zt,v)&=B^0(\Zt+v),\label{Eul-B0} \\ 
B^1(\Zt,v)&= B^1(\Zt+v), \label{Eul-B1}\\
B^2(t,\Zt,v)&= B^2(t,\Zt+v),\label{Eul-B2} \\  
B^3(t,\Zt,v)&= B^3(t, \Zt+v),\label{Eul-B3} \\
\Bc(\Zt,v)&=\Bc(\Zt+v), \label{Eul-B4}
\end{align}
for $t\in [T_0,0)$ and $(\Zt,v)\in Z_1\oplus\Vbb$ with $Z_1=\Sigma\times\Rbb^4$. In the following, we label  components of $\Zt\in Z_1$ by  $(\Zt^0,\Zt^1,\Zt^2,\Zt^3)$.
The map
 $F$ can be expressed as
 \begin{equation}
 F(t,Z,v) = (-t)^{\lambda}\Ft(t,Z)
 +(-t)^{\lambda}F_0(t,Z,v) \label{eq:Eulereqssymmhyp3N2}
 \end{equation}
where $[T_0,0)\times Z_2\oplus\Vbb$ and 
 $Z_2$ is the trivial bundle $Z_1 \oplus Z_1 \oplus Z_1 \oplus Z_1 \oplus Z_1$ over $\Sigma$. We will label the components of $Z\in Z_2$ by
 \begin{equation}
 (Z_0,Z_1,Z_{21},Z_{22},Z_{23}). \label{Eul-Z}
 \end{equation}
With this notation, $\Ft$ and $F_0$ are then given by
 \begin{align}
 \Ft(t,Z)&=-(-t)^{-\lambda}\biggl(B^0(\Th(t)^{-1}Z_0) (-t)^{q-1}Z_1
+B^1(\Th(t)^{-1}Z_0) \Th(t)^{-1}Z_{21}
\notag \\
&\hspace{2.5cm} +B^2(t,\Th(t)^{-1}Z_0)\Th(t)^{-1}Z_{22}  +B^3(t,\Th(t)^{-1}Z_0)\Th(t)^{-1} Z_{23} \notag \\
&\hspace{5.5cm} -\frac 1t \Bc(\Th(t)^{-1}Z_0) \Th(t)^{-1}Z_0-G(\Th(t)^{-1}Z_0)\biggr) \label{Eul-Ft}
 \intertext{and}
F_0(t,Z,v) &=(-t)^{-\lambda} \biggl( -\bigl(B^0(\Th(t)^{-1}Z_0+v)-B^0(\Th(t)^{-1}Z_0)\bigr) (-t)^{q-1}Z_1 \notag -\bigl(B^1(\Th(t)^{-1}Z_0+v) \notag \\
&\qquad -B^1(\Th(t)^{-1}Z_0)\bigr) \Th(t)^{-1}Z_{21}
-\bigl(B^2(t,\Th(t)^{-1}Z_0+v)-B^2(t,\Th(t)^{-1}Z_0)\bigr) \Th(t)^{-1}Z_{22} \notag \\
&\qquad  -\bigl(B^3(t,\Th(t)^{-1}Z_0+v)-B^3(t,\Th(t)^{-1}Z_0)\bigr) \Th(t)^{-1}Z_{23} 
+\frac 1t\bigl(\Bc(\Th(t)^{-1}Z_0+v) \notag \\
&\hspace{2.0cm}\qquad -\Bc(\Th(t)^{-1}Z_0)\bigr) \Th(t)^{-1}Z_0+\bigl(G(\Th(t)^{-1}Z_0+v)-G(\Th(t)^{-1}Z_0)\bigr) \biggr). \label{Eul-F0}
\end{align}

The point of expressing the Euler equations this way  is, as will be verified in the following lemma, that \eqref{eq:Eulereqssymmhyp3N} is now in a form to which Theorem \ref{thm:mainresult}
applies.
Before stating the lemma, we define a family of open and bounded 
sets $\tilde{\Zc}_{\rc,\Rc}$, $0<\rc < \Rc$,  in $Z_1$ by
\begin{equation} \label{Zctdef}
    \tilde{\Zc}_{\rc,\Rc} = \bigl\{ \Zt\in Z_1\,\bigl|\, \Zt^0>\rc,\; |\Zt|< \Rc\,\bigr\}.
\end{equation}

\begin{lemma} \label{lem:Euler1}
Suppose ${K}$ and $\gamma$ satisfy \eqref{eq:gammarestr}, $\Gamma_1$, $\Gamma_2$ and
$\Gamma_3$ are defined by \eqref{eq:GammaDef}, $q>0$, $\mu>\Gamma_1$, $R>0$, $0<\rc<\Rc$,  $\gamma_1=2/(\rc^3(\gamma-1))$,
$\gamma_2=2\Rc^3$, $0<p\leq\min\{1-\Gamma_1+\Gamma_3,\mu,q,\Gamma_3\}$,  $\beta=0$ and $-1<\lambda\leq \min\{\mu+q-1,\mu+\Gamma_3-1\}$.
Then for $T_0<0$ close enough to zero, there exists
a constant $\theta>0$ such that
\eqref{eq:Eulereqssymmhyp3N} is a symmetric hyperbolic Fuchsian system according to \Defref{def:symmhypFuchssystems} 
for the above choices of constants, the open bounded subsets
sets $\Zc_1=\tilde{\Zc}_{\rc,\Rc}$ and $\Zc_2=\Zc_1\oplus B_R(Z_1) \oplus B_R(Z_1) \oplus B_R(Z_1) \oplus B_R(Z_1)$ of of $Z_1$ and $Z_2$, respectively, and
the maps \eqref{Eul-B0}-\eqref{Eul-B4}, \eqref{eq:Eulereqssymmhyp3N2}-\eqref{Eul-F0},
$\Bt^0(t,\Zt)=B^0(\Zt,0)$, $B_0^i(t,\Zt,v)=t^{1-p}B^i(t,\Zt,v)$ and
$B_1^i=0$.
\end{lemma}

\begin{proof}
Since $\mu>\Gamma_1$, it is clear from the
inequality \eqref{eq:Gammarestr}, the definitions \eqref{eq:B0EulerPol}-\eqref{eq:BcEulerPol} and \eqref{Eul-B0}-\eqref{Eul-B4}, and 
\eqref{eq:Eulerpol1}-\eqref{eq:Eulerpol2} 
that, for any $0<\rc<\Rc$ and $R>0$, 
the map 
\begin{equation*}
    (t,\Zt,z) \mapsto B^0(t,\Zt,|t|^{\mu}z)
\end{equation*}
is in $C^1\bigl([T_0,0),C^\infty\bigl(\Zc_1
\oplus B_R(\Vbb),L(\Vbb)\bigr)\bigr)$
while the maps
\begin{align*}
    (t,\Zt,v) &\longmapsto t^{1-p} B^1(t,\Zt,|t|^{\mu}v), \\
    (t,\Zt,v) &\longmapsto  t^{1-p} B^2(t,\Zt,|t|^{\mu}v), \\
    (t,\Zt,v) &\longmapsto t^{1-p} B^3(t,\Zt,|t|^{\mu}v), \\
    (t,\Zt,z) &\longmapsto \Bc(\Zt,|t|^{\mu}v), 
\end{align*}
are in $C^0_b\bigl([T_0,0),C^\infty\bigl(\Zc_1
\oplus B_R(\Vbb),L(\Vbb)\bigr)\bigr)$ provided 
that
\begin{equation*}
0<p\leq 1-\Gamma_1+\Gamma_3.
\end{equation*}

Next, we observe from the definition \eqref{Zctdef} of the set $\Zc_1=\tilde{\Zc}_{\rc,\Rsc}$, the formulas
\eqref{eq:B0EulerPol}, \eqref{eq:Eulerpol1}-\eqref{eq:Eulerpol2} and \eqref{Eul-B0}, and the fact that $\mu>0$,
that, by choosing $T_0<0$ close enough to $0$,
we can ensure that $B^0(t,\Zt,v)$ satisfies
\begin{equation*}
    \frac{1}{\gamma_1}\id\leq B^0(\Zt,|t|^\mu v)
    \leq \gamma_2 \id
\end{equation*}
for all $(t,\Zt,v) \in [T_0,0)\times \tilde{\Zc}_{\rc,\Rsc} \oplus B_R(\Vbb)$
where $\id$ is the identity map on $\Vbb$, 
\begin{equation*}
    \gamma_1 = \frac{2}{\rc^3 (\gamma-1)}
    \AND \gamma_2 = 2\Rc^3.
\end{equation*}
Setting
\begin{equation*}
    \Bt^0(\Zt)= B^0(\Zt,0),
\end{equation*}
it is also not difficult to verify that
$\Bt^0\in C^\infty\bigl(\Zc_1,L(\Vbb)\bigr)$ and that
\begin{equation*}
    B^0(\Zt,|t|^\mu v)-\Bt^0(\Zt)=\Ord(v)
\end{equation*}
for all $(t,\Zt,v)\in  [T_0,0)\times\Zc_1\oplus B_R(\Vbb)$.
Turning to the map $\Ft$ defined by \eqref{Eul-Ft}, it is clear from this definition along
with the formulas \eqref{eq:B0EulerPol}-\eqref{eq:Eulerpol2}, \eqref{Thdef} and \eqref{Eul-Z} that the map
    $(t,Z) \longmapsto \Ft(t,Z)$
is in $C^0\bigl([T_0,0),C^\infty(\Zc_2,\Vbb)\bigr)$.
In anticipation of the discussion below, it is useful to note by \eqref{Eul-Ft} that
    $(t,Z) \longmapsto \Ft(t,Z)$
is in $C_b^0\bigl([T_0,0),C^\infty(\Zc,\Rbb^4)\bigr)$ for any 
\begin{equation}
  \label{eq:lambdatilde}
  \lambda\le \min\{q-1,\Gamma_3-1\}.
\end{equation}
Now, considering the map $F_0$ defined by \eqref{Eul-F0},
we first observe from  \eqref{eq:Gammarestr}, \eqref{eq:B0EulerPol}-\eqref{eq:Eulerpol2}, \eqref{Thdef}, \eqref{Eul-Z} and 
the assumption $\mu>\Gamma_1>0$
that the maps
\begin{align*}  
 (t,Z,v)&\longmapsto |t|^{-(\mu+q-1)} \bigl(B^0(\Th^{-1}(t)Z_0,|t|^\mu v)-B^0(\Th^{-1}(t)Z_0,0)\bigr)(-t)^{q-1}Z_1,\\
   (t,Z,v)&\longmapsto |t|^{-\mu} \bigl(B^1(\Th^{-1}(t)Z_0,|t|^\mu v)-B^1(\Th^{-1}(t)Z_0,0)\bigr)\Th^{-1}(t)Z_{21}, \\  
    (t,Z,v)&\longmapsto |t|^{-(\mu-\Gamma_1+\Gamma_2)} \bigl(B^2(t,\Th^{-1}(t)Z_0,|t|^\mu v)-B^2(t,\Th^{-1}(t)Z_0,0)\bigr)\Th^{-1}(t)Z_{22}, \\ 
    (t,Z,v)&\longmapsto |t|^{-(\mu-\Gamma_1+\Gamma_3)} \bigl(B^3(t,\Th^{-1}(t)Z_0,|t|^\mu v)-B^3(t,\Th^{-1}(t)Z_0,0)\bigr)\Th^{-1}(t)Z_{23}, \\
    (t,Z,v)&\longmapsto |t|^{-(\mu-1+\Gamma_3)} \frac{1}{t}\bigl(\Bc(\Th^{-1}(t)Z_0,|t|^\mu v)-\Bc(\Th^{-1}(t)Z_0,0)\bigr)\Th^{-1}(t)Z_0, \\
    (t,Z,v)&\longmapsto |t|^{-(\mu-1+\Gamma_3)} \bigl(G(\Th^{-1}(t)Z_0,|t|^\mu v)-G(\Th^{-1}(t)Z_0,0)\bigr),
\end{align*}
are all in
$C^0_b\bigl([T_0,0), C^\infty(\Zc_2\oplus B_R(\Vbb),\Vbb)\bigr)$. From this,
\eqref{eq:Gammarestr}, \eqref{Eul-F0} and
the assumption $\mu>\Gamma_1$, we see immediately
that 
\begin{equation*}
    (t,Z,v)\longmapsto F_0(t,Z,|t|^\mu v)
\end{equation*}
will be in $C^0_b\bigl([T_0,0), C^\infty(\Zc\oplus B_R(\Vbb),\Vbb)\bigr)$
provided that $\lambda$ satisfies
\begin{equation*}
 \lambda \leq \min\bigl\{\mu+q-1,\mu+\Gamma_3-1\bigr\}.   
\end{equation*}

To complete the proof of the lemma, we see from \eqref{divBdef} that the map $\Div\! B$
is given by
\begin{align*}
    \Div\!B(t,\Zt,\dot{\Zt},\Zt',v,v') =& D_{\Zt} B^0(\Zt,v)|t|^{q-1}\dot{\Zt}+ D_v B^0(\Zt,v)(B^0(\Zt,v))^{-1}\biggl[-B^i(t,\Zt,v)v'_i \\
    &\qquad+\frac{1}{t}\Bc(\Zt,v)v+F(t,Z,v)\biggr]+D_{\Zt}B^i(t,\Zt,v)\Zt'_i+D_{v}B^i(t,\Zt,v)v'_i.
\end{align*}
From similar considerations as above, it is not difficult to verify from the above, in particular \eqref{eq:lambdatilde}, that  
\begin{equation*}
\Div\!B(t,\Zt,\dot{\Zt},\Zt',|t|^\mu v,|t|^\mu v')=\Ord(\theta |t|^{-(1-p)})
\end{equation*}
for some $\theta >0$ provided that 
\begin{equation*}
p\leq \min\{q,\mu,\Gamma_3\}.
\end{equation*}
\end{proof}

\bigskip
\noindent \underline{Step 2: Conditional existence and uniqueness}
\bigskip

\noindent Given Lemma \ref{lem:Euler1}, we can now invoke \Theoremref{thm:mainresult}, while observing Remark \ref{rem:Fuchsian-SIVP}, to
derive precise conditions \eqref{Uh*-reg}-\eqref{eq:EulerFTildecond} for the leading order term $U_*$. The following proposition states that if these conditions are satisfied, then the singular initial value problem for the Euler equations has a solution. In Step~3  below, see Lemma~\ref{lem:iterateLOT}, we then construct leading order terms for the Euler equations which we show in Step~4 to be consistent with these conditions.

\begin{proposition}
  \label{prop:EulerProp2}
Let ${K}$ and $\gamma$ satisfy \eqref{eq:gammarestr}, $R>0$, $\Gamma_1$, $\Gamma_2$ and
$\Gamma_3$ be as defined by \eqref{eq:GammaDef}, and
$k\in \Zbb_{\geq 3}$. Suppose there
exist constants $\mu>\Gamma_1$, $q>0$, $\rc_0>0$ and $\lambda$ with
\begin{equation}
  \label{eq:EulerlambdaChoice}
  \mu-1<\lambda\le \min\{q,\Gamma_3\}+\mu-1,
\end{equation}
such that 
\begin{equation} \label{Uh*-reg}
\Uh_*=(\Uh^0_*,\Uh^1_*,\Uh^2_*,\Uh^3_*) \in  C^0_b([T_0,0), H^{k+1}(\Sigma,Z_1))\cap C^1([T_0,0), H^{k}(\Sigma,Z_1))
\end{equation}
satisfies
\begin{equation}  \label{Uh*-lowbnd}
\Uh_*^0\geq \rc_0 \quad \text{in $[T_0,0]\times \Sigma$}, \quad
t^{1-q}\del{t} U_*(t) \in  C^0_b\bigl([T_0,0), H^{k}(\Sigma,Z_1)\bigr)
\end{equation}
and
\begin{equation} \label{eq:EulerFTildecond}
\Ft(t,W_*(t)) \in  C^0_b\bigl([T_0,0), H^{k}(\Sigma,\Vbb)\bigr),
\end{equation}
where $U_*$, $W_*$ and $\Ft$ are as defined by \eqref{eq:EulerLOTRem}, \eqref{W*def} and \eqref{Eul-Ft}, respectively.
Then 
there exists a solution 
\begin{equation}
\label{eq:symhypSIVPregNN}
u \in C^0_b\bigl([T_0,0),H^{k}(\Sigma,\Vbb)\bigr)\cap C^1\bigl([T_0,0),H^{k-1}(\Sigma,\Vbb)\bigr)\subset C^1([T_0,0)\times\Sigma,\Vbb)
\end{equation}
of \eqref{eq:Eulereqssymmhyp3N}.
Moreover, this solution satisfies
\begin{equation} \label{L^infty-propertyEuler}
\sup_{t\in [T_0,0)}\max\Bigl\{\bigl\| |t|^{-\mu}u(t)\bigr\|_{L^\infty(\Sigma)},\bigl\| |t|^{-\mu}\nabla u(t)\bigr\|_{L^\infty(\Sigma)}\Bigr\} < R
\end{equation}
and 
\begin{equation}
\label{eq:solution1estimateboundednessNNNNN}
\norm{u(t)}_{H^{k}(\Sigma)}\lesssim
|t|^{\lambda+1}
\end{equation}
for all $t\in [T_0,0)$, and it is unique 
within the $C^1([T_0,0)\times \Sigma,\Vbb)$ class of solutions satisfying
\eqref{L^infty-propertyEuler}.
\end{proposition}

\begin{proof}
First we notice, since $\Gamma_3$, by \eqref{eq:Gammarestr}, and
$q$, by assumption, are both positive, that the
inequality for $\lambda$ in \eqref{eq:EulerlambdaChoice} is consistent. Moreover we know from \eqref{eq:Gammarestr} that $\Gamma_1>0$ and hence, $\mu>0$. It therefore follows that $\lambda$ satisfies the required inequality in \Lemref{lem:Euler1}.   
By \eqref{eq:Gammarestr} and \eqref{eq:EulerlambdaChoice}, we also have
$\min\{1-\Gamma_1+\Gamma_3,\mu,q,\Gamma_3,\lambda+1-\mu\}>0$.
So, we can choose $p$ to satisfy
\begin{equation}
  \label{eq:Eulerprestr}
0<p< \min\{1-\Gamma_1+\Gamma_3,\mu,q,\Gamma_3,\lambda+1-\mu\},
\end{equation}
which implies that $p$ is consistent with the inequality in \Lemref{lem:Euler1}.   
Consequently, we conclude by Lemma \ref{lem:Euler1} that the system \eqref{eq:Eulereqssymmhyp3N} is symmetric hyperbolic Fuchsian  for any choice of constants $\rc \in (0,\rc_0)$
and 
$\Rc>\rc_0$
provided that $T_0<0$ is chosen close enough to zero. 

Next, we set
\[\alpha = \frac{\lambda+1-p-\mu}{p}\]
in order to satisfy \eqref{eq:minimaldecayassumption}, and we note
that $\alpha>0$ due to our choice of $p$ above.
We see also by $\mu>\Gamma_3$, \eqref{eq:Gammarestr}, \eqref{eq:B0EulerPol}, \eqref{eq:BcEulerPol}, \eqref{eq:Eulerpol1} and \eqref{eq:Eulerpol2} that for every sufficiently small $\eta>0$, there exists a small $T_0<0$ so that \eqref{eq:evassumptionforwardN} holds. Because $\beta$ and $\mathtt{b}$ (defined in \eqref{eq:defb}) are zero as a consequence of \Lemref{lem:Euler1}, condition \eqref{eq:forwardposN} is also fulfilled. Finally we notice that the smallness condition \eqref{eq:forwardsmallnessTh} can be verified for any $\delta>0$ by choosing a sufficiently small $|T_0|$, as a consequence of \eqref{eq:EulerFTildecond} for any $p$ with \eqref{eq:Eulerprestr}.

Having verified that all the conditions for Theorem~\ref{thm:mainresult} are met, we conclude, with the help of Remark \ref{rem:Fuchsian-SIVP}, that 
there exists a unique solution 
\begin{equation*} 
u \in C^0_b\bigl([T_0,0),H^{k}(\Sigma,\Vbb)\bigr)\cap C^1\bigl([T_0,0),H^{k-1}(\Sigma,\Vbb)\bigr)\subset C^1([T_0,0)\times\Sigma,\Vbb)
\end{equation*}
of \eqref{eq:Eulereqssymmhyp3N} that is bounded by
\begin{equation} \label{Eul-bnd1}
\norm{|t|^{-(\lambda+1-p)}u(t)}_{H^{k}(\Sigma)}\leq C \int_t^0 |s|^{p-1}\norm{\Ft(s,W_*(s))}_{H^k(\Sigma)}\, ds, \quad T_0\leq t<0,
\end{equation}
and
\begin{equation}\label{Eul-bnd2}
\sup_{t\in [T_0,0)}\max\Bigl\{\bigl\| |t|^{-\mu}u(t)\bigr\|_{L^\infty(\Sigma)},\bigl\| |t|^{-\mu}\nabla u(t)\bigr\|_{L^\infty(\Sigma)}\Bigr\} < R.
\end{equation}
Moreover, this is the unique solution of \eqref{eq:Eulereqssymmhyp3N}
within the class $C^1([T_0,0)\times \Sigma,\Vbb)$ satisfying
\eqref{Eul-bnd2}.
Finally, given \eqref{eq:EulerFTildecond},
we note from \eqref{Eul-bnd1} that
\begin{equation*} 
\norm{|t|^{-(\lambda+1-p)}u(t)}_{H^{k}(\Sigma)}\lesssim |t|^p
\end{equation*}
for all $t\in [T_0,0)$. Multiplying
both sides of this inequality by $|t|^{\lambda+1-p}$
yields the estimate \eqref{eq:solution1estimateboundednessNNNNN}.
This completes the proof of \Propref{prop:EulerProp2}.
\end{proof}

\bigskip
\noindent \underline{Step 3: Constructing accurate leading order terms}
\bigskip

\noindent We now turn to constructing leading order terms $U_*$ that satisfy the conditions \eqref{Uh*-reg}-\eqref{eq:EulerFTildecond}. We do this by constructing a finite sequence $U_n$, $n\geq 0$, of solutions of the following approximate equations:
\begin{equation} \label{eq:EulereqsDecA}
    B^0_n \partial_t U_{n+1}
  +B^i_n \partial_i U_{n}
  +B^0_n \Th^{-1}\partial_t\Th (U_{n+1}-U_{n})
-\frac 1t \Bc_n U_{n}-G_n =0
\end{equation}
where 
\begin{gather}
  B^0_n=B^0(U_n),\quad B^1_n=B^1(U_n),\quad B^2_n=B^2(t,U_n), \label{BnGndef1} \\
  B^3_n=B^3(t,U_n), \quad
  \Bc_n=\Bc(U_n),\quad G_n=G(U_n), \label{BnGndef2}
\end{gather}
and $\Th$ is defined above by \eqref{Thdef}. The leading order term $U_*$ will then be identified with $U_n$ for some sufficiently large $n$. In order to explain how this is done, we note first that, by setting 
\begin{equation}
  \label{eq:askjdj3}
  W_n=\Th U_{n},
\end{equation}
the approximate equations \eqref{eq:EulereqsDecA} can be expressed more compactly as
\begin{equation} \label{eq:EulereqsDec}
  \partial_t W_{n+1}
  + \Bh^i_n\partial_i W_{n}
-\Gh_n=0
\end{equation}
where
\begin{equation}
      \Bh^i_n=\Th (B^0_n)^{-1} B^i_n\Th^{-1} \AND
  \Gh_n=\frac 1t\Th (B^0_n)^{-1}\Bc_n U_n+\partial_t\hat T U_n +\Th (B^0_n)^{-1} G_n. \label{BhnGhndef}
\end{equation}
A straightforward but lengthy calculation shows that
\begin{align}
  \label{eq:BhnGhndef1}
  \Bh^1_n&=\frac {|t|^{-\Gamma_1}}{S_0 {W_n^0}}\\
&\begin{pmatrix}
 0 & - {S_1} |t|^{2 {\Gamma_1}} {W_n^0} &  |t|^{2 ({\Gamma_1}+{\Gamma_2})} {W_n^0} {W_n^1} {W_n^2} \gamma  &  |t|^{2 ({\Gamma_1}+{\Gamma_3})} {W_n^0} {W_n^1} {W_n^3}  \gamma  \\
 - {S_0} {W_n^0} & 2  |t|^{2 {\Gamma_1}} {Y_1} {W_n^1} & {P_1}  |t|^{2 {\Gamma_2}} {W_n^2} & {P_1}  |t|^{2 {\Gamma_3}} {W_n^3}  \\
 0 & -2 {Q_1}  |t|^{2 {\Gamma_1}} {W_n^2} (\gamma -1) & {P_2}  |t|^{2 {\Gamma_1}} {W_n^1} & 2  |t|^{2 ({\Gamma_1}+{\Gamma_3})} {W_n^1} {W_n^2} {W_n^3}  (\gamma -1) \\
 0 & -2 {Q_1}  |t|^{2 {\Gamma_1}} {W_n^3}  (\gamma -1) & 2  |t|^{2 ({\Gamma_1}+{\Gamma_2})} {W_n^1} {W_n^2} {W_n^3}  (\gamma -1) & {P_3}  |t|^{2 {\Gamma_1}} {W_n^1}
\end{pmatrix}\notag\\
\Bh^2_n&=\frac {|t|^{-\Gamma_1}}{S_0 {W_n^0}}\\
&\begin{pmatrix}
 0 &  |t|^{2 ({\Gamma_1}+{\Gamma_2})} {W_n^0} {W_n^1} {W_n^2} \gamma  & - {S_2} |t|^{2 {\Gamma_2}} {W_n^0} &  |t|^{2 ({\Gamma_2}+{\Gamma_3})} {W_n^0} {W_n^2} {W_n^3}  \gamma  \\
 0 & {P_1}  |t|^{2 {\Gamma_2}} {W_n^2} & -2 {Q_2}  |t|^{2 {\Gamma_2}} {W_n^1} (\gamma -1) & 2  |t|^{2 ({\Gamma_2}+{\Gamma_3})} {W_n^1} {W_n^2} {W_n^3}  (\gamma -1) \\
 - {S_0} {W_n^0} & {P_2}  |t|^{2 {\Gamma_1}} {W_n^1} & 2  |t|^{2 {\Gamma_2}} {Y_2} {W_n^2} & {P_2}  |t|^{2 {\Gamma_3}} {W_n^3}  \\
 0 & 2  |t|^{2 ({\Gamma_1}+{\Gamma_2})} {W_n^1} {W_n^2} {W_n^3}  (\gamma -1) & -2 {Q_2}  |t|^{2 {\Gamma_2}} {W_n^3}  (\gamma -1) & {P_3}  |t|^{2 {\Gamma_2}} {W_n^2} 
\end{pmatrix}\notag\\
\Bh^3_n&=\frac {|t|^{-\Gamma_1}}{S_0 {W_n^0}}\\
&\begin{pmatrix}
 0 &  |t|^{2 ({\Gamma_1}+{\Gamma_3})} {W_n^0} {W_n^1} {W_n^3}  \gamma  &  |t|^{2 ({\Gamma_2}+{\Gamma_3})} {W_n^0} {W_n^2} {W_n^3}  \gamma  & - {S_3} |t|^{2 {\Gamma_3}} {W_n^0} \\
 0 & {P_1}  |t|^{2 {\Gamma_3}} {W_n^3}  & 2  |t|^{2 ({\Gamma_2}+{\Gamma_3})} {W_n^1} {W_n^2} {W_n^3}  (\gamma -1) & -2 {Q_3}  |t|^{2 {\Gamma_3}} {W_n^1} (\gamma -1) \\
 0 & 2  |t|^{2 ({\Gamma_1}+{\Gamma_3})} {W_n^1} {W_n^2} {W_n^3}  (\gamma -1) & {P_2}  |t|^{2 {\Gamma_3}} {W_n^3}  & -2 {Q_3}  |t|^{2 {\Gamma_3}} {W_n^2} (\gamma -1) \\
 - {S_0} {W_n^0} & {P_3}  |t|^{2 {\Gamma_1}} {W_n^1} & {P_3}  |t|^{2 {\Gamma_2}} {W_n^2} & 2  |t|^{2 {\Gamma_3}} {Y_3} {W_n^3}
\end{pmatrix}\notag
\end{align}
and
\begin{equation}
  \label{eq:BhnGhndefLast}
  \begin{split}
\Gh_n=&\frac {{\Gamma_1} (W_n^1)^2 |t|^{2
    {\Gamma_1}}+{\Gamma_2} (W_n^2)^2 |t|^{2 {\Gamma_2}}+{\Gamma_3}
  (W_n^3)^2 |t|^{2 {\Gamma_3}}}{t S_0}\\
&\qquad \qquad \qquad \qquad\times\Bigl(\gamma {W_n^0},2(\gamma-1){W_n^1}, 2(\gamma-1){W_n^2} 2(\gamma-1){W_n^3}\Bigr)^{\tr},
\end{split}
\end{equation}
where $Y_1$, $Y_2$, $Y_3$, $Q_1$, $Q_2$, $Q_3$, $P_1$, $P_2$, $P_3$, $S_0$ $S_1$, $S_2$ and $S_3$ 
are quadratic polynomials\footnote{Despite having the same labels, the polynomials here do not agree with the polynomials used to express \eqref{eq:B0EulerPol} -- \eqref{eq:GEulerPol}.}
in the components of $U_n=\Th^{-1}W_n$
with $(t,x,y,z)$-independent coefficients that only depend on $\gamma$. All these polynomials have been chosen to be strictly positive when evaluated at $U_n=(U_n^0,0,0,0)$ with $U_n^0>0$.

In the following lemma, we fix the starting value $U_0$ of the iteration sequence and solve the approximate equations
\eqref{eq:EulereqsDec} to obtain $U_1$, $U_2$, \ldots,
$U_\ell$ for any integer $\ell \geq 0$.

\begin{lemma}
  \label{lem:iterateLOT}
  Suppose ${K}$ and $\gamma$ satisfy \eqref{eq:gammarestr}, $\ell, k\in \Zbb_{\geq 1}$, $q\leq \min\{1-\Gamma_1,2\Gamma_3\}$,
  $v_*=(v_*^0, v_*^1, v_*^2, v_*^3)^{\tr}\in H^{k+\ell}(\Sigma,\Vbb)$ with $v_*^0>0$ in $\Sigma$, and let
\[U_0=(v_*^0, v_*^1 (-t)^{\Gamma_1}, v_*^2 (-t)^{\Gamma_2}, v_*^3 (-t)^{\Gamma_3})^{\tr}.\]
Then for $T_0<0$ close enough to zero,
there exists solutions
$U_0,U_1,\ldots,U_{\ell}$ of \eqref{eq:EulereqsDecA} with the properties:
\begin{align}
 &U_n \in C^0_b\bigl([T_0,0),H^{k+\ell-n}(\Sigma,\Vbb)\bigr),
 \label{induct.i} \\
 &(-t)^{-q}(\Th U_{n}-v_*) \in C^0_b([T_0,0), H^{k+\ell-n}(\Sigma,\Vbb)), \label{induct.ii} \\
 &(-t)^{-n q}\Th (U_{n}-U_{n-1})
 \in C^0_b([T_0,0), H^{k+\ell-n}(\Sigma,\Vbb)) \label{induct.iii}
 \intertext{and}
 &(-t)^{1-q}\partial_t (\Th U_{n}) \in C^0_b([T_0,0), H^{k+\ell-n}(\Sigma,\Vbb)). \label{induct.iv}
 \end{align}
\end{lemma}

\begin{proof}
  We proceed by induction.
  
\medskip
 
\noindent\underline{Base case:}
Given $v_*=(v_*^0, v_*^1, v_*^2, v_*^3)^{\tr}\in H^{k+\ell}(\Sigma,\Vbb)$, we note by \eqref{Thdef} that
\begin{equation*}
U_0=(v_*^0, v_*^1 (-t)^{\Gamma_1}, v_*^2 (-t)^{\Gamma_2}, v_*^3 (-t)^{\Gamma_3})^{\tr}
\end{equation*}
satisfies
\begin{equation}
  \label{eq:Eulerw0}
    \Th U_0-v_*=0.
\end{equation}
From these expressions, it is then clear that $U_0$
satisfies all asserted properties \eqref{induct.i}, \eqref{induct.ii} and \eqref{induct.iv} for $n=0$. Thanks to \eqref{eq:Eulerw0}, we notice that the base case of \eqref{induct.iii} for $n=1$ is equivalent to \eqref{induct.ii} for $n=1$.

\bigskip

\noindent\underline{Induction hypothesis:} 
We now assume that \eqref{induct.i}-\eqref{induct.iv}
hold for $n=0,1,\ldots,m-1$ for some integer $m$
satisfying $1\leq m\leq \ell-1$. For use
below, we observe by setting
 \begin{equation}
    \label{eq:EulerDefuhn}
    w_n=W_n-v_*=\Th U_n-v_* 
  \end{equation}
that we can write \eqref{eq:EulereqsDec} as
  \begin{equation}
    \label{w-evol-A}
    \partial_t w_{n+1}= \Fh_n :=   -\Bh^i_n\partial_i (v_*+w_{n})
    +\Gh_n,
  \end{equation}
and moreover, that $w_{n+1}-w_n$ satisfies
\begin{equation}
    \label{w-evol-B}
    \del{t}(w_{n+1}-w_n)= \fh_n :=  -\Bh^i_n\del{i}(w_n-w_{n-1})-(\Bh^i_n-\Bh^i_{n-1})\del{i}(v_*+w_{n-1})+\Gh_n-\Gh_{n-1}.
\end{equation}
\bigskip
  
\noindent \underline{Induction step:}
With the help of the calculus inequalities, in particular, the Sobolev, Product and Moser calculus inequalities, see  \cite[Ch.~13, \S 2 \&  3]{taylor2011},
we see from \eqref{eq:Gammarestr},  \eqref{eq:BhnGhndef1}--\eqref{eq:BhnGhndef1}, \eqref{w-evol-A}--\eqref{w-evol-B}, and the induction hypothesis that $\Fh_{m-1}$ can
be estimated by
\begin{align}
    \norm{\Fh_{m-1}}_{H^{k+\ell-{m}}(\Sigma)} & 
    \leq \norm{\Bh_{m-1}}_{H^{k+\ell-m}(\Sigma)}
    \norm{D(v_*+w_{m-1})}_{H^{k+\ell-m}(\Sigma)}
    + \norm{\Gh_{m-1}}_{H^{k+\ell-m}(\Sigma)} \notag \\
    &\leq  C\bigl(\norm{v_*+w_{m-1})}_{H^{k+\ell-m}(\Sigma)}\bigr)\bigl(|t|^{-\Gamma_1}\norm{v_*+w_{m-1}}_{H^{k+\ell-(m-1)}(\Sigma)}+|t|^{2\Gamma_3-1}\bigr) \notag \\
    &\lesssim |t|^{-\Gamma_1}+|t|^{2\Gamma_3-1}, 
    \label{Fh-est}
\end{align}
while a similar argument shows that and $\fh_{m-1}$ can
be estimated by
\begin{align}
\norm{\fh_{m-1}}_{H^{k+\ell-{m}}(\Sigma)} &\leq 
     \norm{\Bh_{m-1}}_{H^{k+\ell-m}(\Sigma)}
    \norm{D(w_{m-1}-w_{m-2})}_{H^{k+\ell-m}(\Sigma)}
   \notag \\
   &\qquad + \norm{\Bh_{m-1}-\Bh_{m-2}}_{H^{k+\ell-m}(\Sigma)}
    \norm{D(v_*+w_{m-1})}_{H^{k+\ell-m}(\Sigma)} \notag \\
   &\hspace{5.5cm}  + \norm{\Gh_{m-1}-\Gh_{m-2}}_{H^{k+\ell-m}(\Sigma)}\notag \\
 &
  \lesssim  \bigl(|t|^{-\Gamma_1}+|t|^{2\Gamma_3-1}\bigr) \norm{w_{m-1}-w_{m-2}}_{H^{k+\ell-(m-1)}(\Sigma)} \notag\\
  &\lesssim |t|^{(m-1)q-\Gamma_1}+|t|^{(m-1)q+2\Gamma_3-1}.\label{fh-est}
\end{align}
Using \eqref{Fh-est}, it follows, since
$\Gamma_1<1$ and $2\Gamma_3-1>-1$ by assumption, that we can solve \eqref{w-evol-A} for $n=m-1$ by setting
\begin{equation}
    w_{m}(t) = -\int_{t}^0 \Fh_{m-1}(s)\, ds \label{wm-sol}
\end{equation}
where $w_{m}$ and $\del{t}w_m$ are bounded by
\begin{align} 
    \norm{w_m(t)}_{H^{k+\ell-m}(\Sigma)} & \lesssim 
    |t|^{1-\Gamma_1}+|t|^{2\Gamma_3} \label{wm-bnd}
    \intertext{and}
    \norm{\del{t}w_m(t)}_{H^{k+\ell-m}(\Sigma)}  &\lesssim
    |t|^{-\Gamma_1}+|t|^{2\Gamma_3-1} \label{dtwm-bnd}
\end{align}
for all $t\in [T_0,0)$. But $q\le \min\{1-\Gamma_1,2\Gamma_3\}$, 
and so, we see immediately from \eqref{Thdef}, \eqref{wm-sol}, \eqref{wm-bnd} and \eqref{dtwm-bnd} that $U_m=\Th^{-1}(w_m+v_*)$ satisfies \eqref{induct.i}, \eqref{induct.ii}
and \eqref{induct.iv} for $m=n$. Similar arguments using \eqref{w-evol-B}, \eqref{fh-est}, and
the fact that $\lim_{t\nearrow 0}(w_m(t)-w_{m-1}(t))=0$ by 
\eqref{wm-bnd} and the induction hypothesis shows that $w_m(t)-w_{m-1}(t)$
can be expressed as
\begin{equation}
  \label{eq:EulerSIVPImpr}
    w_{m}(t)-w_{m-1}(t) = -\int_{t}^0 \fh_{m-1}(s)\, ds,
\end{equation}
and is, in turn, bounded by 
\begin{equation*}
\norm{w_m(t)-w_{m-1}(t)}_{H^{k+\ell-m}(\Sigma)}  \lesssim 
    |t|^{(m-1)q+1-\Gamma_1}+|t|^{(m-1)q+2\Gamma_3}.
\end{equation*}
But as $q\le \min\{1-\Gamma_1,2\Gamma_3\}$ and $\Th(U_m-U_{m-1})=w_m-w_{m-1}$, we see that \eqref{induct.iii} is satisfied for $n=m$, which completes the proof.
\end{proof}

\bigskip
\noindent \underline{Step 4: Existence and uniqueness}
\bigskip

\noindent Now, we are in the position to complete the proof of \Theoremref{thm:fluidresult}. Assume that $k\in \Zbb_{\geq 3}$,
${K}$ and $\gamma$ satisfy \eqref{eq:gammarestr}, $\Gamma_1$, $\Gamma_2$ and
$\Gamma_3$ are defined by \eqref{eq:GammaDef},
\begin{equation}
  \label{eq:Eulerchooseq}
  q=\min\{1-\Gamma_1,2\Gamma_3\},
\end{equation}
$\mu>\Gamma_1$, and
$v_*=(v_*^0, v_*^1, v_*^2,v_*^3)^{\tr}\in H^{k+\ell}(\Sigma,\Vbb)$ with $v_*^0>0$ in $\Sigma$.
Then for $\ell \in \Zbb_{\geq 2}$ and $T_0<0$ chosen sufficiently small, let $U_0,U_1,\ldots,U_{\ell-1}$
be the sequence of solutions to \eqref{eq:EulereqsDecA} from Lemma \ref{lem:iterateLOT}.  

Setting
\begin{equation} \label{U*set}
    U_*=U_{\ell-1}
\end{equation}
and letting, as above, $\Uh_*$ and $W_*$ be defined
in terms of $U_*$ by \eqref{eq:EulerLOTRem} and
\eqref{W*def}, respectively, we see by \eqref{Eul-Ft} and
\eqref{BnGndef1}--\eqref{BnGndef2} that
\begin{equation*}
    \Ft(t,W_*)= (-t)^{-\lambda}\Bigl(B^0_{\ell-1} \partial_t U_{\ell-1}
  +B^i_{\ell-1} \partial_i U_{\ell-1}-\frac 1t \Bc_{\ell-1} U_{\ell-1}-G_{\ell-1}\Bigr).
\end{equation*}
Using the definitions \eqref{eq:askjdj3} and \eqref{BhnGhndef} together with the evolution equation \eqref{eq:EulereqsDec}, it then is not difficult to verify via a short calculation that we can express $\Ft(t,W_*)$
as
\begin{equation*}
  \Ft(t,W_*)=   (-t)^{-\lambda}B^0_{\ell-1} \Th^{-1}\Bigl(
 \Bh^i_{\ell-1}\partial_i  (W_{{\ell-1}}-W_{\ell-2})
  +(\Bh^i_{\ell-1}- \Bh^i_{\ell-2})\partial_i W_{\ell-2}
+\Gh_{\ell-2}
-\Gh_{\ell-1}\Bigr).
\end{equation*}
Using this expression, we can, with the help of
\eqref{Thdef}, \eqref{eq:B0EulerPol}, \eqref{BnGndef1}, \eqref{eq:BhnGhndef1} -- \eqref{eq:BhnGhndefLast}, and
the calculus inequalities 
(i.e. Sobolev, Product and Moser calculus inequalities, see  \cite[Ch.~13, \S 2 \&  3]{taylor2011}), estimate
$\Ft(t,W_*)$ by
\begin{align*}
    \norm{\Ft(t,W_*(t))}_{H^k(\Sigma)}\leq& C\Bigl(\norm{v_*}_{H^k(\Sigma)},\norm{W_{\ell-1}(t)}_{H^{k+1}(\Sigma)},\norm{W_{\ell-2}(t)}_{H^{k+2}(\Sigma)}\Bigr) \notag\\
    &\times\bigl(|t|^{-\Gamma_1-\lambda}+|t|^{2\Gamma_3-1-\lambda}\bigr) \Bigl(\norm{W_{\ell-1}(t)-W_{\ell-2}(t)}_{H^{k+1}(\Sigma)}\Bigr).
\end{align*}
By Lemma \ref{lem:iterateLOT} and the above estimate, we have
\begin{equation*}
     \norm{\Ft (t,W_*(t))}_{H^k(\Sigma)}\lesssim 
    |t|^{-\Gamma_1-\lambda+(\ell-1)q}+|t|^{2\Gamma_3-1-\lambda+(\ell-1)q}.
\end{equation*}
Condition \eqref{eq:EulerFTildecond} is therefore satisfied provided
\[\lambda\le \min\{1-\Gamma_1,2\Gamma_3\}-1+(\ell-1)q.\]
Given our choice of $q$ above and the restriction for $\lambda$ and $\mu$ in \Propref{prop:EulerProp2}, we find 
\begin{equation}
  \label{eq:EulerFinalLambda}
  \mu-1<\lambda\le \min\{q+\mu,\Gamma_3+\mu,\ell q\}-1,
\end{equation}
and therefore
\begin{equation}
  \label{eq:Eulerellconstr}
  \ell>\frac{\mu}q.
\end{equation}

Due to the uniform bound \eqref{eq:EulerFTildecond}, the fact that $U_*=U_{\ell-1}$ where
$U_{\ell-1}$ satisfies \eqref{induct.ii} and \eqref{induct.iv}, and the definition $\Uh_*=\Th U_*$, we are now in a position, after choosing $T_0<0$ closer to zero if necessary, to apply Proposition \ref{prop:EulerProp2}. 
For any choice of constants $R>0$, this yields the existence
of a solution 
\begin{equation*}
u \in C^0_b\bigl([T_0,0),H^{k}(\Sigma,\Vbb)\bigr)\cap C^1\bigl([T_0,0),H^{k-1}(\Sigma,\Vbb)\bigr)\subset C^1([T_0,0)\times\Sigma,\Vbb)
\end{equation*}
of \eqref{eq:Eulereqssymmhyp3N} that is bounded by 
\begin{equation} \label{u-decay}
\norm{u(t)}_{H^k(\Sigma)} \lesssim |t|^{\lambda+1},  \quad T_0\leq t < 0,
\end{equation}
and is unique within the
$C^1([T_0,0)\times\Sigma,\Vbb)$ class of solutions satisfying \eqref{L^infty-propertyEuler}.

By \eqref{eq:EulerResc0}, \eqref{eq:EulerResc}, $U=U_*+u=U_{\ell-1}+u$ and the fact that $\lambda+1>\Gamma_1\ge\Gamma_2\ge\Gamma_3$ as a consequence of $\mu>\Gamma_1$, and \eqref{eq:Gammarestr} and \eqref {eq:EulerFinalLambda}, it follows that
\begin{equation*}
V=T U_{\ell-1}+T u \in C^0\bigl([T_0,0),H^k(\Sigma,\Vbb)\bigr)\cap C^1\bigl([T_0,0),H^{k-1}(\Sigma,\Vbb)\bigr)
\end{equation*}
determines a solution of the relativistic
Euler equations \eqref{eq:AAA1}. Now, according to \eqref{eq:askjdj3} and \eqref{eq:EulerDefuhn}, we have
\[\Th T^{-1} V-v_*=w_{\ell-1}+\Th u,\]
which we observe, in turn, implies, via \eqref{eq:Gammarestr} and \eqref{Thdef} and the estimates \eqref{induct.ii} and \eqref{u-decay}, 
that
\begin{equation*}
    \norm{\Th T^{-1} V-v_*}_{H^k(\Sigma)}\lesssim
    |t|^{\epsilon}
\end{equation*}
where
\begin{equation}
  \label{eq:Eulerepsgen}
     \epsilon =\min\{\lambda+1-\Gamma_1,q\}=\min\{\lambda+1-\Gamma_1,1-\Gamma_1,2\Gamma_3\},
   \end{equation}
see \eqref{eq:Eulerchooseq}.

Now, by choosing $\ell$ sufficiently large, \eqref{eq:Eulerellconstr} and \eqref{eq:EulerFinalLambda} allow us to choose $\mu$ and $\lambda$ arbitrarily large. This makes sense since the larger we choose $\ell$, the more accurately we expect the leading order term $U_{\ell-1}$ to describe the actual solution $U$ and therefore the faster the remainder $u$ to decay according to \eqref{u-decay}. However, whatever large values we choose for $\ell$, $\mu$ and $\lambda$, this does not allow us to obtain a better decay exponent $\epsilon$ than $\min\{1-\Gamma_1,2\Gamma_3\}$ in accordance with \eqref{eq:Eulerepsgen}. Moreover, the larger we choose $\ell$ the more orders of differentiability we lose, and the larger we choose $\mu$ the weaker the uniqueness statement \eqref{L^infty-propertyEuler} becomes. In order to optimise differentiability and uniqueness, we obtain the result stated in \Theoremref{thm:fluidresult} as follows. First, we select any $\ell$ with
\begin{equation}
  \label{eq:Eulerellcond}
  \ell>\frac{\Gamma_1}q
\end{equation}
with $q$ given by \eqref{eq:Eulerchooseq}, and we make the specific choice
\[\lambda=\min\{1,\Gamma_1+\Gamma_3,\ell q\}-1=:\mu_0-1.\]
If we then select any $\mu$ with
\[\mu\in (\Gamma_1,\mu_0),\]
where we notice that $\mu_0>\Gamma_1$ by \eqref{eq:Gammarestr} and \eqref{eq:Eulerellcond}, we see that all the conditions $\mu>\Gamma_1$, \eqref{eq:EulerFinalLambda} and \eqref{eq:Eulerellconstr} are satisfied. Given \eqref{eq:Eulerepsgen} and our choice for $\lambda$, we find that
\begin{equation}
  \label{eq:Eulerepsgen2}
     \epsilon =\min\{\ell q-\Gamma_1,1-\Gamma_1,\Gamma_3\}.
   \end{equation}
The estimate holds for every $\ell$ consistent with \eqref{eq:Eulerellcond} and we have $\epsilon>0$. 
The largest value for $\epsilon$, which we can achieve by choosing $\ell$ sufficiently large, is therefore $\min\{1-\Gamma_1,\Gamma_3\}$. In particular, this establishes \eqref{eq:Eulerdecay}.

Having established the existence of this solution, which we can write in the form $U=U_{\ell-1}+u$, consider now any other classical solution $\Ut$ of \eqref{eq:Eulereqssymmhyp2N} such that $U$ and $\Ut$ satisfy \eqref{L^infty-propertyEulerAppl}, where we recall that $V$ and $\Vt$ are defined in terms of $U$ and $\Ut$ by  $U=T^{-1}V$ and $\Ut=T^{-1}\Vt$, respectively.
It follows that $u=U-U_{\ell-1}$ and $\tilde u=\tilde U-U_{\ell-1}$ are both solutions of the same symmetric hyperbolic Fuchsian system \eqref{eq:Eulereqssymmhyp3N} with $U_*=U_{\ell-1}$. According to the uniqueness result of \Propref{prop:EulerProp2}, it follows that $\ut=u$ and hence $U=\Ut$ if the bound \eqref{L^infty-propertyEuler} holds for $\ut$ for some $\mu\in (\Gamma_1,\mu_0)$. Expressing $\ut$ as
$\ut=u+(\Ut-U)$, it then follows from \eqref{L^infty-propertyEulerAppl}, the definitions  $U=T^{-1}V$ and $\Ut=T^{-1}\Vt$, and the bound \eqref{L^infty-propertyEuler} satisfied by $u$ (recall that $\lambda+1=\mu_0$) that $\ut$ satisfies the required bound for uniqueness. We therefore conclude that $u=\ut$ and the proof of Theorem \ref{thm:fluidresult} is complete.

\subsection{Stability}
\label{sec:EulerPert}
Having established the existence of a class of solutions of the Euler equations with the asymptotics given in \Theoremref{thm:fluidresult}, we now turn to establishing
the nonlinear stability of these solutions. The precise sense in which these solutions are stable is given below in \Theoremref{thm:fluidresult2}. The proof of this theorem is carried out in Section \ref{sec:Eulerproofstability}.

\begin{thm}[Stability of fluids on Kasner spacetimes]
  \label{thm:fluidresult2}
  Suppose ${K}$ and $\gamma$ satisfy \eqref{eq:gammarestr}, $\Gamma_1$, $\Gamma_2$ and
$\Gamma_3$ are defined by \eqref{eq:GammaDef}, and $k$, $\ell$ and $\ell_0$ are integers satisfying $k\ge 3$,  $\ell>\Gamma_1/q$ and 
\begin{equation}
  \label{eq:EulerCPDefell0}
  \ell_0\ge \ell+2,\quad \ell_0>\ell+\frac 32+\frac{\Gamma_1}{2\Gamma_3},
\end{equation}
respectively, where $q$ is defined in \eqref{eq:Eulerchooseq}. 
Further, for given
 {asymptotic data}
$v_*=(v_*^0,\ldots,v_*^3)^{\tr}\in H^{k+\ell_0+\ell}(\Sigma)$ with $v_*^0>0$,
let $V_S$ be the solution of the singular initial value problem of the Euler equations \eqref{eq:AAA1} with equation of state \eqref{eq:EulerEOS} defined on the Kasner spacetime given by \eqref{Kasner-k} on a time interval $[\Tt_0,0)$ asserted by \Theoremref{thm:fluidresult}.

\medskip

\noindent Then there exists $\delta>0$ such that for all sufficiently small $T_0\in [\Tt_0,0)$ and for any $V_0\in H^{k+\ell_0}(\Sigma)$ with $\norm{V_0}_{H^{k+\ell_0}(\Sigma)}\le \delta$, the Cauchy problem of
the Euler equations \eqref{eq:AAA1} with equation of state \eqref{eq:EulerEOS} defined on the Kasner spacetime given by \eqref{Kasner-k} with Cauchy data $V_{C,0}=V_S(T_0)+\Th^{-1}(T_0)T(T_0)V_0$ imposed at $t=T_0$
has a unique solution $V_C(t)=T(t) U_C(t)$ with
\begin{equation*}
U_C \in C_b^0\bigl([T_0,0),H^{k+\ell_0}(\Sigma)\bigr)\cap C^1\bigl([T_0,0),H^{k+\ell_0-1}(\Sigma)\bigr),
\end{equation*}
where $T(t)$ and $\Th(t)$ are defined in \eqref{eq:EulerResc0} and \eqref{Thdef}, respectively. Moreover:
\begin{enumerate}
\item The limits
$\lim_{t\nearrow 0} U_C^0$ in $H^{k+\ell_0-1}(\Sigma)$, $    \lim_{t\nearrow 0}|t|^{-\Gamma_1} U_C^1$ in $H^{k+\ell}(\Sigma)$, $    \lim_{t\nearrow 0}|t|^{-\Gamma_2} U_C^2$ in $H^{k+\ell}(\Sigma)$ and $    \lim_{t\nearrow 0}|t|^{-\Gamma_3} U_C^3$ in $H^{k+\ell_0-2}(\Sigma)$,
which we denote by
$W_C^0(0)$, $W_C^1(0)$, $W_C^2(0)$, and $W_C^3(0)$, respectively, exist, and for any $\sigma>0$, the estimates
\begin{gather}
  \label{eq:EulerPertSecLimit1}
  \norm{(-t)^{-\Gamma_1} V_C^0(t)-W_C^0(0)}_{H^{k+\ell_0-1}(\Sigma)}\lesssim |t|^{1-\Gamma_1+\Gamma_3}+|t|^{2(\Gamma_3-\sigma)}
  \intertext{and}
\label{eq:EulerPertSecLimit2}
\begin{split}
  &\norm{(-t)^{-2\Gamma_1} V_C^1(t)-W_C^1(0)}_{H^{k+\ell}(\Sigma)}
  +\norm{(-t)^{-2\Gamma_2} V_C^2(t)-W_C^2(0)}_{H^{k+\ell}(\Sigma)}\\
  &+\norm{(-t)^{-2\Gamma_3} V_C^3(t)-W_C^3(0)}_{H^{k+\ell_0-2}(\Sigma)}
  \lesssim |t|^{1-\Gamma_1}+|t|^{2(\Gamma_3-\sigma)-2\sigma}
\end{split}
\end{gather}
hold for all $t\in [T_0,0)$. 
\item The solution $V_C$ agrees with the solution $\Vt_S$ of the singular initial value problem of the Euler equations \eqref{eq:AAA1} with equation of state \eqref{eq:EulerEOS} asserted by \Theoremref{thm:fluidresult} for the asymptotic data $\tilde v_*=(W^0_C(0), W^1_C(0), W^2_C(0), W^3_C(0))^{\tr}$.
\item The asymptotic data $v_*$ and $\tilde v_*=(W^0_C(0), W^1_C(0), W^2_C(0), W^3_C(0))^{\tr}$  are close in the sense
  \begin{equation}
  \label{eq:EulerCPLimContUFull}
  \norm{\tilde v_*^0-v_*^0}_{H^{k+\ell_0-1}(\Sigma)}\lesssim \norm{V_0}_{H^{k+\ell_0}(\Sigma)}+|T_0|^{\epsilon}
\end{equation}
with $\epsilon$ defined in Theorem~\ref{thm:fluidresult}, and,
\begin{equation}
  \label{eq:stabilityclosenessFull}
  \begin{split}
  \norm{\tilde v_*^j-v_*^j}_{H^{k+\ell}(\Sigma)}\lesssim &\norm{V_0}_{H^{k+\ell_0}(\Sigma)}\\
  &+
  |T_0|^{1-\Gamma_1}+|T_0|^{2(\Gamma_3-\sigma)-2\sigma}
  +|T_0|^{2 (\ell_0-\ell-1) (\Gamma_3-\sigma)-(\Gamma_1-\Gamma_3)-\sigma},
\end{split}
\end{equation}
for each $i=1,2,3$.
\end{enumerate}
\end{thm}

Before proceeding with the proof of this theorem, we briefly discuss its consequences and make a few remarks.
First, given a solution of the singular initial value problem $V_S$ from \Theoremref{thm:fluidresult} that is determined by
the asymptotic data $(v_*^0, v_*^1, v_*^2, v_*^3)^{\tr}$, \Theoremref{thm:fluidresult2} yields an open family of
perturbations of $V_S$ that are obtained by solving the regular Cauchy problem from the initial time $t=T_0$ towards $t=0$ with Cauchy data perturbed around the value $V_S(T_0)$ of the solution of the singular initial value problem at $t=T_0$. Furthermore, \Theoremref{thm:fluidresult2} guarantees, for a sufficiently small perturbation of $V_S(T_0)$, that the resulting solution of the Cauchy problem extends all the way to $t=0$ (global existence) and that the components of the fluid vector field, when rescaled with the appropriate  powers of $t$, converge as $t\nearrow 0$ to the asymptotic data $(W^0_C(0), W^1_C(0), W^2_C(0), W^3_C(0))^{\tr}$. It is important to note that the asymptotic data generated this way will, in general, be different from the asymptotic data $(v_*^0, v_*^1, v_*^2, v_*^3)^{\tr}$ that determines the singular solution $V_S$. However, it is a consequence of Theorem \ref{thm:fluidresult2} that the perturbed solution will agree with the solution to the singular initial value problem that is generated from the asymptotic data $(W^0_C(0), W^1_C(0), W^2_C(0), W^3_C(0))^{\tr}$ via \Theoremref{thm:fluidresult}. This leads to the important conclusion that perturbations of solutions of the singular initial value problem will again be solutions of the singular initial value problem.

Next, we observe that in the unperturbed case, that is, $V_0=0$ in \Theoremref{thm:fluidresult2}, we have $V_C=V_S$ and $(v_*^0, v_*^1, v_*^2, v_*^3)^{\tr}=(W^0_C(0), W^1_C(0), W^2_C(0), W^3_C(0))^{\tr}$.
Comparing \eqref{eq:Eulerdecay}, \eqref{eq:EulerPertSecLimit1} and \eqref{eq:EulerPertSecLimit2} with the given value of $\epsilon$, it is interesting to note that \Theoremref{thm:fluidresult2} yields larger decay exponents compared to  \Theoremref{thm:fluidresult}. The reason for this difference is due to the regularity requirements which are higher for  \Theoremref{thm:fluidresult2}. 

Let us comment on part (3) of the theorem which establishes the \emph{closeness} of the \emph{given} asymptotic data $v_*$ (which determine $V_S$) and the \emph{perturbed} asymptotic data $\tilde v_*$ obtained as the limit of the perturbed solution $V_C$ observing that all powers of $|T_0$| in \eqref{eq:EulerCPLimContUFull} and \eqref{eq:stabilityclosenessFull} are positive by virtue of the hypothesis. We expect however that all terms involving powers of $|T_0|$ on the right sides of these estimates can in fact be removed by performing a more detailed analysis especially of $V_S$ and the nonlinear terms in the Euler equations (see the corresponding remark after \Propref{prop:EulerSolveCPCor} and, especially, the suboptimal treatment of nonlinear terms in the step from \eqref{eq:suboptimal1} to \eqref{eq:suboptimal2} which we have performed here for brevity). As a consequence of this we cannot conclude here that the map $V_0\mapsto \tilde v_*$ given by the theorem is continuous from the ball of radius $\delta$ in $H^{k+\ell_0}(\Sigma)$ to $H^{k+\ell}(\Sigma)$. Even though we think that establishing this notion of continuity would be possible with our techniques, we think that it would only be interesting if we could pair it with a corresponding notion of continuity for the forward map given by \Theoremref{thm:fluidresult} to show that it is a homeomorphism in a $C^\infty$-topology as pioneered by Ringstr\"om \cite{ringstrom2017,ringstrom2021,ringstrom2021a}. Conclusions regarding the $C^\infty$-setting can however not be drawn from the current versions of Theorems~\ref{thm:mainresult}~and~\ref{thm:Fuchsian-IVP}, but we expect that it would not be difficult to modify the proofs of these theorems so that such conclusions could be drawn. 
  
We conclude our discussion of \Theoremref{thm:fluidresult2}  with two additional comments. First, Remark~\ref{rem:KSF} also applies to \Theoremref{thm:fluidresult2}, and consequently this stability result automatically extends to the more general case of fluids on Kasner-scalar field spacetimes. Second, the stability of Euler fluids near Kasner singularities in the sense of \Theoremref{thm:fluidresult2} was outstanding even in the Gowdy symmetric setting; see  \cite{beyer2019a} and the related numerical investigations from \cite{beyer2020b} for details.

\subsubsection{Proof of \Theoremref{thm:fluidresult2}}
\label{sec:Eulerproofstability}
The proof of \Theoremref{thm:fluidresult2} involves three main steps. We begin by fixing  {asymptotic data}\footnote{The stronger regularity requirements for $v_*$ in \Theoremref{thm:fluidresult2} will only become relevant in Step~2 of this proof.} $v_*=(v_*^0,\ldots,v_*^3)^{\tr}\in H^{k+\ell}(\Sigma)$ with $v_*^0>0$ and $\ell$ as defined in \Theoremref{thm:fluidresult}. We then label the resulting solution to the singular the value problem of the Euler equations that is generated by this asymptotic data via  \Theoremref{thm:fluidresult} by $V_S$. 

\bigskip
\noindent \underline{Step 1: Global existence and preliminary estimates of small data perturbations of $V_S$}
\bigskip

\noindent Letting $V_C$ denote the solution of the Cauchy problem for the Euler equations \eqref{eq:AAA1} with initial
data at $T_0<0$ determined by
\begin{equation} \label{VC@T0}
 V_C(T_0)=V_{C,0}:=V_S(T_0)+V_0,
\end{equation} 
where $V_0$ is to be considered as the perturbed initial data that measures the deviation from the solution $V_S(T_0)$ at $t=T_0>0$,
we set \begin{equation} \label{VC2UC}
    U_C=T^{-1}V_C,
\end{equation} 
where $T=T(t)$ is defined by \eqref{eq:EulerResc0},
and we recall that $U_C$ satisfies the system \eqref{eq:Eulereqssymmhyp2N} -- \eqref{eq:GEulerPol}. To proceed, we subtract off from $U_C$ the first component
of the asymptotic data $v_*$ by setting\footnote{Observe carefully the variable $u$ defined here is in general different from the variable $u$ introduced in \Sectionref{sec:proofEulerSIVP} and \eqref{eq:Eulereqssymmhyp3N}.}
\begin{equation}
  \label{eq:defvtimes}
  u:=U_C-v_{\times},\quad v_{\times}:=(v_*^0,0,0,0)^{\tr}.
\end{equation}
A straightforward calculation then shows that the system \eqref{eq:Eulereqssymmhyp2N} -- \eqref{eq:GEulerPol} is given in terms of $u$ by
\begin{align}
    B^0(v_{\times}+u) \partial_tu
    +B^1(v_{\times}&+u) \partial_xu
   +B^2(t,v_{\times}+u) \partial_yu
    +B^3(t,v_{\times}+u) \partial_zu
    =\frac 1t \Bc(v_{\times}+u) u \notag\\
    &-B^1(v_{\times}+u) \partial_x v_{\times}
    -B^2(t,v_{\times}+u) \partial_y v_{\times}
    -B^3(t,v_{\times}+u) \partial_z v_{\times}+G(t,u),\label{eq:Eulereqssymmhyp2NCP}
\end{align}
where in deriving this we have exploited the fact that $v_{\times}$, defined in  \eqref{eq:defvtimes}, does not depend on $t$, $\Bc(v_{\times}+u) v_{\times}=0$ as a consequence of \eqref{eq:BcEulerPol}, and $G(v_{\times}+u)=G(u)$ thanks to \eqref{eq:GEulerPol}.

Next, we choose\footnote{The stronger regularity requirements for $V_0$ in \Theoremref{thm:fluidresult2} will only become relevant in Step~2 of this proof.} $V_0\in H^{k}(\Sigma,\Vbb)$, and we observe that the Cauchy data \eqref{VC@T0} for the Euler equations translate into the Cauchy data
\begin{equation}
  \label{eq:EulerPertuData}
 u(T_0)=u_0:=T(T_0)^{-1} V_S(T_0)-v_\times+\Th(T_0)^{-1} V_0
\end{equation}
for the system \eqref{eq:Eulereqssymmhyp2NCP}. Recall that $T(t)$ and $\Th(t)$ are defined in \eqref{eq:EulerResc0} and \eqref{Thdef}, respectively.
Notice here that it is a consequence of \eqref{eq:Gammarestr}, \eqref{VC2UC}, and \eqref{eq:defvtimes}, and \Theoremref{thm:fluidresult}, in particular \eqref{eq:Eulerdecay} for $V=V_S$, that \eqref{eq:EulerPertuData} implies
\begin{equation}
  \label{eq:EulerPertuDataEst}
  \norm{u_0}_{H^{k}(\Sigma)}\lesssim  |T_0|^{\epsilon}+\norm{V_0}_{H^{k}(\Sigma)},
\end{equation}
where the implicit constant depends monotonically on $T_0$ and on the choice of asymptotic data $v_*$ and therefore on $V_S(T_0)$, but not on $V_0$. This estimate will allow us to meet all smallness requirement for $u_0$, which we will encounter in the proof of 
\Theoremref{thm:fluidresult2}, by imposing appropriate smallness conditions on $T_0$ together with corresponding  bounds on $V_0$.

To complete the first step of the proof, we now establish the existence of solutions to the Cauchy problem for the system \eqref{eq:Eulereqssymmhyp2NCP} where the initial data at time $t=T_0$ is specified
according to \eqref{eq:EulerPertuData} and must be chosen suitably small. The precise statement for the existence result is given in the following proposition.
The proposition also yields a first preliminary estimate of the behaviour of the  solutions at $t=0$. This preliminary estimate, which is not yet sharp, will be improved successively in Step~2 of the proof. We express all the results in terms of the quantity $U_C$ instead of $u$; cf.\ \eqref{eq:defvtimes}.

\begin{proposition}
\label{prop:EulerSolveCPCor}
Suppose ${K}$ and $\gamma$ satisfy \eqref{eq:gammarestr}, and $\Gamma_1$, $\Gamma_2$ and
$\Gamma_3$ are defined by \eqref{eq:GammaDef}, $k\in \Zbb_{\geq 3}$, $\ell \in \Zbb_{\geq 1}$ with
$\ell>\Gamma_1/q$ for
$q=\min\{1-\Gamma_1,2\Gamma_3\}$ and $\sigma\in (0,\Gamma_3)$. Further, for given asymptotic data $v_*=(v_*^0,\ldots,v_*^3)^{\tr}\in H^{k+\ell}$ with $v_*^0>0$,
let $V_S$ be the solution  asserted by \Theoremref{thm:fluidresult} of the singular initial value problem of the Euler equations \eqref{eq:AAA1} with equation of state \eqref{eq:EulerEOS} defined on the Kasner spacetime given by \eqref{Kasner-k} on a time interval $[\Tt_0,0)$.

\medskip

\noindent Then there exists $\delta>0$ such that for all sufficiently small $T_0\in [\Tt_0,0)$ and for any $V_0\in H^k(\Sigma)$ with $\norm{V_0}_{H^k(\Sigma)}\le \delta$, the Cauchy problem of
the Euler equations \eqref{eq:AAA1} with equation of state \eqref{eq:EulerEOS} defined on the Kasner spacetime given by \eqref{Kasner-k} with Cauchy data $V_{C,0}=V_S(T_0)+\Th^{-1}(T_0)T(T_0)V_0$ imposed at $t=T_0$
has a unique solution $V_C(t)=T(t) U_C(t)$ with
\begin{equation}
  \label{eq:EulerPertFirstRegU}
U_C \in C_b^0\bigl([T_0,0),H^{k}(\Sigma)\bigr)\cap C^1\bigl([T_0,0),H^{k-1}(\Sigma)\bigr),
\end{equation}
where $T(t)$ and $\Th(t)$ are defined in \eqref{eq:EulerResc0} and \eqref{Thdef}, respectively, and $U_C$ is bounded by
\begin{equation}
  \label{eq:EulerPertFirstDecayU}
  \norm{U_C^1(t)}_{H^{k-1}(\Sigma)}
+\norm{U_C^2(t)}_{H^{k-1}(\Sigma)}
+\norm{U_C^3(t)}_{H^{k-1}(\Sigma)}
\lesssim |t|^{\Gamma_3-\sigma}
\end{equation}
for all $t\in [T_0,0)$. Moreover, the limit
 $\lim_{t\nearrow 0} U_C^0$, denoted $U_C^0(0)$, exists in $H^{k-1}(\Sigma)$
and satisfies
\begin{equation}
  \label{eq:EulerCPLimContU}
  \norm{U_C^0(0)-v_*^0}_{H^{k-1}(\Sigma)}\lesssim |T_0|^{\epsilon}+\norm{V_0}_{H^{k}(\Sigma)},
\end{equation}
with $\epsilon$ defined in Theorem~\ref{thm:fluidresult},
and the estimate
\begin{equation}
  \label{eq:EulerPertFirstLimitU}
  \norm{U_C^0(t)-U_C^0(0)}_{H^{k-1}(\Sigma)}
\lesssim |t|^{1-\Gamma_1+\Gamma_3}+|t|^{2(\Gamma_3-\sigma)}
\end{equation}
holds for all $t\in [T_0,0)$. Finally, $U_C$
satisfies $U_C^0(0)>0$
provided $|T_0|$ and $\norm{V_0}_{H^k(\Sigma)}$ are chosen sufficiently small.
\end{proposition}

We emphasize here that the implicit smallness condition for $T_0$ does not depend on $V_0$ so long as $\norm{V_0}_{H^k(\Sigma)}\le \delta$. This proposition establishes the existence of a limit $U_C ^0(0)$ for the time component of the fluid vector field and  \eqref{eq:EulerCPLimContU} states that this limit is close to the asymptotic data component $v_*^0$ provided $T_0$ is small. We expect that the first term on the right side of \eqref{eq:EulerCPLimContU} can be avoided by using more detailed estimates of $V_S$; in fact, it follows from uniqueness for the Cauchy problem that $v_*=U_C^0(0)$ if $V_0=0$. In any case, it will be the purpose of Proposition~\ref{prop:EulerSolveCPCorStrong} to establish corresponding statements for the spatial components of the fluid vector field.

\begin{proof}[Proof of \Propref{prop:EulerSolveCPCor}]
We will prove this proposition by first showing that the system \eqref{eq:Eulereqssymmhyp2NCP} satisfies all of the coefficient assumptions from Appendix \ref{coeffassumps}. After doing so, the proof will then follow directly from an application of Theorem \ref{thm:Fuchsian-IVP} together with Remark~\ref{rem:Fuchsian-IVP}.(i). To this end, we express \eqref{eq:Eulereqssymmhyp2NCP} in the form
\begin{align}
    B^0(v_{\times},u) \partial_tu
    +B^1(v_{\times},u) \partial_xu
    +B^2(t,v_{\times},u) \partial_yu
    +B^3(t,v_{\times},u) \partial_zu
    =\frac 1t \bar\Bc(v_{\times},u)\Pbb u 
    +F(t,w_{\times},u) \label{Euler-CPCor}
\end{align}
where 
\begin{equation}
  \label{eq:EulerCPDefwx}
  w_{\times}=(v_{\times},\partial_x v_{\times}, \partial_y v_{\times}, \partial_z v_{\times}),
\end{equation}
\begin{equation}
  \label{eq:EulerPertPbb}
  \Pbb=\mathrm{diag}\,(0,1,1,1),\quad \Pbb^\perp=\mathrm{diag}\,(1,0,0,0),
\end{equation}
and, with the same slight abuse of notation as in \Sectionref{sec:proofEulerSIVP}, we have set 
\begin{align*}
B^0(\Zt,v)&=B^0(\Zt+v),\\
B^1(\Zt,v)&= B^1(\Zt+v), \\
B^2(t,\Zt,v)&= B^2(t,\Zt+v),\\
B^3(t,\Zt,v)&= B^3(t, \Zt+v),\\
\bar\Bc(\Zt,v)&=\Bc(\Zt+v)+r(\Zt+v)(\Zt^0+v^0) a\Pbb^\perp
\end{align*}
where $r$ is the polynomial used in \eqref{eq:B0EulerPol} -- \eqref{eq:GEulerPol}
and
\begin{equation}
\label{eq:Eulereqssymmhyp2NCPNN}
F(t,Z,v)=-B^1(Z_0+v) Z_{01}-B^2(t, Z_0+v) Z_{02}-B^3(t, Z_0+v) Z_{03}+G(t,v).
\end{equation}
Here, $a\in \Rbb$ is a constant to be fixed below, $v \in \Vbb$, $\Zt\in Z_1$ where
$Z_1$ is the rank-$1$-subbundle of $\Vbb$ defined by elements of the form $\Zt=(\Zt^0,0,0,0)^{\tr}$, and 
$Z\in Z_2$ where $Z_2$ is the trivial rank-$4$ bundle $Z_2=Z_1\oplus Z_1\oplus Z_1\oplus Z_1$ whose elements we can can describe in components $Z=(Z_0, Z_{01}, Z_{02}, Z_{03})^{\tr}$.

Now, for any  $\Rc>0$ and $\rc\in (0,\Rc)$, we further define $\Zc_1$ as the bounded open subset of $Z_1$ given by vectors $\Zt$  with components $(\Zt^0,0,0,0)$ with $r<\Zt^0<\Rc$. Correspondingly, the bounded open subset $\Zc_2$ of $Z_2$ is defined as $\Zc_1\oplus B_\Rc(Z_1)\oplus B_\Rc(Z_1)\oplus B_\Rc(Z_1)$. Given these definitions and the expression \eqref{eq:B0EulerPol} and \eqref{eq:BcEulerPol}, it is not difficult to verify, for any $R>0$, that
\begin{equation}\label{B0barBc-reg-CPcor}
B^0 \in  C^\infty\bigl(\Zc_1\oplus B_R(\Vbb),L(\Vbb)\bigr) \AND \bar\Bc\in C^\infty\bigl(\Zc_1
\oplus B_R(\Vbb),L(\Vbb)\bigr).
\end{equation} Moreover, given
\begin{equation*} 
    \sigma \in (0,\Gamma_3),
\end{equation*}
we can, by reducing the size of $R$ if necessary, ensure that the inequality 
\begin{equation} 
\frac{1}{\gamma_1} \text{id}_{\Vbb_{\pi(v)}} \leq  B^0(\Zt, v)\leq \frac{1}{\kappa} \bar\Bc(\Zt,v) \leq \gamma_2 \textrm{id}_{\Vbb_{\pi(v)}}
\end{equation}
holds
for  all $(\Zt,v)\in \Zc_1\oplus B_{R}(V)$ by setting
\begin{equation}
\label{eq:EulerCPConst}
\gamma_1=\frac 2{\rc^3(\gamma-1)},\quad \kappa=\Gamma_3-\sigma,\quad a=\Gamma_3/(\gamma-1),
\end{equation}
and choosing $\gamma_2>0$ accordingly
making use of \eqref{eq:gammarestr} and \eqref{eq:Gammarestr}.

We further observe, with the help of \eqref{eq:B0EulerPol} and \eqref{eq:BcEulerPol}, that the matrices $\Pbb$ and $\bar\Bc$ commute, that is,
\begin{equation}
[\Pbb,B^0(\Zt,v)]=0
\end{equation}
for all $(t,\Zt,v)\in [T_0,0)\times \Zc_1\oplus B_{R}(V)$, 
and the matrix $B^0$ is symmetric and satisfies
\begin{equation}\label{ProjB0-CPCor}
\Pbb B^0(\Zt,v) \Pbb^\perp=(\Pbb ^\perp B^0(\Zt,v) \Pbb)^{\tr}=\Ord(\Pbb v).
\end{equation}
 It also follows from \eqref{eq:B0EulerPol}, \eqref{eq:BcEulerPol}, \eqref{eq:Eulerpol1} and \eqref{eq:EulerCPConst} that
 $\bar{\Bc}$ and $B^0$ can be expanded, respectively, as
\begin{equation} \label{barBcexp-CPCor}
\bar\Bc(\Zt,v)=\tilde{\bar\Bc}(\Zt)+\Ord(v),\quad
\tilde{\bar\Bc}(\Zt) =(\gamma-1)(\Zt^0)^3\mathrm{diag}\,\left(a,\Gamma_1,\Gamma_2,\Gamma_3\right),
\end{equation}
and
\begin{equation} \label{B0exp-CPCor}
B^0(\Zt,v)={\Bt^0}(\Zt)+\Ord(v),\quad
{\Bt^0}(\Zt)=(\Zt^0)^3\mathrm{diag}\,\left(1, \gamma-1, \gamma-1, \gamma-1\right),
\end{equation}
where the maps 
$\Bt^0$ and $\tilde{\bar\Bc}$ lie in
$C^0\bigl([T_0,0], C^\infty(\Zc_1,L(\Vbb))\bigr)$.
We further note by \eqref{eq:B1EulerPol} -- \eqref{eq:B3EulerPol} that the matrices $B^1$, $B^2$ and $B^3$ are all symmetric and satisfy
\begin{equation}\label{B1B2B3-reg-CPcor}
B^1 \in  C^\infty\bigl(\Zc_1\oplus B_R(\Vbb),L(\Vbb)\bigr) \AND B^2,B^3\in C^0\bigl([T_0,0],C^\infty\bigl(\Zc_1
\oplus B_R(\Vbb),L(\Vbb)\bigr)\bigr).
\end{equation}

Next, we turn our attention to the map $F$ defined above by \eqref{eq:Eulereqssymmhyp2NCPNN}.
We decompose $F$ as
\begin{equation} \label{Fexp-CPCor}
    F(t,Z,v)= \frac{1}{|t|^{1-p}}\Ft(t,Z)
    + \frac{1}{|t|^{1-p}}F_0(t,Z,v) +
    \frac{1}{|t|}F_2(t,Z,v)
\end{equation}
where
\begin{align}
  \Ft(t,Z)&=-(-t)^{1-p}\Bigl(B^1(Z_0,0) Z_{01}+B^2(t, Z_0,0) Z_{02}+B^3(t, Z_0,0) Z_{03}\Bigr), \label{eq:EulerCPFt}\\
  F_0(t,Z,v)&=-(-t)^{1-p}\Bigl( (B^1(Z_0,v)-B^1(Z_0,0)) Z_{01}
            +\bigl(B^2(t,Z_0,v)-B^2(t,Z_0,0)) Z_{02} \notag \\
            &\hspace{5.0cm}
            +(B^3(t,Z_0,v)-B^3(t,Z_0,0)\bigr) Z_{03}\Bigr) \label{eq:EulerCPF0}
\intertext{and}
  F_2(t,Z,v)&=-t\, G(t,v).\label{eq:EulerCPF2}
\end{align}
Assuming for the moment that $p\in (0,1]$,
it not difficult to verify from \eqref{eq:B1EulerPol} -- \eqref{eq:B3EulerPol}, \eqref{eq:GEulerPol} and \eqref{eq:EulerCPFt} -- \eqref{eq:EulerCPF2} that
\begin{equation}\label{FtF0F2-reg-CPcor}
\Ft \in C^0\bigl([T_0,0],C^\infty(\Zc_2,\Vbb )\bigr)
\AND
F_0,F_2 \in C^0\bigl([T_0,0],C^\infty\bigl(\Zc_2
\oplus B_R(\Vbb),\Vbb\bigr)\bigr).
\end{equation}
We further observe from \eqref{eq:GEulerPol} and \eqref{eq:EulerCPF0}-\eqref{eq:EulerCPF2} that 
\begin{equation} \label{F2-CPCor}
F_0 = \Ord(v), \quad 
\Pbb F_2 =0 \AND  F_2 = \Ordc\biggl(\frac{\lambda_3}{R} v\otimes v\Bigr)
\end{equation}
for some constant $\lambda_3=\Ord(R)$.

Before proceeding, we note that because $v_*^0\in H^{k+1}(\Sigma)$ satisfies $v^0_*>0$ and is time-independent by assumption, we have by \eqref{eq:defvtimes} and the Sobolev inequality that there exist constants $0<\rc<\Rc_0$ such that
\begin{equation*} 
    v_{\times}\in H^{k+1}(\Sigma,\Zc_1) \AND \del{t}v_{\times}=0
\end{equation*}
for any $\Rc\geq \Rc_0$.
This, in turn, implies with the help of the \eqref{eq:EulerCPDefwx} and the Sobolev inequality that
\begin{equation} \label{wtimesreg}
    w_{\times}\in H^{k}(\Sigma,\Zc_2) \AND \del{t}w_{\times}=0 
\end{equation}
for $\Rc$ chosen sufficiently large.

To complete the analysis of the coefficients of the Fuchsian equation \eqref{Euler-CPCor},  
we need to examine the divergence map $\Div\!B$ given locally by \eqref{divBdef}, where all the connection coefficients in \eqref{divBdef} vanish since the connection we are using is flat and its coefficients vanish in the global frame $\{\del{x},\del{y},\del{z}\}$. Furthermore, since $v_\times$ is independent of $t$, we can take the map $\Div\! B$ to be independent of the variable $\dot{z}_1$ since $z_1=v_\times$ and $\del{t}v_\times=0$. This implies in particular that the constant $q$ in \eqref{divBdef} plays no role here. 
We  recall also that the map $\Div \! B$, when evaluated on solutions
of \eqref{Euler-CPCor}, is given by $\Div\! B= \del{t}B^0+\del{x}B^1+\del{y}B^2+\del{z}B^3$, see \eqref{divBonshell}.

Now, it is easy to see that all the terms in $\Div\! B$ other than
\[M:=\frac 1t D_v B^0(x,z_1,v)\cdot (B^0(x,z_1,v))^{-1}\Bigl[\Bc(x,z_1,v)\Pbb v + t G(t,x,v)\Bigr]\]
are $\Ordc( |t|^{-(1-p)}\theta)$ for some $\theta>0$ provided $p$ is chosen to satisfy
\begin{equation*}
  p= 1-\Gamma_1+\Gamma_3.
\end{equation*}
To proceed with the analysis, we decompose the $M$ as follows
\begin{align*}
M=\frac 1t \Bigl[\bigl(\underbrace{D_{\Pbb v} B^0(x,z_1,v)}_{=\Ord(1)}\cdot \underbrace{\Pbb (B^0(x,z_1,v))^{-1}\Pbb}_{=\Ord(1)}
+\underbrace{D_{\Pbb^\perp v} B^0(x,z_1,v)}_{=\Ord(1)}\cdot \underbrace{\Pbb^\perp (B^0(x,z_1,v))^{-1}\Pbb}_{=\Ord(\Pbb v)}\bigr) \underbrace{\Bc(x,z_1,v)\Pbb v}_{=\Ord(\Pbb v)}\\
+ \bigl(\underbrace{D_{\Pbb v} B^0(x,z_1,v)}_{=\Ord(1)}\cdot \underbrace{\Pbb (B^0(x,z_1,v))^{-1}\Pbb^\perp}_{=\Ord(\Pbb v)}
+\underbrace{D_{\Pbb^\perp v} B^0(x,z_1,v)}_{=\Ord(1)}\cdot \underbrace{\Pbb^\perp (B^0(x,z_1,v))^{-1}\Pbb^\perp}_{=\Ord(1)}\bigr) \underbrace{t G(t,x,v)}_{=\Ord(\Pbb v\otimes\Pbb v)}\Bigr].
\end{align*}
From this expression, it is clear that there exist constants $\beta_1$, $\beta_3$ and $\beta_5$ of $\Ord(R)$ such that
\begin{equation}  \label{DivBexp1-CPCor}
\Pbb M\Pbb=\Ordc(|t|^{-1}\beta_1),\quad \Pbb M\Pbb^\perp=\Ordc\left(|t|^{-1}\frac{\beta_3}R \Pbb v\right),\quad \Pbb^\perp M\Pbb=\Ordc\left(|t|^{-1}\frac{\beta_5}R \Pbb v\right).
\end{equation}
In order to show that there also exists a constant $\beta_7=O(R^2)$ such that
\begin{equation} \label{DivBexp2-CPCor}
\Pbb^\perp M\Pbb^\perp=\Ordc\left(|t|^{-1}\frac{\beta_7}{R^2} \Pbb v\otimes\Pbb v\right),
\end{equation}
we need to have a closer look at $\Pbb^\perp D_{\Pbb v} B^0(x,z_1,v) \Pbb^\perp$
in the above formula for $M$. Since it follows immediately from \eqref{eq:B0EulerPol}, \eqref{eq:Eulerpol2} and \eqref{eq:EulerPertPbb} that 
$\Pbb^\perp D_{\Pbb v} B^0(x,z_1,v)\Pbb^\perp=\Ord(\Pbb v)$, we see that \eqref{DivBexp2-CPCor} does in fact hold for some some constant $\beta_7=O(R^2)$.

From the above analysis, and in particular the results \eqref{eq:EulerPertPbb}, \eqref{B0barBc-reg-CPcor} -- \eqref{Fexp-CPCor}, \eqref{FtF0F2-reg-CPcor}-\eqref{F2-CPCor} and \eqref{DivBexp1-CPCor} -- \eqref{DivBexp2-CPCor}, it follows that the Fuchsian equation \eqref{Euler-CPCor} satisfies all of the coefficient assumptions from Appendix \ref{coeffassumps} for the following choice of constants:
\begin{itemize}
    \item[(i)] $\gamma_1$, $\gamma_2$, $\kappa$, $p$, $\theta$, $\lambda_3$, $\beta_1$, $\beta_3$, $\beta_5$ and $\beta_7$ as above,
    \item[(ii)]  and $\alpha=0$, $\lambda_1=\lambda_2=0$, and $\beta_0=\beta_2=\beta_4=\beta_6=0$.
\end{itemize}

Noting that the constant $\mathtt{b}$ defined by \eqref{thm:Fuchsian-IVP-btt} vanishes, it is clear, by reducing the size of $R$ if necessary, that we can ensure that all the constants $\beta_i$ and $\lambda_i$ are small enough so that $\kappa$ satisfies the inequality \eqref{thm:Fuchsian-IVP-kappa}.
Additionally, we note from \eqref{eq:Gammarestr},  \eqref{eq:B1EulerPol} -- \eqref{eq:B3EulerPol}, \eqref{eq:EulerCPFt} and \eqref{wtimesreg} that
\begin{equation*}
  \norm{\Ft(t, w_\times)}_{H^k(\Sigma)}\lesssim |t|^{1-\Gamma_1+\Gamma_3-p},
\end{equation*}
and, since $\Ft(t, w_\times)$ is independent of the choice of $V_0$, the implicit constant here does not depend on $V_0$ (but on $V_S$).
From this, we deduce, since $1-\Gamma_1+\Gamma_3>0$, that
\begin{equation} \label{eq:EulerCPFTEst}
    \int_{T_0}^0 |s|^{p-1}\norm{\Ft(s, w_\times)}_{H^k(\Sigma)}\, ds \lesssim
    \int_{T_0}^0 (-s)^{\Gamma_3-\Gamma_1}\,ds \lesssim |T_0|^{\Gamma_3-\Gamma_1+1}.
\end{equation}
Thus, we can
arrange that 
\begin{equation}
  \label{eq:EulerCPFTEst2}
 \int_{T_0}^0 |s|^{p-1}\norm{\Ft(s, w_\times)}_{H^k(\Sigma)}\, ds \le\delta
\end{equation}
in a way that is independent of the choice of $V_0$
for any given $\delta>0$ by reducing the size of $|T_0|$ as needed. 
This allows us to satisfy the smallness condition in Theorem \ref{thm:Fuchsian-IVP} for the term $\Ft$ from the expansion \eqref{Fexp-CPCor} for $F$. 

Given all this, an application of Theorem \ref{thm:Fuchsian-IVP} from the Appendix \ref{appendix:existence} yields, for $u_0$ satisfying $\norm{u_0}_{H^k(\Sigma)}\leq \delta$ with $\delta>0$ sufficiently small, the existence of a unique solution $u$ of the Cauchy problem of \eqref{eq:Eulereqssymmhyp2NCP} with  $u(T_0)=u_0$ and satisfying
\begin{equation}
\label{eq:EulerPertFirstReg}
u \in C_b^0\bigl([T_0,0),H^{k}(\Sigma)\bigr)\cap C^1\bigl([T_0,0),H^{k-1}(\Sigma)\bigr).
\end{equation}
Shrinking $R$  even further if necessary to make $\beta_1$ sufficiently small, we can make the quantity $\zeta$ in  Theorem~\ref{thm:Fuchsian-IVP} arbitrarily close to $\kappat=\kappa$, and the estimate
\begin{equation}
  \label{eq:EulerPertFirstDecay}
  \norm{u^1(t)}_{H^{k-1}(\Sigma)}+\norm{u^2(t)}_{H^{k-1}(\Sigma)}+\norm{u^3(t)}_{H^{k-1}(\Sigma)}
\lesssim|t|^{\Gamma_3-\sigma}, \quad T_0\leq t < 0,
\end{equation}
follows from
the decay estimate \eqref{eq:decayPQunt} for $\norm{\Pbb u}_{H^{k-1}}$ from Theorem~\ref{thm:Fuchsian-IVP}  since  $\lambda_1=\alpha=\mathtt{b}=0$, $\kappa=\Gamma_3-\sigma>0$, and the projection matrix $\Pbb$ is defined by \eqref{eq:EulerPertPbb}.  
Theorem \ref{thm:Fuchsian-IVP} also yields the existence of the limit $\lim_{t\nearrow 0} u^0$  in $H^{k-1}(\Sigma)$, which we denote as $u^0(0)$, and the estimate
\begin{equation}
  \label{eq:EulerPertFirstLimit}
  \norm{u^0(t)-u^0(0)}_{H^{k-1}(\Sigma)}
\lesssim|t|^{1-\Gamma_1+\Gamma_3}+|t|^{2(\Gamma_3-\sigma)}, \quad T_0\leq t < 0,
\end{equation}
is also a direct consequence of this theorem and the improvement to the decay estimate for $\norm{u(t)-u(0)}_{H^{k-1}}$ described in Remark~\ref{rem:Fuchsian-IVP}.(i). 
Finally, from the energy estimate from Theorem \ref{thm:Fuchsian-IVP},
we have that
\begin{equation*}
  \norm{u^0(0)}_{H^{k-1}(\Sigma)}\lesssim \norm{u_0}_{H^k(\Sigma)}+\int_{T_0}^0 |s|^{p-1}\norm{\Ft(s, w_\times)}_{H^k(\Sigma)}\, ds.
\end{equation*}
In the light of \eqref{eq:EulerCPFTEst},
it is then clear that the inequality
\begin{equation}
  \label{eq:EulerCPLimCont}
  \norm{u^0(0)}_{H^{k-1}(\Sigma)}\lesssim \norm{u_0}_{H^k(\Sigma)}+|T_0|^{1-\Gamma_1+\Gamma_3}
\end{equation}
holds. 

Finally, we notice that the existence of the solution $V_S$ is guaranteed by \Theoremref{thm:fluidresult} under our hypothesis here. Since $U_C^0(0)=v_*^0+u^0(0)$, the proof is completed by invoking \eqref{eq:defvtimes}, \eqref{eq:EulerPertuData} and \eqref{eq:EulerPertuDataEst}.
\end{proof}

\bigskip
\noindent \underline{Step 2: Improved asymptotics and existence of limits for \emph{all} components of $V_C$}
\bigskip

\noindent 
Setting 
\begin{equation} \label{WCdef}
    W_C=\Th U_C
\end{equation}
where $\Th$ is defined by \eqref{Thdef},
\Propref{prop:EulerSolveCPCor} implies that
\begin{equation}
\label{eq:EulerPertFirstRegNNN}
W_C^0, |t|^{-(\Gamma_3-\Gamma_1-\sigma)} W_C^1, |t|^{-(\Gamma_3-\Gamma_2-\sigma)} W_C^2, |t|^{\sigma} W_C^3\in C_b^0\bigl([T_0,0),H^{k-1}(\Sigma)\bigr),
\end{equation}
and that  the limit $W_C^0(0)=\lim_{t\nearrow 0}W_C^0$ exists in $H^{k-1}(\Sigma)$ and
is given by
\begin{equation*}
W_C^0(0)=U_C^0(0).
\end{equation*}
The aim of this step of the proof is now to establish the existence of limits for \emph{all} components of $W_C$ at the cost of some amount of differentiability. The precise statement is summarised in the following proposition.

\begin{proposition}
  \label{prop:EulerSolveCPCorStrong}
  Suppose in addition to the hypothesis of Proposition \ref{prop:EulerSolveCPCor} that the following hold: $v_*=(v_*^0,\ldots,v_*^3)^{\tr}\in H^{k+\ell_0+\ell}(\Sigma)$ with $v_*^0>0$, $V_0\in H^{k+\ell_0}(\Sigma)$ with $\norm{V_0}_{H^{k+\ell_0}(\Sigma)}\le \delta$, and $0<\sigma < \min\{\Gamma_3/4,\Gamma_2-\Gamma_3\}$ for any integer $\ell_0$ satisfying 
\begin{equation}
  \label{eq:EulerCPDefell0Pre}
  \ell_0\ge \ell+2,\quad \ell_0>\ell+\frac 12+\frac{\Gamma_1}{2(\Gamma_3-\sigma)}.
\end{equation}
Then the solution $U_C$ asserted by Proposition \ref{prop:EulerSolveCPCor} has the property that
\begin{equation*}
U_C \in C_b^0\bigl([T_0,0),H^{k+\ell_0}(\Sigma)\bigr)\cap C^1\bigl([T_0,0),H^{k+\ell_0-1}(\Sigma)\bigr),
\end{equation*}
and the limits
$\lim_{t\nearrow 0} U_C^0$ in $H^{k+\ell_0-1}(\Sigma)$, $    \lim_{t\nearrow 0}|t|^{-\Gamma_1} U_C^1$ in $H^{k+\ell}(\Sigma)$, $    \lim_{t\nearrow 0}|t|^{-\Gamma_2} U_C^2$ in $H^{k+\ell}(\Sigma)$ and $    \lim_{t\nearrow 0}|t|^{-\Gamma_3} U_C^3$ in $H^{k+\ell_0-2}(\Sigma)$,
which we denote by
$W_C^0(0)$, $W_C^1(0)$, $W_C^2(0)$, and $W_C^3(0)$, respectively, all exist. Moreover, the estimates
\begin{gather}
  \label{eq:EulerPertSecLimit1Pre}
  \norm{W_C^0 (t)-W_C^0(0)}_{H^{k+\ell_0-1}(\Sigma)}\lesssim |t|^{1-\Gamma_1+\Gamma_3}+|t|^{2(\Gamma_3-\sigma)},\\
\label{eq:EulerPertSecLimit2Pre}
\begin{split}
  \norm{W_C^1(t)-W_C^1(0)}_{H^{k+\ell}(\Sigma)}
  +\norm{W_C^2(t)-W_C^2(0)}_{H^{k+\ell}(\Sigma)}\lesssim & |t|^{1-\Gamma_1}+|t|^{2(\Gamma_3-\sigma)-2\sigma}\\
  &+|t|^{2 (\ell_0-\ell-1) (\Gamma_3-\sigma)-(\Gamma_1-\Gamma_3)-\sigma}
\end{split}
\intertext{and}
\label{eq:EulerPertSecLimit3Pre}
  \norm{W_C^3(t)-W_C^3(0)}_{H^{k+\ell_0-2}(\Sigma)}
  \lesssim |t|^{1-\Gamma_1}+|t|^{2(\Gamma_3-\sigma)-\sigma},
\end{gather}
hold for all $t\in [T_0,0)$.

Given another field $\Vt_0\in H^{k+\ell_0}(\Sigma)$ with $\norm{\Vt_0}_{H^{k+\ell_0}(\Sigma)}\le \delta$, the limits $W_C^i(0)$ and $\Wt_C^i(0)$ attained by the corresponding solutions $W_C(t)$ and $\Wt_C(t)$, respectively, launched from the same initial time $T_0$, satisfy the estimate
\begin{equation}
  \label{eq:stabilitycloseness}
  \begin{split}
  \norm{W_C^i(0)-\tilde W_{C}^i(0)}_{H^{k+\ell}(\Sigma)}\le &\norm{V_0^i-\Vt_0^i }_{H^{k+\ell}(\Sigma)}\\
  &+C\bigl(
  |T_0|^{1-\Gamma_1}+|T_0|^{2(\Gamma_3-\sigma)-2\sigma}
  +|T_0|^{2 (\ell_0-\ell-1) (\Gamma_3-\sigma)-(\Gamma_1-\Gamma_3)-\sigma}\bigr),
\end{split}
\end{equation}
for each $i=1,2,3$.
\end{proposition}

Before we prove \Propref{prop:EulerSolveCPCorStrong}, we observe that every $\ell_0$ that satisfies \eqref{eq:EulerCPDefell0} can also be made to satisfy \eqref{eq:EulerCPDefell0Pre} by choosing $\sigma>0$ sufficiently small. Further, we note that the estimate \eqref{eq:EulerPertSecLimit1} is the same as \eqref{eq:EulerPertSecLimit1Pre}, and the estimate \eqref{eq:EulerPertSecLimit2}  can be deduced from \eqref{eq:EulerPertSecLimit2Pre} and \eqref{eq:EulerPertSecLimit3Pre} if
\[2 (\ell_0-\ell-1) (\Gamma_3-\sigma)-(\Gamma_1-\Gamma_3)-\sigma\ge 2(\Gamma_3-\sigma)-2\sigma\]
which we establish as follows.
We observe that every $\ell_0$ that satisfies \eqref{eq:EulerCPDefell0} also
satisfies
\[\ell_0-\ell\ge \frac 32+\frac{\Gamma_1}{2(\Gamma_3-\sigma)}-\frac{\sigma}{\Gamma_3-\sigma}\]
provided we choose $\sigma>0$ sufficiently small.
Given this, we have
\begin{align*}
  &2 (\ell_0-\ell-1) (\Gamma_3-\sigma)-(\Gamma_1-\Gamma_3)-\sigma- 2(\Gamma_3-\sigma)+2\sigma\\
=& 2 (\ell_0-\ell-3/2) (\Gamma_3-\sigma)-(\Gamma_3-\sigma)-\Gamma_1+\Gamma_3+\sigma\\
\ge &(\Gamma_1-2\sigma)-(\Gamma_3-\sigma)-\Gamma_1+\Gamma_3+\sigma=0,
\end{align*}
and consequently, the estimate \eqref{eq:EulerPertSecLimit2}  follows from \eqref{eq:EulerPertSecLimit2Pre} and \eqref{eq:EulerPertSecLimit3Pre}.
\Propref{prop:EulerSolveCPCorStrong} therefore can be used to complete the proof of  \Theoremref{thm:fluidresult2} except for the last statement regarding the solution $\Vt_S$ of the singular initial value problem, which we will address in Step~3 of this proof below. 
\begin{proof}
By assumption $k\geq 5$ and $\ell_0\ge 0$ (as a consequence of
\eqref{eq:EulerCPDefell0Pre}), and so, we have by 
\eqref{eq:EulerPertFirstRegNNN} that
\begin{equation} \label{eq:EulerPertFirstRegNNNa}
W_C^0, |t|^{\Gamma_1-\Gamma_3+\sigma} W_C^1, |t|^{\Gamma_2-\Gamma_3+\sigma} W_C^2, |t|^{\sigma} W_C^3\in C_b^0\bigl([T_0,0),H^{k+\ell_0-1}(\Sigma)\bigr).
\end{equation}
Additionally, we observe from the definitions \eqref{eq:askjdj3} and \eqref{WCdef} that $W_C$
satisfies \eqref{eq:EulereqsDec} and \eqref{BhnGhndef} if we set $W_{n+1}=W_n=W_C$, $\Bh^i_{n+1}=\Bh^i_{n}=\Bh^i_{C}$ and $G_{n+1}=G_{n}=G_{C}$ in all the expressions there. Doing so shows that $W_C$ satisfies
  \begin{equation}
    \label{eq:3333339392}
  \partial_t W_C=
  - \Bh_C^i\partial_i W_C
+G_C.
  \end{equation}
Integrating this equation in time and expressing the result in components, we find that the components $W_C^j$, $j=0,1,2,3$, satisfy
\begin{equation}
  \label{eq:691823iwej2}
  W_C^j(t)=W_C^j(T_0)+\int_{T_0}^t\Bigl(- \sum_{i=1}^3\sum_{k=0}^3\Bh_{C,k}^{j,i}\partial_i W^k_C(s)
+G_C^j(s)\Bigr)ds
\end{equation}
for all $t\in [T_0,0)$. Applying the $H^l(\Sigma)$ norm to this expression for any integer $l$ with $2\le l\le k+\ell_0-2$, we obtain, with the help of the triangle inequality, the estimate
\begin{equation}
  \label{eq:691823iwej}
  \norm{W_C^j(t)}_{H^l(\Sigma)}\le \norm{W_C^j(T_0)}_{H^l(\Sigma)} +\int_{T_0}^t\bnorm{- \sum_{i=1}^3\sum_{k=0}^3\Bh_{C,k}^{j,i}\partial_i W^k_C(s)
+G_C^j(s)}_{H^l(\Sigma)}ds.
\end{equation}

In order to successively improve the estimates \eqref{eq:EulerPertFirstRegNNNa} for the components $W_C^1$ and $W_C^2$, we now set up the following inductive argument to show that
\begin{equation}
\label{eq:EulerPertFirstRegNNNNbb}
|t|^{-\min\{2m(\Gamma_3-\sigma)-(\Gamma_1-\Gamma_3)-\sigma,0\}}W_C^1,
|t|^{-\min\{2m(\Gamma_3-\sigma)-(\Gamma_2-\Gamma_3)-\sigma,0\}}W_C^2\in C_b^0\bigl([T_0,0),H^{k+\ell_0-1-m}(\Sigma)\bigr)
\end{equation}
for each integer $m\in [0,k+\ell_0-4]$. The base case $m=0$ is given by \eqref{eq:EulerPertFirstRegNNNa}. For the induction step $m\mapsto m+1$, let us assume that \eqref{eq:EulerPertFirstRegNNNNbb} holds for an arbitrary integer $m\in [0,k+\ell_0-5]$ and that $W^0_C$ and $W^3_C$ satisfy \eqref{eq:EulerPertFirstRegNNNa}. For convenience, let us define
\begin{equation}
  \label{eq:etadef}
  \eta_a:=-\min\{2m(\Gamma_3-\sigma)-(\Gamma_a-\Gamma_3),-\sigma\}
=\max\{\Gamma_a-\Gamma_3-2m(\Gamma_3-\sigma),\sigma\},
\end{equation}
for $a=1,2$. From this, the condition for $m$ above, the condition for $\sigma$ in the proposition, and \eqref{eq:Gammarestr}, we easily see that 
\begin{equation}
  \label{eq:etaest}
  0\le \eta_a\le\Gamma_a-\Gamma_3,\quad \eta_2\le\eta_1.
\end{equation}
A lengthy analysis involving \eqref{eq:BhnGhndef1} -- \eqref{eq:BhnGhndefLast} and \eqref{eq:EulerPertFirstRegNNNNbb} then shows, observing that $k+\ell_0-1-(m+1)>3/2$ as required to apply the calculus inequalities, in particular, the product and Moser estimates, we conclude with the help of the analogue of \eqref{eq:FuchsianEnergyEstimate} for \Propref{prop:EulerSolveCPCor} (which was not listed there for brevity) and \eqref{eq:EulerCPFTEst2} both for $k$ replaced by $k+\ell_0$, and, the condition $\norm{V_0}_{H^{k+\ell_0}(\Sigma)}\le \delta$, that
\begin{align}
  \label{eq:EulerCPEst11}
  \bnorm{- \sum_{i=1}^3\sum_{k=0}^3\Bh_{C,k}^{1,i}\partial_i W^k_C(s)
+G_C^1(s)}_{H^{k+\ell_0-1-(m+1)}(\Sigma)}&\lesssim |s|^{\min\{2(\Gamma_3-\sigma)-\eta_1-\sigma,1-\Gamma_1\}-1},\\
  \label{eq:EulerCPEst22}
\bnorm{- \sum_{i=1}^3\sum_{k=0}^3\Bh_{C,k}^{2,i}\partial_i W^k_C(s)
+G_C^2(s)}_{H^{k+\ell_0-1-(m+1)}(\Sigma)}&\lesssim |s|^{\min\{2(\Gamma_3-\sigma)-\eta_2-\sigma,1-\Gamma_1\}-1},\\
\label{eq:EulerCPEst33}
\bnorm{- \sum_{i=1}^3\sum_{k=0}^3\Bh_{C,k}^{3,i}\partial_i W^k_C(s)
+G_C^3(s)}_{H^{k+\ell_0-1-(m+1)}(\Sigma)}&\lesssim |s|^{\min\{2(\Gamma_3-\sigma)-\sigma,1-\Gamma_1\}-1},
\end{align}
with implicit constants that are independent of $V_0$ so long as $\norm{V_0}_{H^{k+\ell_0}(\Sigma)}\le \delta$.
From this \eqref{eq:691823iwej} implies that
\begin{equation*}
|t|^{-\min\{2(\Gamma_3-\sigma)-\eta_1-\sigma,0\}}W_C^1,
|t|^{-\min\{2(\Gamma_3-\sigma)-\eta_2-\sigma,0\}}W_C^2
\in C_b^0\bigl([T_0,0), H^{k+\ell_0-1-(m+1)}(\Sigma)\bigr),
\end{equation*}
where in deriving this we have invoked \eqref{eq:Gammarestr} again.
For each $a=1,2$, we have
\begin{align*}
  \min\{2(\Gamma_3-\sigma)-\eta_a-\sigma,0\}  &=\min\{2(\Gamma_3-\sigma)-\sigma+2m(\Gamma_3-\sigma)-(\Gamma_a-\Gamma_3), 2(\Gamma_3-\sigma)-2\sigma,0\}\\
  &=\min\{2(m+1)(\Gamma_3-\sigma)-(\Gamma_a-\Gamma_3) -\sigma, 0\}
\end{align*}
thanks to the condition for $\sigma$ in the hypothesis.
This establishes the induction step $m\mapsto m+1$ and therefore confirms that \eqref{eq:EulerPertFirstRegNNNNbb} holds for every integer $m\in [0,k+\ell_0-4]$.

Now, consider an arbitrary monotonically increasing sequence $\{t_n\}$ of negative numbers in $[T_0,0)$ that converges to $0$. \Eqref{eq:691823iwej2} implies
\begin{equation*}
  W_C^j(t_n)-W_C^j(t_{\nt})=\int_{t_{\nt}}^{t_n}\Bigl(- \sum_{i=1}^3\sum_{k=0}^3\Bh_{C,k}^{j,i}\partial_i W^k_C(s)
+G_C^j(s)\Bigr)\,ds,
\end{equation*}
and hence, that
\begin{equation}
  \label{eq:691823iwej2NNN}
  \bnorm{W_C^j(t_n)-W_C^j(t_{\nt})}_{H^{k+\ell_0-1-(m+1)}(\Sigma)}\le\int_{t_{\nt}}^{t_n}\bnorm{- \sum_{i=1}^3\sum_{k=0}^3\Bh_{C,k}^{j,i}\partial_i W^k_C(s)
+G_C^j(s)}_{H^{k+\ell_0-1-(m+1)}(\Sigma)}\, ds
\end{equation}
given any integer $m\in [0,k+\ell_0-5]$ provided $t_{\nt}\le t_n$.

Let us now pick $m=\ell_0-\ell-2$. It follows that $m\in [0,k+\ell_0-5]$ because of the first condition for $\ell_0$ in \eqref{eq:EulerCPDefell0Pre} and because $m=\ell_0-\ell-2\le (k-3)+\ell_0-\ell-2\le k+\ell_0-5$ given that $k\ge 3$ and $\ell\ge 0$ by assumption. With this choice of $m$ we also have that $2(\Gamma_3-\sigma)-\eta_1-\sigma>0$ since, by \eqref{eq:etadef}, we have
$2(\Gamma_3-\sigma)-\eta_1-\sigma
  =\min\{2(m+1)(\Gamma_3-\sigma)-(\Gamma_1-\Gamma_3)-\sigma, 2(\Gamma_3-\sigma)-2\sigma\}$, from which we get $2(\Gamma_3-\sigma)-2\sigma>0$ as a consequence of the restrictions for $\sigma$ and 
  \begin{align*}
    2(m+1)(\Gamma_3-\sigma)-(\Gamma_1-\Gamma_3)-\sigma
    =2(\ell_0-\ell-1)(\Gamma_3-\sigma)-(\Gamma_1-\Gamma_3)-\sigma\\
    >\frac{\Gamma_1-\Gamma_3+\sigma}{\Gamma_3-\sigma}(\Gamma_3-\sigma)
    -(\Gamma_1-\Gamma_3+\sigma)
    =0,
  \end{align*}
  using the second condition for $\ell_0$ in \eqref{eq:EulerCPDefell0Pre}.
Estimates \eqref{eq:EulerCPEst11}, \eqref{eq:EulerCPEst22} together with \eqref{eq:691823iwej2NNN} therefore imply that
$\{W_C^a(t_n)\}$ is a Cauchy sequence in $H^{k+\ell}(\Sigma)$ for $a=1,2$.  Consequently, the sequence $\{W_C^a(t_n)\}$ has a limit in $H^{k+\ell}(\Sigma)$, which we denote by $W_C^a(0)$ for $a=1,2$. 
We recall that the existence of the limit for $W^0_C$ had already been established in Proposition \ref{prop:EulerSolveCPCor}. Regarding component $W^3_C$, it follows that $\{W_C^3(t_n)\}$ is a Cauchy sequence in $H^{k+\ell_0-2}(\Sigma)$ from the $j=3$, $m=0$-case of \eqref{eq:691823iwej2NNN} combined with \eqref{eq:EulerCPEst33} while observing the same conditions for $\sigma$ as before. As a consequence, $\{W_C^3(t_n)\}$ has a limit $W_C^3(0)\in  H^{k+\ell_0-2}(\Sigma)$.

Estimate \eqref{eq:EulerPertSecLimit1Pre} is the same as \eqref{eq:EulerPertFirstLimitU}.
Regarding estimate \eqref{eq:EulerPertSecLimit2Pre},
we consider \eqref{eq:691823iwej2NNN} with $j=1,2$, $m=\ell_0-\ell-2$ and $t=t_n$, and let $t_{\tilde{n}}\nearrow 0$. Given that
\[-\eta_a=\min\{2 (\ell_0-\ell-2) (\Gamma_3-\sigma)-(\Gamma_a-\Gamma_3),-\sigma\}\]
from \eqref{eq:etadef}, we then see, with the help of \eqref{eq:EulerCPEst11} and \eqref{eq:EulerCPEst22}, that
\begin{equation*}
  \bnorm{W_C^a(t)-W_C^a(0)}_{H^{k+\ell}(\Sigma)}\le |t|^{\min\{2 (\ell_0-\ell-1) (\Gamma_3-\sigma)-(\Gamma_a-\Gamma_3)-\sigma, 2(\Gamma_3-\sigma)-2\sigma,1-\Gamma_1\}}, \quad a=1,2,
\end{equation*}
which establishes \eqref{eq:EulerPertSecLimit2Pre}. Finally, we use again \eqref{eq:691823iwej2NNN}, this time with $j=3$, $m=0$ and $t=t_n$. Letting $t_{\tilde{n}}\nearrow 0$, we find, with the help of \eqref{eq:EulerCPEst33}, that
\begin{equation*}
  \bnorm{W_C^3(t)-W_C^3(0)}_{H^{k+\ell}(\Sigma)}\le |t|^{\min\{ 2(\Gamma_3-\sigma)-\sigma,1-\Gamma_1\}},
\end{equation*}
from which we deduce  \eqref{eq:EulerPertSecLimit3Pre}.

Finally, we establish \eqref{eq:stabilitycloseness} by considering, for fixed $V_S$, two solutions $W_C$ and $\Wt_{C}$ obtained under the same conditions as before for two Cauchy data perturbations $V_0$ and $\Vt_0$ imposed at the same time $T_0\in [\Tt_0,0)$ satisfying $\norm{V_0}_{H^{k+\ell_0}(\Sigma)}\le \delta$ and $\norm{\Vt_0}_{H^{k+\ell_0}(\Sigma)}\le \delta$.
According to \eqref{eq:EulerResc} and \eqref{WCdef} and the definitions of $V_0$ and $\Vt_0$, this means that
\begin{equation}
  \label{eq:WIDPert}
  W_C(T_0)=\Th(T_0)T^{-1}(T_0)V_S(T_0)+V_0,\quad \Wt_{C}(T_0)=\Th(T_0)T^{-1}(T_0)V_S(T_0)+\Vt_0. 
\end{equation}
The difference $W_C(t)-\Wt_{C}(t)$ can be written with the help of \eqref{eq:691823iwej2} as
\begin{equation}
  \label{eq:suboptimal1}
  \begin{split}
  W_C^j(t)&-\Wt_{C}^j(t)=W_C^j(T_0)-\Wt_{C}^j(T_0)\\
  &+\int_{T_0}^t\Bigl(- \sum_{i=1}^3\sum_{k=0}^3\Bh_{C,k}^{j,i}\partial_i W^k_C(s)
  +G_C^j(s)\Bigr)ds
  -\int_{T_0}^t\Bigl(- \sum_{i=1}^3\sum_{k=0}^3{\tilde\Bh}_{C,k}^{j,i}\partial_i \Wt^k_{C}(s)
+\Gt_{C}^j(s)\Bigr)ds.
\end{split}
\end{equation}
And therefore, exploiting the $m=\ell_0-\ell-2$-case of \eqref{eq:EulerCPEst11} -- \eqref{eq:EulerCPEst33}, \eqref{eq:etadef} and \eqref{eq:etaest}
\begin{equation}
  \label{eq:suboptimal2}
  \norm{W_C^j(t)-\Wt_{C}^j(t)}_{H^{k+\ell}(\Sigma)}\le \norm{W_C^j(T_0)-\Wt_{C}^j(T_0) }_{H^{k+\ell}(\Sigma)}
  +C |T_0|^{\min\{2(\Gamma_3-\sigma)-\eta_1-\sigma,1-\Gamma_1\}},
\end{equation}
for all $t\in [\Tt_0,0)$ and for each $j=1,2,3$,
where the constant $C>0$ is in particular independent of $t$, $V_0$ and $\Vt_0$. Due to the convergence result above for this choice of $m$, this estimate also holds in the limit $t\nearrow 0$. We now establish \eqref{eq:stabilitycloseness} by re-expressing the power of $|T_0|$ as above and using \eqref{eq:WIDPert}.  
\end{proof}

\bigskip

\noindent \underline{Step 3: Establish that $V_C$ agrees with a solution $\Vt_S$ of the singular initial value problem}

\bigskip

To complete the proof of \Theoremref{thm:fluidresult2}, we need to verify the solution $V_C$ from  \Propref{prop:EulerSolveCPCorStrong} that determines via \eqref{VC2UC} and \eqref{WCdef} the limits $W_C^0(0)$, $W_C^1(0)$, $W_C^2(0)$, $W_C^3(0)$ agrees with the solution $\tilde V_S$ from \Theoremref{thm:fluidresult} that is generated by the asymptotic data $\tilde v_*=(W^0_C(0), W^1_C(0), W^2_C(0), W^3_C(0))^{\tr}$. Note that \Propref{prop:EulerSolveCPCorStrong} guarantees that $\tilde v_*\in H^{k+\ell}(\Sigma)$ and $W^0_C(0)>0$, which is sufficient to apply \Theoremref{thm:fluidresult} and obtain the solution $\tilde{V}_S$ satisfying \eqref{eq:EulerregNN}.

Since $V_C$ defines a classical solution of the Euler equations \eqref{eq:AAA1} on $[T_0,0)\times \Sigma$, we would be able to conclude from the uniqueness statement from \Theoremref{thm:fluidresult} that $V_C=\Vt_S$, which is what we need to establish, if we can show that
\begin{equation} \label{VC=VtS}
\sup_{t\in [T_0,0)}\norm{|t|^{-\mu}(U_C-\Ut_S)}_{H^k(\Sigma)}\lesssim 1
\end{equation}
for some $\mu>\Gamma_1$ where $U_C=T^{-1} V_C$ and $\Ut_S=T^{-1}\Vt_S$. But, by 
\Theoremref{thm:fluidresult}, we have
\[\Ut_S=\Ut_{\ell-1}+\ut\]
where $\ut$ is given by
\Propref{prop:EulerProp2}
and $\Ut_{\ell-1}$ is constructed in \Lemref{lem:iterateLOT} from the asymptotic data $\tilde v_*$; see in particular, \eqref{remainder} and \eqref{U*set}. Therefore, due to \eqref{eq:solution1estimateboundednessNNNNN} and the fact that $\lambda+1 >\Gamma_1$ in \Propref{prop:EulerProp2}, it is sufficient to show that 
\[\sup_{t\in [T_0,0)}\norm{|t|^{-\mu}(\Ut_{\ell-1} -U_C)}_{H^k(\Sigma)}\lesssim 1\]
for some $\mu>\Gamma_1$ in order to guarantee that \eqref{VC=VtS} holds. 
Using the same notation as previously, see \eqref{eq:EulerDefuhn}, we set
\[\Ut_{\ell-1}=\Th^{-1}\Wt_{\ell-1}=\Th^{-1}(\tilde v_*+\wt_{\ell-1})\] 
and define analogously
\[U_{C}=\Th^{-1}W_{C}=\Th^{-1}(\tilde v_*+w_{C}).\] Then because of \eqref{eq:Gammarestr} and \eqref{Thdef}, we conclude that 
\eqref{VC=VtS} will hold provided that
\begin{equation}
  \label{eq:EulerCPUniqueness}
  \sup_{t\in [T_0,0)}\norm{|t|^{-\mu}(\wt_{\ell-1} -w_C)}_{H^k(\Sigma)}\lesssim 1
\end{equation}
for some $\mu>\Gamma_1$.

Let us now consider the finite sequence of leading-order term fields $(\wt_m)$ for $m=0,\ldots,\ell-1$ constructed as in \Lemref{lem:iterateLOT}. This means that for each $m\in [0,\ell-1]$, the function $\wt_m$ is given by \eqref{w-evol-A} with $n+1=m$ and with $v_*$ replaced by $\vt_*$.
Since $W_C=\vt_*+w_C$, it is clear from \eqref{eq:3333339392} that $w_C$ will satisfy \eqref{w-evol-A} 
if we set $w_{n+1}=w_n=w_C$, $\Bh^i_{n+1}=\Bh^i_{n}=\Bh^i_{C}$, $G_{n+1}=G_{n}=G_{C}$ and $v_*=\vt_*$.
In full analogy to \eqref{w-evol-B} and \eqref{eq:EulerSIVPImpr}, this implies that
\begin{equation}
    \label{w-evol-BC}
    \del{t}(\wt_{m}-w_C)= \fb_{m-1} :=  -\Bh^i_{m-1}\del{i}(\wt_{m-1}-w_{C})-(\Bh^i_{m-1}-\Bh^i_{C})\del{i}(\vt_*+w_{C})+\Gh_{m-1}-\Gh_{C}
  \end{equation}
  and
\begin{equation}
  \label{eq:EulerSIVPImprNNN}
    \wt_{m}(t)-w_{C}(t) = -\int_{t}^0 \fb_{m-1}(s)\, ds.
  \end{equation}

  We now employ induction to verify that the inequality
  \begin{equation}
    \label{eq:asaskld2934ldfmv}
    \norm{\wt_{m}-w_{C}}_{H^{k+\ell-m}(\Sigma)}\lesssim |t|^{(m+1)q-4\sigma}
  \end{equation}
  holds for each integer $0\le m\le\ell-1$, where $q$ is defined by \eqref{eq:Eulerchooseq}.
  For the base case $m=0$, we observe that
  $\wt_0=0$, see \eqref{eq:Eulerw0}, and the estimates \eqref{eq:EulerPertSecLimit1} and \eqref{eq:EulerPertSecLimit2}, which we know hold by Step 2, can be combined to give
  \[\norm{\wt_{0}-w_{C}}_{H^{k+\ell}(\Sigma)}\lesssim |t|^{q-4\sigma}.\]
  For the induction step $m\mapsto m+1$, suppose that \eqref{eq:asaskld2934ldfmv} holds for some $0\le m\le\ell-2$.
  Using analogous arguments to the ones used to derive \eqref{fh-est}, we find that
\begin{align*}
  \norm{\fb_{m}}_{H^{k+\ell-{(m+1)}}(\Sigma)}\lesssim \bigl(|t|^{-\Gamma_1}+|t|^{2\Gamma_3-1}\bigr) \norm{\wt_{m}-w_{C}}_{H^{k+\ell-m}(\Sigma)}
\end{align*}
for all $t\in [T_0,0)$, and hence, due to \eqref{eq:asaskld2934ldfmv}, that
\begin{align*}
  \norm{\fb_{m}}_{H^{k+\ell-{(m+1)}}(\Sigma)}\lesssim \bigl(|t|^{1-\Gamma_1}+|t|^{2\Gamma_3}\bigr) |t|^{(m+1)q-4\sigma-1}
  =|t|^{(m+2)q-4\sigma-1},
\end{align*}
where in deriving this expression, we have again made use of the definition of $q$ given by \eqref{eq:Eulerchooseq}.
\Eqref{eq:EulerSIVPImprNNN}, with $m$ replaced by $m+1$, then gives
\begin{equation*}
    \norm{\wt_{m+1}-w_{C}}_{H^{k+\ell-(m+1)}(\Sigma)}\lesssim |t|^{(m+2)q-4\sigma},
  \end{equation*}
  and we conclude that \eqref{eq:asaskld2934ldfmv} holds for each integer $0\le m\le\ell-1$. Setting $m=\ell-2$, we deduce that
  \begin{equation*}
    \norm{\wt_{\ell-1}-w_{C}}_{H^{k+1}(\Sigma)}\lesssim |t|^{\ell q-4\sigma}.
  \end{equation*}
But, since $\ell q>\Gamma_1$ and we are allowed to choose $\sigma$ as small as we like, we see that \eqref{eq:EulerCPUniqueness} holds, and so, we conclude that $V_C=\Vt_S$, which completes the proof of \Theoremref{thm:fluidresult2}.
\hfill$\Box$

\appendix 

\section{Cauchy problem for Fuchsian equations}
\label{sec:FuchsianCauchyProblem}

In \cite{beyer2019a}, it was established that the initial value problem 
\begin{align*} 
    B^0(t,u)\del{t}u + B^i(t,u)\nabla_i u &= \frac{1}{t}\Bc(t,u)u +F(t,u) \quad \text{in $[T_0,0)\times \Sigma$,} 
    \\
    u &= u_0 \quad \hspace{2.6cm}\text{in $\{0\}\times \Sigma$,} 
\end{align*}
for Fuchsian equations admits solutions under a small initial data assumption provided that the coefficients satisfy certain structural conditions. Strictly speaking, the existence results from \cite{beyer2019a} do not apply to the initial value problem for Fuchsian equations of the form \eqref{symivp.1}, that is,
\begin{align} 
    B^0(t,w_1,u)\del{t}u + B^i(t,w_1,u)\nabla_i u &= \frac{1}{t}\Bc(t,w_1,u)u +F(t,w_2,u) \quad \text{in $[T_0,0)\times \Sigma$,} \label{FuchsianIVPB.1} \\
    u &= u_0 \quad \hspace{3.7cm}\text{in $\{T_0\}\times \Sigma$,} \label{FuchsianIVPB.2} 
\end{align}
which are considered in this article
due to the presence of the ``background'' fields $w_1$ and $w_2$. However, it is not difficult to verify that all of the results of \cite{beyer2019a} still apply. Since the modifications needed to adapt the results from \cite{beyer2019a} to the current setting are so straightforward, we will only discuss the essential changes below.  Instead, we will content ourselves with stating the corresponding existence result for the initial value problem \eqref{FuchsianIVPB.1} --\eqref{FuchsianIVPB.2}.
However, before we state the existence theorem, we first list the coefficient assumptions that the Fuchsian equation \eqref{FuchsianIVPB.1} needs to satisfy.

\subsection{Coefficient assumptions\label{coeffassumps}}
As in Definition \ref{def:symmhypFuchssystems}, we assume here that $\Zc_1$ and $\Zc_2$ are open and bounded subsets of the vector bundles $Z_1$ and $Z_2$, respectively, where $\pi(\Zc_1)=\pi(\Zc_2)=\Sigma$, and $T_0<0$, $R>0$, $\Rc>0$ and $p\in (0,1]$ are fixed constants. We then assume the coefficients of the Fuchsian equation \eqref{FuchsianIVPB.1} satisfy the following set of assumptions, where here, we employ the same notation as in Section \ref{prelim}:

\begin{enumerate}
\item The section $\Pbb \in \Gamma(L(V))$ satisfies
\begin{equation*} 
\Pbb^2 = \Pbb,  \quad  \Pbb^{\tr} = \Pbb, \quad \del{t}\Pbb =0 \AND \nabla \Pbb =0.
\end{equation*}
\item There exist constants  $\kappa, \gamma_1, \gamma_2 >0$ such that the maps 
\begin{equation*}
    B^0 \in 
C^0_b\bigl([T_0,0), C^\infty\bigl(\Zc_1\oplus B_R(V),L(V)\bigr)\bigr)\cap C^1\bigl([T_0,0), C^\infty\bigl(\Zc_1\oplus B_R(V),L(V)\bigr)\bigr)
\end{equation*}
and $\Bc\in C^0_b\bigl([T_0,0), C^\infty\bigl(\Zc_1\oplus B_R(V),L(V)\bigr)\bigr)$  satisfy
\begin{equation*}
\pi(B^0(t,z_1,v))=\pi(\Bc(t,z_1,v))=\pi((z_1,v)), 
\end{equation*}
and
\begin{equation}\label{B0BCbndtrafoBW}
\frac{1}{\gamma_1} \text{id}_{V_{\pi(v)}} \leq  B^0(t,z_1,v)\leq \frac{1}{\kappa} \Bc(t,z_1,v) \leq \gamma_2 \textrm{id}_{V_{\pi(v)}}
\end{equation}
for  all $(t,z_1,v)\in [T_0,0)\times \Zc_1\oplus B_{R}(V)$. Moreover,
\begin{align*} 
[\Pbb(\pi(v)),\Bc(t,z_1,v)] = 0 \AND 
(B^0(t,z_1,v))^{\tr} = B^0(t,z_1,v)  
\end{align*}
for all $(t,z_1,v) \in [T_0,0) \times \Zc_1 \oplus B_{R}(V)$,
\begin{align*}
\Pbb(\pi(v)) B^0(t,z_1,v)\Pbb^\perp(\pi(v)) &= \Ord\bigl(|t|^{\frac{p}{2}}+\Pbb(\pi(v)) v\bigr) 
\intertext{and}
\Pbb^\perp(\pi(v)) B^0(t,z_1,v) \Pbb(\pi(v)) &= \Ord\bigl(|t|^{\frac{p}{2}}+\Pbb(\pi(v)) v\bigr),
\end{align*}
where $\Pbb^\perp=\id-\Pbb$, and there exist maps $\Bt^0, \tilde{\Bc} \in C^0_b\bigl([T_0,0), C^\infty\bigl(\Zc_1,L(V)\bigr))\bigr)$
satisfying
\begin{align*}
    \pi(\Bt^0(t,z_1))=\pi(\tilde{\Bc}(t,z_1))&=\pi(z_1) \AND
[\Pbb(\pi(v)),\tilde{\Bc}(t,z_1)] =0
\end{align*}
for all $(t,z_1,v)\in  [T_0,0)\times \Zc_1\oplus B_R(V)$ such that
\begin{align*}
B^0(t,z_1,v)-\Bt^0(t,z_1) &= \Ord(v) \AND 
\Bc(t,z_1,v)-\tilde{\Bc}(t,z_1)=\Ord(v). 
\end{align*}

\item The map $F\in C^0\bigl([T_0,0), C^\infty\bigl(\Zc_2\oplus B_R(V),V\bigr)\bigr)$ can be expanded as
\begin{equation*}
F(t,z_2,v) = |t|^{-(1-p)} \Ft(t,z_2) + |t|^{-(1-p)}F_0(t,z_2,v) + |t|^{-(1-\frac{p}{2})}F_1(t,z_2,v) + |t|^{-1}F_2(t,z_2,v)
\end{equation*}
where $\Ft \in C^0\bigl([T_0,0), C^\infty(\Zc_2,V)\bigr)$, the maps
 $F_0,F_1,F_2 \in C^0_b\bigl([T_0,0), C^\infty(\Zc_2\oplus B_R(V),V)\bigr)$
 satisfy
 \begin{align*}
 \pi(\Ft(t,z_2))&=\pi(z_2), \\
 \pi (F_a (t,z_2,v))&=\pi((z_2,v)), \quad a=0,1,2,
 \end{align*}
 and
\begin{equation*} 
\Pbb(\pi(v)) F_2(t,z_2,v) = 0 
\end{equation*}
for all $(t,z_2,v)\in [T_0,0)\times\Zc_2\oplus B_R(V)$, and
there exist constants $\lambda_a\geq 0$, $a=1,2,3$, such that
\begin{align*}
F_0(t,z_2,v) &=\Ord(v),\\
\Pbb(\pi(v)) F_1(t,z_2,v) &= \Ordc(\lambda_1 v),\\
\Pbb^\perp(\pi(v)) F_1(t,z_2,v) &= \Ordc(\lambda_2\Pbb(\pi(v)) v) 
\intertext{and}
\Pbb^\perp(\pi(v)) F_2(t,z_2,v) & = \Ordc\biggl(\frac{\lambda_3}{R}\Pbb(\pi(v)) v\otimes\Pbb(\pi(v))v \biggr).
\end{align*}

\item The  map $B\in C^0\bigl([T_0,0), C^\infty\bigl(\Zc_1\oplus B_R(V),L(V)\otimes T\Sigma\bigr)\bigr)$ satisfies
\begin{equation*}
\pi(B(t,z_1,v))=\pi((z_1,v))
\end{equation*}
and
\begin{equation*}
\bigl[\sigma(\pi(v))(B(t,z_1,v))\bigr]^{\tr}=\sigma(\pi(v))(B(t,z_1,v))
\end{equation*}
for all $(t,z_1,v)\in [T_0,0)\times \Zc_1\oplus B_R(V)$ and $\sigma \in \mathfrak{X}^*(\Sigma)$. 
Moreover, $B$ can be expanded as
\begin{equation*} 
B(t,z_1,v) = |t|^{-(1-p)}B_0(t,z_1,v) + |t|^{-(1-\frac{p}{2})}B_1(t,z_1,v) + |t|^{-1}B_2(t,z_1,v)
\end{equation*}
where  $B_0,B_1,B_2 \in C^0_b\bigl([T_0,0), C^\infty(\Zc_1\oplus B_R(V),L(V)\otimes T^*\Sigma)\bigr)$ satisfy
\begin{equation*}
\pi(B_a(t,z_1,v))=\pi((z_1,v))
\end{equation*}
for all  $(t,z_1,v)\in [T_0,0)\times\Zc_1\oplus B_R(V)$, and
there exist a constant $\alpha\geq 0$ and a map $\Bt_2\in  C^0\bigl([T_0,0), C^\infty\bigl(\Zc_1,L(V)\otimes T^*\Sigma\bigr)\bigr)$  such that
\begin{align*} 
\Pbb(\pi(v)) B_1(t,z_1,v) \Pbb(\pi(v)) &=   \Ord(1), \\
\Pbb(\pi(v)) B_1(t,z_1,v) \Pbb^\perp(\pi(v))&=  \Ordc(\alpha),\\
 \Pbb^\perp(\pi(v)) B_1(t,z_1,v) \Pbb(\pi(v)) &= \Ordc(\alpha), \\
\Pbb^\perp(\pi(v)) B_1(t,z_1,v) \Pbb^\perp(\pi(v)) &= \Ord(\Pbb(\pi(v)) v),\\
\Pbb(\pi(v)) B_2(t,z_1,v) \Pbb^\perp(\pi(v)) &= \Ord(\Pbb(\pi(v)) v),\\
\Pbb^\perp(\pi(v)) B_2(t,z_1,v)\Pbb(\pi(v)) &=  \Ord(\Pbb(\pi(v)) v),\\
\Pbb^\perp(\pi(v)) B_2(t,z_1,v) \Pbb^\perp(\pi(v))  &= \Ord\bigl(\Pbb(\pi(v)) v\otimes \Pbb(\pi(v)) v \bigr)\\
\intertext{and}
\Pbb(\pi(v))\bigl(B_2(t,z_1,v)-\Bt_2(t,z_1)\bigr)\Pbb(\pi(v)) &= \Ord(v).
\end{align*}

\item  There exist constants $q,\theta\geq 0$, and $\beta_a \geq 0$, $a=0,1,\ldots,7$, such that the map $\Div\!B$ defined locally by \eqref{divBdef} satisfies
\begin{gather*}
\Pbb(\pi(v)) \Div \! B(\xi)\Pbb(\pi(v)) = 
\Ordc\bigl(|t|^{-(1-p)}\theta+  |t|^{-(1-\frac{p}{2})}\beta_0 + |t|^{-1}\beta_1\bigl), \\
\Pbb(\pi(v))  \Div\! B(\xi) \Pbb^\perp(\pi(v)) = 
\Ordc\biggl(|t|^{-(1-p)}\theta+|t|^{-(1-\frac{p}{2})}\beta_2
+ \frac{|t|^{-1}\beta_3}{R}\Pbb(\pi(v)) v\biggr), \\
\Pbb^\perp(\pi(v)) \Div\! B(\xi) \Pbb(\pi(v))= 
\Ordc\biggl(|t|^{-(1-p)}\theta+|t|^{-(1-\frac{p}{2})}\beta_4
+ \frac{|t|^{-1}\beta_5}{R}\Pbb(\pi(v)) v\biggr) 
\intertext{and}
\Pbb^\perp(\pi(v)) \Div\! B(\xi) \Pbb^\perp(\pi(v)) =\Ordc\biggl(|t|^{-(1-p)}\theta+
\frac{|t|^{-(1-\frac{p}{2})}\beta_6}{R}\Pbb(\pi(v)) v\\
\qquad\qquad\qquad\qquad\qquad\qquad\qquad+ \frac{|t|^{-1}\beta_7}{R^2}\Pbb(\pi(v)) v\otimes\Pbb(\pi(v)) v \biggr) 
\end{gather*}
where $\xi=(t,z_1,\dot{z}_1,z'_1,z_2,v,v')$ denotes an element of $[0,T_0)\times\Zc_1\oplus B_{\Rc}(Z_1)\oplus B_{\Rc}(Z_1\times T^*\Sigma)\oplus \Zc_2\oplus B_R(V)\oplus B_R(V\otimes T^*\Sigma)$.
\end{enumerate}

\begin{rem} \label{kappatrem}
By \eqref{B0BCbndtrafoBW}, we note that there exist constants $0<\gammat_1\leq \gamma_1$ and $\kappat\geq \kappa >0$
such that
\begin{equation*}
\frac{1}{\gammat_1} \Pbb(\pi(v)) \leq  \Pbb(\pi(v))B^0(t,z_1,v)\Pbb(\pi(v))\leq \frac{1}{\kappat} \Bc(t,z_1,v)\Pbb(\pi(v))  \leq \gamma_2 \Pbb(\pi(v))
\end{equation*}
for  all $(t,z_1,v)\in [T_0,0)\times\Zc_1\oplus B_R(V)$.
\end{rem}

\subsection{Existence and uniqueness}
\label{appendix:existence}

\begin{thm} \label{thm:Fuchsian-IVP}
Suppose $k \in \Zbb_{>n/2+3}$, $\sigma>0$, $u_0\in H^k(\Sigma,V)$, assumptions (1)-(5) from Appendix \ref{coeffassumps} are fulfilled,
$w_1\in C^0_b\bigl([T_0,0),H^k(\Sigma,Z_1)\bigr)\cap C^1\bigl([T_0,0),H^{k-1}(\Sigma,Z_1)\bigr)$ and $w_2\in C^0_b\bigl([T_0,0),H^k(\Sigma,Z_2)\bigr)$ satisfy
\begin{equation*}
   \sup_{T_0<t<0} \max\Bigl\{ \norm{\nabla w_1(t)}_{H^{k-1}(\Sigma)},
    \norm{|t|^{1-q}\del{t}w_1(t)}_{H^{k-1}(\Sigma)}\Bigr\} < \frac{\Rsc}{C_{\text{Sob}}}
\end{equation*}
and
\begin{equation*}
w_a([T_0,0)\times\Sigma)\subset \Zc_a, \quad a=1,2,
\end{equation*}
and the constants $\kappa$, $\gamma_1$, $\lambda_3$, $\beta_0$, $\beta_1$, $\beta_3$, $\beta_5$, $\beta_7$ from Appendix \ref{coeffassumps} satisfy
\begin{equation*}
\label{thm:Fuchsian-IVP-kappa}
\kappa > \frac{1}{2}\gamma_1\max\biggl\{\sum_{a=0}^3  \beta_{2a+1}+2\lambda_3,\beta_1+ 2k(k+1)\mathtt{b}\biggr\}
\end{equation*}
where
\begin{align}
\mathtt{b} =  \sup_{T_0 \leq t < 0}\Bigl( \big\|\big|\Pbb\tilde{\Bc}(t)\nabla(\tilde{\Bc}(t)^{-1}\tilde{B}^0(t))\tilde{B}^0(t)^{-1}\Pbb \tilde{B}_2(t)\Pbb\bigl|_{\op}\bigr\|_{L^\infty(\Sigma)}
+ \big\|\big|\Pbb\tilde{\Bc}(t)\nabla(\tilde{\Bc}(t)^{-1}\tilde{B}_2(t))\Pbb\bigl|_{\op}\bigr\|_{L^\infty(\Sigma)} \Bigr) \label{thm:Fuchsian-IVP-btt}
\end{align}
with $\Bt^0(t)=\Bt^0(t,w_1(t))$ and $\Bt_1(t)=\Bt_1(t,w_1(t))$.

\medskip

\noindent Then there exists
a $\delta > 0$ such that if 
\begin{equation*}
 \max\biggr\{\norm{u_0}_{H^k(\Sigma)},\int_{T_0}^0 |s|^{p-1}\norm{\Ft(\tau)}_{H^k(\Sigma)}\, ds\biggr\}
 \leq \delta,
\end{equation*}
where  $\Ft(t)=\Ft(t,w_2(t))$, 
then there exists a unique solution 
\begin{equation*}
u \in C^0_b\bigl([T_0,0),H^k(\Sigma,V)\bigr)\cap C^1\bigl([T_0,0),H^{k-1}(\Sigma,V)\bigr)
\end{equation*}
of the IVP \eqref{FuchsianIVPB.1}-\eqref{FuchsianIVPB.2} such that the limit $\lim_{t\nearrow 0} \Pbb^\perp u(t)$, denoted $\Pbb^\perp u(0)$, exists in $H^{k-1}(\Sigma,V)$.

\medskip

\noindent Moreover, for $T_0 \leq t < 0$,  the solution $u$ satisfies  the energy estimate
\begin{equation}
  \label{eq:FuchsianEnergyEstimate}
\biggl(\norm{u(t)}_{H^k(\Sigma)}^2 - \int_{T_0}^t \frac{1}{\tau} \norm{\Pbb u(\tau)}_{H^k(\Sigma)}^2\, d\tau \biggr)^{\frac{1}{2}}  \lesssim \biggl( \norm{u_0}_{H^k(\Sigma)} + \int_{T_0}^t |s|^{p-1}\norm{\Ft(s)}_{H^k(\Sigma)}\, ds \biggr)
\end{equation}
and the decay estimates
\begin{align}
  \label{eq:decayPQunt}
\norm{\Pbb u(t)}_{H^{k-1}(\Sigma)} &\lesssim \begin{cases}
|t|^p+(\lambda_1+\alpha)|t|^{\frac{p}{2}} & \text{if $\zeta > p$} \\
|t|^{\zeta-\sigma}+(\lambda_1+\alpha)|t|^{\frac{p}{2}} & \text{if $\frac{p}{2} < \zeta \leq p$} \\
|t|^{\zeta-\sigma}   & \text{if $0 < \zeta \leq \frac{p}{2}$}\end{cases}
\intertext{and}
\norm{\Pbb^\perp u(t) - \Pbb^\perp u(0)}_{H^{k-1}(\Sigma)} &\lesssim
 \begin{cases}  |t|^{p/2}+ |t|^{\zeta-\sigma}
& \text{if $\zeta > p/2$} \\
 |t|^{\zeta-\sigma}  &  \text{if $\zeta \leq p/2 $ }
 \end{cases},
\end{align}
where
\begin{equation*}
\zeta = \kappat -\frac{1}{2}\gammat_1\bigl(\beta_1+ (k-1)k\tilde{\mathtt{b}}\bigr)
\end{equation*}
and
\begin{equation*}
\tilde{\mathtt{b}} =  \sup_{T_0 \leq t < 0}\Bigl( \big\|\big|\Pbb\tilde{\Bc}(t)\nabla(\tilde{\Bc}(t)^{-1}\Pbb \tilde{B}^0(t) \Pbb)\Pbb \tilde{B}^0(t)^{-1} \tilde{B}_2(t)\Pbb\bigl|_{\op}\bigr\|_{L^\infty(\Sigma)}
+ \big\|\big|\Pbb\tilde{\Bc}(t)\nabla(\tilde{\Bc}(t)^{-1}\tilde{B}_2(t))\Pbb\bigl|_{\op}\bigr\|_{L^\infty(\Sigma)} \Bigr).
\end{equation*}
\end{thm}
\begin{proof}
The proof follows from a straightforward adaptation of the proof of Theorem 3.8.~from \cite{beyer2019a} together with the time-reparameterisation argument from Section 3.4.~of that same article; see also Lemma \ref{lem:trafoproof} which shows how the Fuchsian system \eqref{FuchsianIVPB.1} transform under the time-reparameterization $t=-(-t)^p$. The only important change to note is the following simple improvement to the proof of \cite[Theorem 3.8.]{beyer2019a}:

\bigskip

\noindent Adapting the arguments in the proof of \cite[Theorem 3.8.]{beyer2019a}, it is not difficult to verify that the following variation of the differential energy estimate given on the line equation below equation (3.71) in \cite{beyer2019a} holds in our setting:
\begin{equation*}
\del{t} E_k \leq C(\delta,\delta^{-1})\bigl(E_k+\norm{\Ft}_{H^k}\sqrt{E_k}\bigr), \qquad T_0\leq t < T_\delta,
\end{equation*}
where $\Ft=\Ft(t,w_2(t))$ and the energy $E_k$
is now defined by
\begin{equation*} 
E_k(t) = \nnorm{u(t)}^2_k 
+\rho^{-1}_0c(\delta,\delta^{-1})\nnorm{u(t)}_0^2 
-\int_{T_0}^t \frac{\rho_k}{\tau}\nnorm{\Pbb u(\tau)}^2_k\, d\tau.
\end{equation*}
The main advantage of this variation is that it allows us to write the above differential energy inequality as
\begin{equation*}
\del{t} \sqrt{E_k} \leq C(\delta,\delta^{-1})\bigl(\sqrt{E_k}+\norm{\Ft}_{H^k(\Sigma)}\bigr),
\end{equation*}
which, in turn, yields via Gr\"{o}nwall's inequality the energy estimate
\begin{equation} \label{new-energy}
    \sqrt{E_k(t)} \leq e^{C(\delta,\delta^{-1})T_0}\Bigl(\sqrt{E_k(T_0)}+ \int_{T_0}^t \norm{\Ft(s)}_{H^k(\Sigma)}\,ds\Bigr).
\end{equation}
Comparing this estimate with the estimate (3.73) from \cite{beyer2019a}, we see that the advantage of
the new energy estimate \eqref{new-energy} is that it only requires $\int_{T_0}^0 \norm{\Ft(s,w_2(s))}_{H^k(\Sigma)}\,ds$ to be small rather than $\sup_{T_0 < t<0}\norm{\Ft(s,w_2(s))}_{H^k(\Sigma)}$ as in the proof of \cite[Theorem 3.8.]{beyer2019a}. With the exception of this change, the rest of the proof mimics that of \cite[Theorem 3.8.]{beyer2019a} with the obvious changes needed to account for the background fields $w_1$ and $w_2$.
\end{proof}

\begin{rem} \label{rem:Fuchsian-IVP}$\;$

\begin{enumerate}[(i)]
    \item By \cite[Remark 3.10.(a).(ii)]{beyer2019a}, we know, taking into account the time-reparameterization argument from \cite[\S 3.4.]{beyer2019a}, that  if the
    maps $B_1$ and $B_2$ satisfy \begin{equation*} 
    \Pbb^\perp B_1=\Pbb^\perp B_2=0,     
    \end{equation*}
    then the decay estimate for $\Pbb^\perp u(t)-\Pbb^\perp u(0)$ from Theorem \ref{thm:Fuchsian-IVP} improves to 
    \begin{equation*}
\norm{\Pbb^\perp u(t)-\Pbb^\perp u(0)}_{H^{k-1}} \lesssim 
\begin{cases}  |t|^p
& \text{if $\zeta > p$ } \\
  |t|^p +  |t|^{2(\zeta-\sigma)}  &  \text{if $\zeta \leq p$ }
 \end{cases} , 
\quad T_0\leq t < 0.
\end{equation*}
\item If the map $B_2$ satisfies
\begin{equation*}
    B_2=0,
\end{equation*}
then the regularity requirement in the statement of Theorem \ref{thm:Fuchsian-IVP} can be lowered to $k \in \Zbb_{>n/2+1}$. This
is because in this situation the use of \cite[Lemma 3.5.]{beyer2019a} can be avoided in the proof of \cite[Theorem 3.8.]{beyer2019a}, which leads to the reduction in the required regularity.
\end{enumerate}
\end{rem}

\bibliographystyle{abbrv}
\bibliography{bibliography}

\end{document}